\documentclass[11pt, reqno]{amsart}
\usepackage[pdftex]{hyperref}
\usepackage[numbers]{natbib}
\usepackage{graphicx}
\usepackage{tikz-cd}
\usepackage[ansinew]{inputenc} 

\usepackage[scr=rsfso,cal=zapfc,frak=euler,bb=ams]{mathalfa}

\usepackage{amssymb}

\usepackage{bm} 
\usepackage{bbm}

\usepackage{mathtools}
\usepackage{amsthm}
\usepackage{cases}
\usepackage{booktabs}
\usepackage{enumitem}
\usepackage{fullpage}
\usepackage{eurosym}
\usepackage{stmaryrd}
\usepackage{pifont} 
\usepackage{manfnt} 

\numberwithin{equation}{section}

\makeatletter
\def\@tocline#1#2#3#4#5#6#7{\relax
  \ifnum #1>\c@tocdepth 
  \else
    \par \addpenalty\@secpenalty\addvspace{#2}%
    \begingroup \hyphenpenalty\@M
    \@ifempty{#4}{%
      \@tempdima\csname r@tocindent\number#1\endcsname\relax
    }{%
      \@tempdima#4\relax
    }%
    \parindent\z@ \leftskip#3\relax \advance\leftskip\@tempdima\relax
    \rightskip\@pnumwidth plus4em \parfillskip-\@pnumwidth
    #5\leavevmode\hskip-\@tempdima
      \ifcase #1
       \or\or \hskip 1em \or \hskip 2em \else \hskip 3em \fi%
      #6\nobreak\relax
    \dotfill\hbox to\@pnumwidth{\@tocpagenum{#7}}\par
    \nobreak
    \endgroup
  \fi}
\makeatother

\renewcommand{\Re}{\operatorname{Re}}
\renewcommand{\Im}{\operatorname{Im}}

\newcommand{\E}{\mathbb E }
\newcommand{\R}{\mathbb{R}}
\newcommand{\N}{\mathbb{N}}
\newcommand{\C}{\mathbb{C}}
\newcommand{\Z}{\mathbb{Z}}

\renewcommand{\P}{\mathbb{P}}

\newcommand{\bmB}{{\bm{B}}}

\newcommand{\bmrho}{{\bm{\rho}}}
\newcommand{\bmphi}{{\bm{\phi}}}

\DeclareMathOperator{\Beta}{B}

\newcommand{\Stirone}[2]{{\rm Stir}1[#1;\,#2]}
\newcommand{\Stirtwo}[2]{{\rm Stir}2[#1;\,#2]}
\newcommand{\StirSibuya}[2]{{\rm StirSibuya}[#1;\,#2]}
\newcommand{\rStirone}[3]{#3\text{-}{\rm Stir}1[#1;\,#2]}
\newcommand{\rStirtwo}[3]{#3\text{-}{\rm Stir}2[#1;\,#2]}
\newcommand{\rStirSibuya}[3]{#3\text{-}{\rm StirSibuya}[#1;\,#2]}

\newcommand{\Mult}[2]{{\rm Mult}[#1;\,(#2)]}
\newcommand{\MultNP}[2]{{\rm Mult}[#1;\,#2]}

\newcommand{\MultStirone}[3]{{\rm MultiStir}1[#1;\,#2;\,(#3)]}

\newcommand{\Ewens}[2]{{\rm Ewens}[#1;\,#2]}
\newcommand{\RRT}[1]{{\rm RRT}[#1]}
\newcommand{\Hoppe}[2]{{\rm Hoppe}[#1;\,#2]}
\newcommand{\MultiHoppe}[3]{{\rm MultiHoppe}[#1;\,#2;\,(#3)]}
\newcommand{\MultiEwens}[3]{{\rm MultiEwens}[#1;\,#2;\,(#3)]}

\newcommand{\GibbsPart}[3]{#3\text{-}{\rm GibbsPart}[#1;\,#2]}

\newcommand{\rEwens}[3]{#3\text{-}{\rm Ewens}[#1;\,#2]}
\newcommand{\rHoppe}[3]{#3\text{-}{\rm Hoppe}[#1;\,#2]}

\newcommand{\Poi}[1]{{\rm Poi}[#1]}
\newcommand{\Gammadistr}[2]{{\rm Gamma}[#1;\,#2]}
\newcommand{\Betadistr}[2]{{\rm Beta}[#1;\,#2]}
\newcommand{\Bin}[2]{{\rm Bin}[#1;\,#2]}

\newcommand{\Bern}[1]{{\rm Bern}[#1]}
\newcommand{\NBin}[2]{{\rm NBin}[#1;\,#2]}
\newcommand{\Geo}[1]{{\rm Geom}[#1]}
\newcommand{\Exp}[1]{{\rm Exp}[#1]}
\newcommand{\Normal}[2]{{\rm N}\left[#1;\,#2\right]}

\newcommand{\Dir}[1]{{\rm Dir}[#1]}
\newcommand{\MDir}[2]{{\rm DirMult}[#1;\,#2]}

\newcommand{\Lah}[4]{{\rm Lah}[#1,\,#2]_{#3,\,#4}}

\newcommand{\Bell}{\mathfrak{B}}

\DeclareMathOperator{\cycles}{{\sf cycles}}
\DeclareMathOperator{\blocks}{{\sf blocks}}
\DeclareMathOperator{\type}{{\sf type}}
\DeclareMathOperator{\rt}{{\sf root}}
\DeclareMathOperator{\close}{{\sf close}}

\newcommand{\Var}{\mathop{\mathrm{Var}}\nolimits}

\newcommand{\eee}{{\rm e}}

\newcommand{\todistr}{\overset{{\rm d}}{\underset{n\to\infty}\longrightarrow}}
\newcommand{\todistrk}{\overset{{\rm d}}{\underset{n,k\to\infty}\longrightarrow}}
\newcommand{\todistrnk}{\overset{{\rm d}}{\underset{n,k\to\infty}\longrightarrow}}

\newcommand{\ton}{\overset{}{\underset{n\to\infty}\longrightarrow}}

\newcommand{\bsl}{\backslash}
\newcommand{\1}{\mathbf{1}}

\newcommand{\stirling}[2]{\genfrac{[}{]}{0pt}{}{#1}{#2}}
\newcommand{\stirlingsec}[2]{\genfrac{\{}{\}}{0pt}{}{#1}{#2}}

\newcommand{\EulerNum}[2]{\genfrac{\langle}{\rangle}{0pt}{}{#1}{#2}}

\theoremstyle{plain}
\newtheorem{theorem}{Theorem}[section]
\newtheorem{proposition}[theorem]{Proposition}
\newtheorem{prop}[theorem]{Proposition}
\newtheorem{lemma}[theorem]{Lemma}
\newtheorem{corollary}[theorem]{Corollary}
\newtheorem{cor}[theorem]{Corollary}

\theoremstyle{definition}
\newtheorem{definition}[theorem]{Definition}

\newtheorem{example}[theorem]{Example}

\theoremstyle{remark}
\newtheorem{remark}[theorem]{Remark}

\begin{document}

\title[Multinomial random combinatorial structures and $r$-versions of Stirling, Eulerian and Lah numbers]{Multinomial random combinatorial structures and $r$-versions of Stirling, Eulerian and Lah numbers}

\begin{abstract}

We introduce multinomial and $r$-variants of several classic objects of combinatorial probability, such as the random recursive and Hoppe trees, random set partitions and compositions, the Chinese restaurant process, Feller's coupling, and some others.
Just as various classic combinatorial numbers~--~like
Stirling, Eulerian 
and Lah numbers~--~emerge as essential ingredients defining the distributions of the mentioned processes, the so-called $r$-versions of these numbers appear in exact distributional formulas for the multinomial and $r$-counterparts. This approach allows us to offer a concise probabilistic interpretation for various identities involving $r$-versions of these combinatorial numbers, which were either unavailable or meaningful only for specific values of the parameter $r$. We analyze the derived distributions for fixed-size structures and establish distributional limit theorems as the size tends to infinity. Utilizing the aforementioned generalized Stirling numbers of both kinds, we define and analyze $(r,s)$-Lah distributions, which have arisen 
in the existing literature on combinatorial probability in various contexts.
\end{abstract}

\author{Alexander Iksanov, Zakhar Kabluchko, Alexander Marynych, Vitali Wachtel}

\maketitle

\tableofcontents

\section{Introduction}

\subsection{\texorpdfstring{$r$}{r}-Stirling numbers}

The usual Stirling numbers of both kinds play a pervasive role in combinatorics and discrete probability, often appearing prominently in diverse counting problems as an indispensable element. The list of quintessential examples begins with the enumeration of cycles in permutations, records in random samples, and blocks in set partitions. It extends by enumerating alignments and surjections, progressing to the counting of recursive trees with various profile restrictions. This list could have potentially been
indefinitely supplemented. We refer to the excellent treatises
\cite{charalambides_book_enum_combin,graham_etal_book,mansour_book,mansour_schork_book_commutation_relations,mezo_book,Stanley_book} for comprehensive information on the classic Stirling numbers.  A thorough historic account and an extensive bibliography can be found in the books~\cite{mansour_book} and~\cite{mansour_schork_book_commutation_relations}.

As detailed in the aforementioned references, there exist numerous generalizations of the Stirling numbers. In the focus of the present paper are the so-called $r$-Stirling numbers which depend on an additional parameter $r\in \R$. These numbers have been intensively investigated
in the combinatorics literature. Given that the usual Stirling numbers play a central role in combinatorial probability, it is natural to ask whether there exist natural $r$-deformations of the classic discrete random structures (such as random permutations and partitions, or random recursive trees) whose distributional properties are described by the $r$-Stirling numbers. One of the purposes
of the present paper is to construct these $r$-deformed structures and explore
their properties.

Perhaps, the easiest way to define the $r$-Stirling numbers $\stirling{n}{k}_r$ and $\stirlingsec{n}{k}_r$ is by using the finite differences calculus. Let $\C[r]$ be the vector space of polynomials in the variable $r$ with coefficients from $\C$.
Let $D$ be the differentiation operator and $\Delta$ the finite difference operator acting on $\C[r]$ as follows:
$$
D f(r) = f'(r),
\qquad
\Delta f(r) = f(r+1) - f(r).
$$
Then, using the notation $r^{n\uparrow}=r(r+1)\cdots(r+n-1)$ for the rising
factorial,
\begin{equation}\label{eq:r_stirling_def_finite_difference}
\stirling{n}{k}_r := \frac {D^k} {k!}  r^{n\uparrow},
\qquad
\stirlingsec{n}{k}_r := \frac {\Delta^k} {k!}  r^{n},
\qquad
n\in \N,
\quad
k\in \{0,1,\dots, n\}.
\end{equation}
For example, $\stirling{n}{0}_r = r^{n\uparrow}$ and $\stirlingsec{n}{0}_r = r^{n}$.
The usual Stirling numbers $\stirling{n}{k}$ and $\stirlingsec{n}{k}$ are recovered by evaluating the expressions at $r=0$, that is,
$$
\stirling{n}{k} = \stirling{n}{k}_0 = \left(\frac {D^k} {k!}  r^{n\uparrow}\right)\Big|_{r=0},
\qquad
\stirlingsec{n}{k} = \stirlingsec{n}{k}_0 = \left(\frac {\Delta^k} {k!}  r^{n}\right)\Big|_{r=0}.
$$
In general, the parameter $r$ may take arbitrary real (or complex) values. If we fix $n$ and $k$, then $\stirling{n}{k}_r$ and $\stirlingsec{n}{k}_r$ are polynomials in $r$, as directly follows from~\eqref{eq:r_stirling_def_finite_difference}.  The above definitions can be naturally extended to arbitrary $n\in \N_0:=\N\cup\{0\}$ and $k\in \Z$ by putting
$$
\stirling{n}{k}_r:=\stirlingsec{n}{k}_r:=0
\text{ for }
n\in \N_0
\text{ and }
k\notin\{0,\dots,n\},
\qquad
\stirling{0}{0}_r:=\stirlingsec{0}{0}_r:=1.
$$

The first systematic investigation of the
$r$-Stirling numbers (also called the `weighted Stirling numbers' and 'non-central Stirling numbers') is due to
Carlitz~\cite{Carlitz:1980-1,Carlitz:1980-2}. Some follow-up papers are~\cite{broder:1984,charalambides_koutras_diff_gen_factorials,koutras,shanmugam_r_stirling}.
A clear exposition of
properties of the $r$-Stirling numbers can be found in the paper of Broder~\cite{broder:1984} (where the term `$r$-Stirling numbers' appears
for the first time) and in the books of Mez\H{o}~\cite[Chapter~8]{mezo_book} and Charalambides~\cite[\S~8.5]{charalambides_book_enum_combin}, \cite[Chapter~3]{charalambides_book_combinatorial_methods}. The paper~\cite{nyul:2015}, whose notation we follow, contains a nice overview of the identities satisfied by the $r$-Stirling numbers. We note in passing that our notation corresponds to Carlitz's notation $R(n,k,r)=\stirlingsec{n}{k}_r$ and $R_1(n,k,r)=\stirling{n}{k}_r$ and what we denote by $\stirling{n}{k}_r$ and $\stirlingsec{n}{k}_r$ corresponds to $\stirling{n+r}{k+r}_r$ and $\stirlingsec{n+r}{k+r}_r$ in~\cite{broder:1984,mezo_book}. In particular, the second definition in~\eqref{eq:r_stirling_def_finite_difference} can be read from Eq.~(3.8) in~\cite{Carlitz:1980-1}.

While Carlitz's approach is purely analytic and hinges on manipulations with generating functions, Broder's starting point is a combinatorial definition, see Eq.~(1) and (2) in~\cite{broder:1984}. Adapting
to our notation and assuming that $r$ is a non-negative integer, this is
\begin{definition}\label{def:broder} For $r\in\N_0$, $n\in\N$ and $k\in\{0,1,\dots,n\}$, $\stirling{n}{k}_r$ is the number of  permutations of the set $\{1,2,\dots,n+r\}$ having $k+r$ cycles such that the numbers $1,2,\dots,r$ are in different cycles; $\stirlingsec{n}{k}_r$ is the number of partitions of the set $\{1,2,\dots,n+r\}$ into $k+r$ disjoint subsets such that the numbers $1,2,\dots,r$ are in distinct blocks.
\end{definition}
While being combinatorially appealing Definition \ref{def:broder}
has a drawback of being only valid
for integer $r$. In contrast, Carlitz's method is free of this shortcoming.
Unless stated otherwise, in all subsequent considerations we stick to Carlitz's viewpoint and do not assume that $r$ is integer.

The properties of the $r$-Stirling numbers are parallel to those of the usual Stirling numbers. For example, see Eq.~(3.4) and (5.8) in~\cite{Carlitz:1980-1} or Eq.~(48) and (50) in~\cite{broder:1984},
\begin{align}\label{eq:stirling_r_rising_factorials}
(x+r)^{n\uparrow} =  \sum_{k=0}^n \stirling{n}{k}_r x^k,
\qquad
(x+r)^n = \sum_{k=0}^n \stirlingsec{n}{k}_r x^{k\downarrow},
\end{align}
where $x^{k\downarrow}=x(x-1)\cdots(x-k+1)$ denotes the falling factorial. The (vertical) exponential generating functions of the $r$-Stirling numbers are given by
\begin{align}
\sum_{n=k}^\infty \stirling{n}{k}_r \frac{x^n}{n!}
&=
\frac{1}{k!} \left( \log\left( (1-x)^{-1} \right) \right)^k (1-x)^{-r}
,\label{eq:def_r_stir_first}
\\
\sum_{n=k}^\infty \stirlingsec{n}{k}_r \frac{x^n}{n!}
&=
\frac{1}{k!} \left(\eee^x-1\right)^k \eee^{rx}, \label{eq:def_r_stir_second}
\end{align}
see Eq.~(3.9) in~\cite{Carlitz:1980-1} and Eq.~(36) in~\cite{broder:1984}. It is known (see, e.g.\ p.~1661 in~\cite{nyul:2015}) that
\begin{align}
\stirling{n}{k}_r
&=
\sum_{j=k}^n \stirling{n}{j} \binom{j}{k} r^{j-k}
=
\sum_{j=k}^n \binom{n}{j} \stirling{j}{k} r^{(n-j)\uparrow},  \label{eq:stir1_as_poly}
\\
\stirlingsec{n}{k}_r
&=
\sum_{j=k}^n \binom{n}{j} \stirlingsec{j}{k} r^{n-j}
=
\sum_{j=k}^n \stirlingsec{n}{j} \binom{j}{k} r^{(j-k)\downarrow}. \label{eq:stir2_as_poly}
\end{align}
These formulas demonstrate once again that, for fixed $n\in \N_0$, $k\in \{0,\dots, n\}$, the $r$-Stirling numbers $\stirling{n}{k}_r$ and $\stirlingsec{n}{k}_r$ are polynomials in $r$.

\subsection{Three families of probability distributions related to the \texorpdfstring{$r$}{r}-Stirling numbers}
There are various probability distributions related to usual Stirling numbers, see~\cite{charalambides_book_combinatorial_methods,charalambides_singh_stirling_review,nishimura_sibuya,pinsky_a_view,Pitman_book}. The most prominent example is the so-called \emph{Stirling distribution of the first kind} (or a \emph{Stirling--Karamata} distribution) $\Stirone{n}{\theta}$ with parameters $n\in \N$ and $\theta>0$, which is a discrete probability measure on $\{0,\dots, n\}$ given by
\begin{equation}\label{def:stirling-karamata}
\Stirone{n}{\theta}(\{k\})=\stirling{n}{k} \frac{\theta^k}{\theta^{n\uparrow}}, \qquad k\in \{0,\dots, n\}.
\end{equation}
If $\theta = 1$, $\Stirone{n}{1}$ is the distribution of the number of cycles in a random permutation picked uniformly at random from the symmetric group of $n$ elements. More generally, $\Stirone{n}{\theta}$ is the distribution of the number of cycles in a random Ewens permutation with parameter $\theta$. This distribution also pops up as the number of occupied tables in the Chinese restaurant process, see Eq.~(3.11) in~\cite{Pitman_book}, or the root degree in the random Hoppe tree~\cite{Hoppe:1986}. For a review of the related concept of the Ewens sampling formula and similar topics we refer to~\cite{crane_ewens_formula}. We will discuss these connections later on in a larger 
generality.

Related to the Stirling numbers of the second kind is the {\em Stirling--Sibuya} distribution defined by
\begin{equation}\label{def:stirling-sibuya}
\StirSibuya{n}{\theta}(\{k\})=\stirlingsec{n}{k} \frac{\theta^{k\downarrow}}{\theta^{n}}, \qquad k\in \{0,\dots, n\},\quad n\in\N.
\end{equation}
Here,
$\theta$ is either a positive integer or a real number larger than $n-1$, so that the falling factorial is nonnegative. If $\theta$ is a positive integer, $\StirSibuya{n}{\theta}$ is the distribution of the number of occupied boxes after allocating $n$ balls over $\theta$ urns equiprobably. The fact that both~\eqref{def:stirling-karamata} and~\eqref{def:stirling-sibuya} do define
probability distributions
follows from equalities~\eqref{eq:stirling_r_rising_factorials} with $r=0$. Yet another instance is the \emph{Stirling distribution of the second kind} $\Stirtwo{n}{\theta}$. To define it, let $T_n$ be the Touchard polynomial given by
$$
T_n(z)=\sum_{k=0}^{n}\stirlingsec{n}{k}z^k,\quad z\in\C,\quad n\in\N.
$$
Then, for a parameter $\theta >0$,
$$
\Stirtwo{n}{\theta}(\{k\})=\frac{1}{T_n(\theta)}\stirlingsec{n}{k}\theta^k, \qquad k\in \{0,\dots, n\},\quad n\in\N.
$$
This is the distribution of the number of blocks in a Gibbs (uniform, for $\theta=1$) random partition of $\{1,\dots,n\}$. We refer to~\cite{charalambides_book_combinatorial_methods,charalambides_singh_stirling_review,Kabluchko+Marynych+Pitters:2024} and references therein for further properties and applications of the aforementioned distributions.

From a purely analytic viewpoint one can introduce the $r$-versions of the above distributions by simply using~\eqref{eq:stirling_r_rising_factorials} and the $r$-Touchard polynomials.
\begin{definition}[\cite{nishimura_sibuya,sibuya_stirling_family,sibuya_nishimura_breaking}]\label{def:rStirling1}
Let $n\in\N$, $\theta>0$ and $r\geq 0$. The \emph{$r$-Stirling distribution of the first kind} $\rStirone{n}{\theta}{r}$ is a discrete
probability measure on $\{0,\dots, n\}$ given by
\begin{equation}\label{def:r_stirling-first}
\rStirone{n}{\theta}{r}(\{k\})=\stirling{n}{k}_r \frac{\theta^k}{(\theta+r)^{n\uparrow}}, \qquad k\in \{0,\dots, n\}.
\end{equation}
\end{definition}
The distribution defined by~\eqref{def:r_stirling-first} appears
in~\cite{nishimura_sibuya}, see also~\cite[Example~2.1]{charalambides_book_combinatorial_methods} and a survey~\cite{sibuya_stirling_family}. Also, in~\cite{sibuya_nishimura_breaking} it pops
up as the distribution of a random variable $X_{n+1}-1$, where $X_{n+1}$ is the number of records in a sample of size $n+1$ of independent non-identically distributed random variables (Nevzorov's model).
Our notation $\rStirone{n}{\theta}{r}$ corresponds to $\mathsf{Str1F}(n,\theta,r)$ in~\cite{nishimura_sibuya} and~\cite{sibuya_nishimura_breaking}.

For $r\geq 0$ and $n\in\N$, let the \emph{$r$-Touchard polynomial}, also called the \emph{$r$-Bell polynomial} in~\cite[\S~8.3]{mezo_book} and~\cite{mezoe_r_bell}, be defined by
$$
T_{n,r}(z):=\sum_{k=0}^{n}\stirlingsec{n}{k}_r z^k,\quad z\in\C.
$$
\begin{definition}
Let $n\in\N$, $\theta>0$ and $r\geq 0$. The \emph{$r$-Stirling distribution of the second kind} $\rStirtwo{n}{\theta}{r}$ is a discrete
probability measure on $\{0,\dots, n\}$ given by
\begin{equation}\label{def:r_stirling-second}
\rStirtwo{n}{\theta}{r}(\{k\})=\stirlingsec{n}{k}_r \frac{\theta^k}{T_{n,r}(\theta)}, \qquad k\in \{0,\dots, n\}.
\end{equation}
\end{definition}
\begin{definition}[\cite{nishimura_sibuya,sibuya_stirling_family}]
Assume that $n\in\N$, $r\geq 0$ and either $\theta\in\N$ or $\theta>n-1$. The \emph{$r$-Stirling--Sibuya distribution} $\rStirSibuya{n}{\theta}{r}$ is a discrete
probability measure on $\{0,\dots, n\}$ given by
\begin{equation}\label{def:r_stirling-sibuya}
\rStirSibuya{n}{\theta}{r}(\{k\})=\stirlingsec{n}{k}_r \frac{\theta^{k\downarrow}}{(\theta+r)^n}, \qquad k\in \{0,\dots, n\}.
\end{equation}
\end{definition}

The $r$-Stirling--Sibuya distribution $\rStirSibuya{n}{\theta}{r}$ appears
in~\cite{nishimura_sibuya}, see Section 5 therein, as the distribution of a position at time $n$ of a certain time-inhomogeneous random walk on $\N_0$; see also~\cite[Example 2.6]{charalambides_book_combinatorial_methods}. An urn model leading to the $\rStirSibuya{n}{\theta}{r}$ distribution is also discussed in~\cite{nishimura_sibuya} and will be important in our considerations of random $r$-partitions in Section~\ref{sec:r_versions2} below.  Our notation $\rStirSibuya{n}{\theta}{r}$ corresponds to $\mathsf{Str2F}(n,\theta,r)$ in~\cite{nishimura_sibuya}.

Distributions related to further generalizations of Stirling numbers are discussed in~\cite{charalambides_book_combinatorial_methods}, \cite[Example~3.2.3]{Pitman_book} and~\cite{huillet_occupancy_problems_generalized_stirling}.

In this work, we introduce a general concept of multinomial combinatorial structures which we find interesting in their own. Multinomial structures are obtained by taking the classic random structures (such as random permutations, partitions or compositions) and coloring their cycles or blocks into $d\in\N$ colors. It will become clear that the
$r$-distributions defined above (and some other) occur in the special case $d=2$.

\vspace{2mm}

The rest of the paper is organized as follows. In Section~\ref{sec:multinomial} we introduce and investigate
the multinomial versions of some classic combinatorial stochastic processes. In Sections~\ref{sec:r_versions1} we show how the $r$-distributions
arise as the distributions of some functionals defined on the
multinomial structures of Section~\ref{sec:multinomial}. Section~\ref{sec:r_versions2} is devoted to the analysis of random incomplete partitions and an accompanying notion of random $r$-partitions. A similar analysis of incomplete compositions and random $r$-compositions is carried out in Section~\ref{sec:r_versions3}. In Section~\ref{sec:Lah_distribution} we discuss yet another class of combinatorial numbers, called $r$-Lah numbers, and introduce
generalized Lah distributions. Specific occurrences of these distributions have already received some attention
in the literature on combinatorial probability across diverse contexts. Various limit theorems for these distributions are discussed in the concluding Section~\ref{sec:Lah_limits}.

\subsection{Notation}
Besides the already introduced notation, we also use the following abbreviations and conventions. We put $[n]:=\{1,2,\dots,n\}$. For the asymptotic equivalence we use the symbol $\simeq$. Thus, $a_n\simeq b_n$ if, and only if, $a_n/b_n\to 1$ as $n\to\infty$. 
We use the standard notation $x_+=\max(x, 0)$ and $x_-=\max(-x,0)$ for $x\in\mathbb{R}$. The cardinality of a finite set $A$ is denoted by $\#A$. For a formal power series $f(x)=\sum_{n=0}^{\infty}f_n x^n$, we stipulate that $[x^n]f(x)=f_n$. If $X$ is a random element taking values in some space $S$ and $\mu$ is a probability distribution on $S$, we write
$X\sim \mu$ to denote that $X$ has distribution $\mu$. If $(\mu[\theta])_{\theta\in P}$ is a family of probability measures depending on a parameter $\theta$ and $\mathfrak{M}$ is a probability measure on $P$, we denote by $\mu[\mathfrak{M}]$ the mixture of $(\mu[\theta])_{\theta\in P}$. Thus,
\begin{equation}\label{eq:mixtures_convention}
(\mu[\mathfrak{M}])(\cdot):=\int_P \mu[\theta](\cdot)\mathfrak{M}({\rm d}\theta).
\end{equation}
In addition to
the already introduced probability measures related to the Stirling numbers, many standard distributions show up throughout the text.
We denote by $\Bin{n}{p}$ the binomial distribution with parameters $n$ and $p$ and by $\Bern{p}$ the Bernoulli distribution with parameter $p$, that is, $\Bern{p}(\{1\})=p=1-\Bern{p}(\{0\})$. Furthermore,
$\Poi{\lambda}$ denotes
the Poisson distribution with parameter $\lambda$ and $\Mult{n}{p_1,\dots,p_d}$ denotes
the multinomial distribution with parameters $n$ and $p_1,\dots,p_d$. Notation for other distributions will be introduced in the
text just before the first appearance. Keeping in mind Eq.~\eqref{eq:mixtures_convention} we ubiquitously use
formulas like $\Bin{\Poi{\lambda}}{p}=\Poi{\lambda p}$ involving the mixtures of distributions.
With an abuse of notation, $\Bern{p}$ can denote either
the distribution or
a random variable with this distribution (implicitly assumed to be defined on some probability space). The same disclaimer also applies to the other distributions.

\section{Multinomial structures}\label{sec:multinomial}
In this section we introduce
$r$-analogues of some classic random structures. Examples include random Ewens permutations, random recursive (and Hoppe) trees and the Chinese restaurant process. The $r$-versions will be defined as special cases of more general, multinomial versions of these structures.

\subsection{Multinomial Stirling numbers}
The \emph{multinomial Stirling cycle number} $\stirling{n}{k_1,\dots, k_d}$ is defined as the number of permutations of $[n]$ with $k_1+\dots+k_d$ cycles which
are colored in $d$ possible colors $1,\dots, d$ such that the number of cycles with color $j$ is $k_j$, for all $j=1,\dots, d$. Clearly,
$$
\stirling{n}{k_1, \dots, k_d} = \stirling{n}{k_1+ \dots + k_d} \binom{k_1 + \dots + k_d}{k_1,\dots, k_d}.
$$
Similarly, the \emph{multinomial Stirling partition number} $\stirlingsec{n}{k_1,\dots, k_d}$ is defined as the number of partitions of $[n]$ into $k_1+\dots + k_d$ nonempty blocks which are colored in $d$ possible colors such that $k_j$ is the number of blocks with color $j$, for all $j=1,\dots, d$. Plainly,
$$
\stirlingsec{n}{k_1, \dots, k_d} = \stirlingsec{n}{k_1+ \dots + k_d} \binom{k_1 + \dots + k_d}{k_1,\dots, k_d}.
$$
In both definitions, the parameters $n\in \N_0$ and $k_1,\dots,k_d\in \N_0$ satisfy $k_1+\dots + k_d\leq n$ with the convention that $\stirling{0}{0,\dots, 0} = 1$. Additionally, if $n\in \N_0$ and  $k_1, \dots, k_d \in \Z$ fail to satisfy the above restrictions (meaning that one of the $k_j$'s is negative or their sum is larger than $n$), we define the corresponding multinomial Stirling numbers to be $0$. The properties of the multinomial Stirling numbers are reviewed in~\cite[Section~4]{belkhir_multivar_lah}. Here we only need
the following result.
\begin{proposition}
With
$x_1,\dots, x_d$ denoting
variables,
\begin{align}
&(x_1 + \dots + x_d)^{n\uparrow} = \sum_{\substack{k_1,\dots, k_d\in \N_0\\ k_1+\dots+ k_d\leq n}} \stirling{n}{k_1,\dots, k_d} x_1^{k_1}\cdots x_d^{k_d}, \label{eq:stirling1_multi_genfunct}
\\
&(x_1 + \dots + x_d)^{n} = \sum_{\substack{k_1,\dots, k_d\in \N_0\\ k_1+\dots+ k_d\leq n}} \stirlingsec{n}{k_1,\dots, k_d} x_1^{k_1\downarrow}\cdots x_d^{k_d\downarrow}.  \label{eq:stirling2_multi_genfunct}
\end{align}
\end{proposition}
\begin{proof}
To prove~\eqref{eq:stirling1_multi_genfunct}, we use the formula $x^{n\uparrow} = \sum_{k=0}^n \stirling{n}{k} x^k$ with $x= x_1+\dots+ x_d$ and then the multinomial theorem:
\begin{align*}
(x_1 + \dots + x_d)^{n\uparrow}
&=
\sum_{k=0}^n \stirling{n}{k} (x_1+\dots + x_d)^k
=
\sum_{k=0}^n \stirling{n}{k}  \sum_{\substack{k_1,\dots, k_d\in \N_0\\ k_1+\dots+ k_d = k}} \binom{k}{k_1,\dots, k_d}x_1^{k_1}\cdots x_d^{k_d}
\\
&=
\sum_{\substack{k_1,\dots, k_d\in \N_0\\ k_1+\dots+ k_d\leq n}} \stirling{n}{k} \binom{k}{k_1,\dots, k_d} x_1^{k_1}\cdots x_d^{k_d}
=
\sum_{\substack{k_1,\dots, k_d\in \N_0\\ k_1+\dots+ k_d\leq n}} \stirling{n}{k_1,\dots, k_d} x_1^{k_1}\cdots x_d^{k_d}.
\end{align*}
Similarly, to prove~\eqref{eq:stirling2_multi_genfunct}, we use the formula $x^n = \sum_{k=0}^{n} \stirlingsec{n}{k}x^{k \downarrow}$ together with the analogue of the multinomial theorem for falling/rising factorials:
\begin{multline*}
(x_1 + \dots + x_d)^{n}=
\sum_{k=0}^n \stirlingsec{n}{k} (x_1+\dots + x_d)^{k\downarrow}=
\sum_{k=0}^n \stirlingsec{n}{k}  \sum_{\substack{k_1,\dots, k_d\in \N_0\\ k_1+\dots+ k_d = k}} \binom{k}{k_1,\dots, k_d} x_1^{k_1\downarrow}\cdots x_d^{k_d\downarrow}
\\
=\sum_{\substack{k_1,\dots, k_d\in \N_0\\ k_1+\dots+ k_d\leq n}} \stirlingsec{n}{k} \binom{k}{k_1,\dots, k_d} x_1^{k_1\downarrow}\cdots x_d^{k_d\downarrow}
=\sum_{\substack{k_1,\dots, k_d\in \N_0\\ k_1+\dots+ k_d\leq n}} \stirlingsec{n}{k_1,\dots, k_d} x_1^{k_1\downarrow}\dots x_d^{k_d\downarrow}.
\end{multline*}
The proof is complete.
\end{proof}

\subsection{Multinomial Ewens permutation}
Let $\mathfrak{S}_n$ be the set of all permutations of the set $[n]$. Let $\cycles (\rho)$ denote the set of cycles of a permutation $\rho\in \mathfrak{S}_n$ and $\#\cycles(\rho)$ denote the number of cycles.

A random \textit{Ewens permutation} with parameter $\theta>0$ is a random element $\bmrho$ of $\mathfrak{S}_n$ such that for each 
deterministic permutation $\rho \in \mathfrak{S}_n$
$$
\P[\bmrho = \rho] = \frac{\theta^{\# \cycles(\rho)}}{\theta^{n\uparrow}},
$$
see Example 2.19 and Chapter 5 in~\cite{Arratia+Barbour+Tavare:2003}, and also~\cite[Section 3.2]{Pitman_book}. We 
denote this
distribution by $\Ewens{n}{\theta}$, so that $\bmrho\sim \Ewens{n}{\theta}$. Note that the special case $\theta=1$ corresponds to a random uniform permutation. It is important for what follows
that the Ewens distribution can also be defined for $\theta=0$. In this case it coincides with the uniform distribution on the subset of $\mathfrak{S}_n$ consisting of permutations with exactly one cycle. The total number of such permutations is $(n-1)!$.

The number of cycles of $\bmrho$ follows the Stirling distribution of the first kind with parameter $\theta$:
$$
\P[\# \cycles(\bmrho) = k] = \stirling{n}{k} \frac{\theta^k}{\theta^{n\uparrow}}, \qquad k\in \{0,\dots, n\}
$$
(in short, $\# \cycles (\bmrho) \sim \Stirone{n}{\theta}$).
Indeed, $\stirling{n}{k}$ is the number of permutations of $[n]$ with $k$ cycles, and each
such a permutation has probability $\theta^k/\theta^{n\uparrow}$ under the Ewens distribution $\Ewens{n}{\theta}$.

Now we introduce the multinomial analogue of the Ewens distribution. A \emph{colored permutation} is a permutation of $[n]$ in which each
cycle has one of $d$ possible colors. The number of colors $d\in \N$ is fixed once and for all. Formally, a colored permutation is a  pair $(\rho, \phi)$, where $\rho\in \mathfrak{S}_n$ is a permutation and $\phi: \cycles(\rho) \to \{1,\dots, d\}$ is a map assigning colors to cycles of $\rho$. A \emph{multinomial random Ewens permutation} is obtained by taking a Ewens permutation and coloring the cycles independently of each other, with $p_j$ denoting the probability of assigning color $j\in \{1,\dots, d\}$ to a cycle.  More precisely, we fix parameters $\theta>0$ and $p_1,\dots, p_d\geq 0$ such that $p_1+\dots + p_d = 1$, define $\theta_j := \theta p_j$ and adopt the following
\begin{definition}
A random colored permutation $(\bmrho, \bmphi)$ has a \emph{multinomial Ewens distribution} $\MultiEwens{n}{\theta}{p_1,\dots,p_d}$ if
\begin{equation}\label{eq:multinomialEwensDistr}
\P[(\bmrho, \bmphi) = (\rho, \phi)] = \frac{\theta^{\# \cycles (\rho)}}{\theta^{n\uparrow}} p_1^{\# \cycles_1(\rho, \phi)} \dots p_d^{\# \cycles_d(\rho, \phi)}
=
\frac{\theta_1^{\# \cycles_1(\rho, \phi)} \dots \theta_d^{\# \cycles_d(\rho, \phi)}}{\theta^{n\uparrow}}
\end{equation}
for each 
deterministic colored permutation $(\rho, \phi)$. Here, $\cycles_j(\rho, \phi)$ denotes the set of cycles of color $j$ in the colored permutation $(\rho, \phi)$.
\end{definition}
From the definition of multinomial Stirling numbers it follows that the cycle counts  of such $(\bmrho, \bmphi)$ are distributed as follows:
$$
\P[\#\cycles_1(\bmrho, \bmphi) = k_1,\dots, \#\cycles_d(\bmrho, \bmphi) = k_d]
=
\stirling{n}{k_1,\dots, k_d}
\frac{\theta^{k_1+\dots+ k_d}}{\theta^{n\uparrow}} p_1^{k_1}\cdots p_d^{k_d} 
$$
for all $k_1,\dots, k_d\in \N_0$ subject to $k_1+\dots + k_d \leq n$. This motivates the following
\begin{definition}
A random vector $(K_1,\dots, K_d)$ with values in $\N_0^d$ has a \emph{multinomial Stirling distribution} of the first kind $\MultStirone{n}{\theta}{p_1,\dots, p_d}$ if
$$
\P[K_1 = k_1,\dots, K_d = k_d]
=
\stirling{n}{k_1,\dots, k_d}
\frac{\theta^{k_1+\dots+ k_d}}{\theta^{n\uparrow}} p_1^{k_1}\cdots p_d^{k_d}
=
\stirling{n}{k_1,\dots, k_d}
\frac{\theta_1^{k_1}\cdots \theta_d^{k_d}}{\theta^{n\uparrow}} 
$$
for all $k_1,\dots, k_d\in \N_0$ satisfying $k_1+\dots + k_d \leq n$, where $\theta_j=\theta p_j$ for $j\in\{1,\dots,d\}$.
\end{definition}
Thus,
the cycle counts of $(\bmrho, \bmphi)\sim \MultiEwens{n}{\theta}{p_1,\dots, p_d}$ satisfy
$$
(\#\cycles_1(\bmrho, \bmphi),\dots, \#\cycles_d(\bmrho, \bmphi))~\sim~\MultStirone{n}{\theta}{p_1,\dots, p_d}.
$$
\begin{proposition}\label{prop:rep_mult_stir_sum_indicators}
The multivariate generating function of the random vector $(K_1,\dots, K_d)$ that follows $\MultStirone{n}{\theta}{p_1,\dots, p_d}$ distribution is given by
\begin{equation}\label{eq:multistir_gen_funct}
\E [t_1^{K_1}\cdots t_d^{K_d}] = \frac{(\theta_1 t_1 + \dots + \theta_d t_d)^{n\uparrow}}{(\theta_1 + \dots + \theta_d)^{n\uparrow}},
\end{equation}
where $\theta_j= \theta p_j$ for $i\in\{1,\dots, d\}$.  Moreover, if $I_1,\dots, I_n$ are independent random vectors in $\N_0^d$ with
\begin{equation}\label{eq:rep_multistir_indicator_distr}
\P[I_\ell = \eee_j] = \frac{\theta_j}{\theta + \ell - 1},
\qquad
j\in \{1,\dots, d\},
\qquad
\P[I_\ell =  {\bf 0}] = \frac{\ell-1}{\theta + \ell - 1},
\end{equation}
where $\eee_1,\dots, \eee_d$ is the standard basis of $\R^d$ and ${\bf 0}$ is the zero vector in $\R^d$, then $(K_1,\dots, K_d)$ has the same distribution as $I_1+\dots+ I_n$. In particular, $K_j \sim \rStirone{n}{\theta_j}{(\theta-\theta_j)}$.
\end{proposition}
\begin{proof}
By definition of the multinomial Stirling distribution,
\begin{align*}
\E [t_1^{K_1}\cdots t_d^{K_d}]
=
\frac{1}{\theta^{n\uparrow}}
\sum_{\substack{k_1,\dots, k_d\in \N_0\\ k_1+\dots + k_d \leq n}}
\stirling{n}{k_1,\dots, k_d}
 (\theta_1 t_1)^{k_1}\dots (\theta_d t_d)^{k_d}
&=
\frac{(\theta_1 t_1 + \dots + \theta_d t_d)^{n\uparrow}}{\theta^{n\uparrow}},
\end{align*}
where the second equality  follows from~\eqref{eq:stirling1_multi_genfunct}. It remains to note that the generating function of the random vector $I_1 + \dots + I_n$ is the same. To identify the distribution of $K_j$, observe that~\eqref{eq:multistir_gen_funct} implies
$$
\E [t_j^{K_j}] = \frac{(\theta_j t_j  + \theta - \theta_j)^{n\uparrow}}{\theta^{n\uparrow}}
\overset{\eqref{eq:stirling_r_rising_factorials}}{=}
\sum_{k=0}^n \stirling{n}{k}_{\theta-\theta_j} \frac{\theta_j^k t_j^k}{(\theta_j+ \theta - \theta_j)^{n\uparrow}},
$$
which proves that $K_j \sim \rStirone{n}{\theta_j}{(\theta-\theta_j)}$, see Definition~\ref{def:rStirling1}.
\end{proof}

The following aggregation property follows directly from the definition:

\begin{proposition}
If $(K_1,\dots, K_d) \sim \MultStirone{n}{\theta}{p_1,p_2,\dots, p_d}$, then
$$
(K_1 + K_2, K_3,\dots, K_d)~\sim~\MultStirone{n}{\theta}{p_1+p_2,p_3,\dots, p_d}.
$$
\end{proposition}

Recalling our convention~\eqref{eq:mixtures_convention} concerning the
mixtures of distributions we
write
\begin{equation}\label{MultStirone_as_mixed_mult}
\MultStirone{n}{\theta}{p_1,\dots,p_d} = \Mult{\Stirone{n}{\theta}}{p_1,\dots, p_d}.
\end{equation}
This provides
another way to prove
Proposition~\ref{prop:rep_mult_stir_sum_indicators}.

\subsection{Multinomial Chinese restaurant and the colored Feller coupling}
Fix parameters $\theta_1,\dots, \theta_d\geq 0$ and suppose that $\theta := \theta_1 + \dots + \theta_d>0$.
Customers labeled by $1,2,\dots,n$ arrive in discrete time and take seats
at round tables. Each table is colored in one of $d$ colors.
Each customer $\ell$, after arrival, either creates a new table of color $j\in \{1,\dots, d\}$ with probability $\theta_j/(\theta + \ell-1)$ or takes a seat
at some already existing table next to
customer $i$ (in the counterclockwise direction) with probability $1/(\theta + \ell-1)$, for each customer $i=1,\dots, \ell-1$.  After arrival of $n$ customers, this process generates a random permutation of $[n]$, if we regard tables (read counterclockwise) as cycles.

\begin{proposition}\label{prop:multinomialCRP}
The colored permutation generated by the multinomial Chinese restaurant process has the distribution
$$
\MultiEwens{n}{\theta}{\theta_1/\theta,\dots, \theta_d/\theta}.
$$
In particular, if $C_{n;j}$ denotes the number of tables (cycles) of color $j\in \{1,\dots, d\}$, then
$$
(C_{n;1},\dots, C_{n;d})~\sim~\MultStirone{n}{\theta}{\theta_1/\theta,\dots, \theta_d/\theta}.
$$
\end{proposition}
\begin{proof}[First proof]
Consider the multinomial Chinese restaurant process and remove the colors of the tables. This is the usual Chinese restaurant process with parameter $\theta$. It is 
known, see Example 2.19 in~\cite{Arratia+Barbour+Tavare:2003}, that the process
generates a random permutation which is distributed according to $\Ewens{n}{\theta}$. Now, recall that in the colored version of the process each new table receives color $j$ with probability $\theta_j/\theta$. By definition, this gives the multinomial Ewens distribution.
\end{proof}

\begin{proof}[Second proof]
Another approach hinges on a colored version of the classic Feller coupling for Ewens' permutations, see pp.~95-96 in~\cite{Arratia+Barbour+Tavare:2003}. Let $(D_i)_{i\in\mathbb{N}}$, be independent random variables such that $D_i$ takes values in the set
$$
\{\close_1, \dots, \close_d\}\cup \{2,3,\dots,i\}
$$
(the meaning of `$\close_j$' will be explained a bit later) and has the distribution
$$
\P[D_i=\close_j]=\frac{\theta_j}{\theta+i-1},\quad j=1,\dots,d,\quad \P[D_i=k]=\frac{1}{\theta+i-1},\quad k=2,\dots,i,\quad i\in\mathbb{N}.
$$
These variables are then used to generate a colored random permutation in the ordered cycle notation as follows. Start the first cycle with `$(1$' and make a $(n+d-1)$-choice using the variable $D_n$. If $D_n=k$, then  $k$ is added to the first cycle transforming `$(1$' to `$(1\; k$'. If $D_n=\close_j$, then one closes
the first cycle, paints it using the color $j$ and opens the next cycle. Continuing in this way and using $D_{n-1},\dots,D_1$ produces a colored random permutation, say $(\bmrho, \bmphi)$, in the ordered cycle notation. Thus, the value $D_i=\close_j$ corresponds to closure of the current cycle, painting it using the color $j$ and starting a new cycle with the smallest unused integer. The total number of cycles in $(\bmrho, \bmphi)$ is equal to
$$
C_n=\sum_{i=1}^{n}\1[D_i\in \{\close_1,\ldots, \close_d\}]
$$
and, since $\P[D_i\in \{\close_1,\ldots, \close_d\}]=\theta/(\theta+i-1)$, follows the Stirling distribution of the first kind $\Stirone{n}{\theta}$. Furthermore, the probability of seeing any fixed (uncolored) permutation $\rho\in\mathfrak{S}_n$, is equal to
$$
\P[\bmrho =\rho]
=
\frac{\theta^{\#{\cycles}(\rho)}}{\theta^{n\uparrow}}.
$$
Thus, erasing the colors in $(\bmrho, \bmphi)$ produces Ewens' permutation. The number of cycles of color $j$ is
$$
C_{n;j}=\sum_{i=1}^{n}\1[D_i=\close_j],\quad j=1,\dots,d.
$$
Using that, independently of $i=1,\dots,n$,
$$
\P[D_i=\close_j|D_i\in \{\close_1,\dots,\close_d\}]=\frac{\theta_j}{\theta},\quad j=1,\dots,d,
$$
we conclude
that $(\bmrho, \bmphi)\sim \MultiEwens{n}{\theta}{\theta_1/\theta,\dots, \theta_d/\theta}$ and $(C_{n;1},\dots, C_{n;d})$ has the  multinomial Stirling distribution $\MultStirone{n}{\theta}{\theta_1/\theta,\dots, \theta_d/\theta}$.
\end{proof}

Since the cycles are painted independently, the following is a simple consequence of the marking theorem for Poisson processes in conjunction with Theorem 5.1 in~\cite{Arratia+Barbour+Tavare:2003}.
\begin{proposition}\label{prop:cycle_counts}
Let $(\bmrho, \bmphi)$ be the multinomial Ewens permutation and let $C_{n;j;k}$ be the number of cycles of color $j\in \{1,\dots,d\}$ and length $k\in\N$. Then,
$$
(C_{n;j;k}:\;j\in \{1,\dots,d\},\;k\in\N)~\todistr~(\Poi{\theta_j/k}:\;j\in \{1,\dots,d\},\;k\in\N)
$$
in the product topology on $\mathbb{N}_0^{d\times \N}$, where the limiting Poisson variables are mutually independent.
\end{proposition}

\subsection{Multinomial Hoppe forests}\label{sec:multi_hoppe}
The classic sequential construction of a \emph{random recursive tree} $\RRT{n}$ with $n$ nodes $1, \dots, n$  and the root labeled by $0$ proceeds by recursively attaching node $\ell\in \{1,\dots,n\}$ to one of the existing nodes $0,1,\dots,\ell-1$ picked uniformly at random and independently of the previous choices; see, e.g.,~\cite{smythe_mahmoud_survey_RRT}. The \emph{Hoppe tree} $\Hoppe{n}{\theta}$ is defined similarly, see~\cite{Leckey+Neininger:2013}, but now the probability of attaching the node $\ell$ to the root is $\theta/(\theta+\ell-1)$, while the probability of attaching $\ell$ to any other node $1,\dots, \ell-1$ is $1/(\theta + \ell-1)$. Here, $\theta>0$ is a parameter representing the `weight' of the root, while the `weight' of any other node is $1$. In the special case $\theta=1$ one recovers the random recursive tree $\RRT{n}$. The name `Hoppe tree' originates from the corresponding Hoppe's urn model introduced in~\cite{hoppe_urns_ewens,Hoppe:1986,hoppe_sampling} in the context of the Ewens sampling formula.

We now define a multinomial generalization of Hoppe trees, called \emph{multinomial Hoppe forest} and denoted by $\MultiHoppe{n}{\theta}{\theta_1/\theta,\dots, \theta_d/\theta}$.   Fix parameters $\theta_1,\dots, \theta_d\geq 0$ and suppose that $\theta := \theta_1 + \dots + \theta_d>0$.
Start at time $0$ with $d$ roots, labeled $\rt_1,\dots, \rt_d$,  which have weights $\theta_1,\dots, \theta_d$. Nodes, labeled by $1,2,\dots, n$, arrive in discrete time and are added as children to the roots (with probabilities proportional to the weights $\theta_1, \dots, \theta_d$) or to the nodes already present in the forest (with probability proportional to $1$ for each node). More precisely, suppose that the nodes $1,\dots, \ell-1$ have already arrived, where $\ell\in \{1,\dots,n\}$. Then,
the next node, with label $\ell$, arrives at time $\ell$ and is attached to $\rt_j$ with probability $\theta_j/(\theta + \ell-1)$,  for each $j\in \{1,\dots, d\}$,  or to any of the already existing
nodes $1,\dots, \ell-1$, with probability $1/(\theta + \ell-1)$ for each node.
The usual Hoppe trees are recovered by taking $d=1$.

Recall that a random vector $(X_1,\dots,X_d)$ has the
\emph{Dirichlet multinomial distribution}, denoted hereafter by $\MDir{n}{\theta_1,\dots,\theta_d}$, if
\begin{multline}\label{eq:MDir_def}
\P[(X_1,\dots,X_d)=(k_1,\dots,k_d)]\\
=\frac{\Gamma(n+1)\Gamma(\theta_1+\dots+\theta_d)}{\Gamma(n+\theta_1+\dots+\theta_d)}
\prod_{i=1}^{d}\frac{\Gamma(k_i+\theta_i)}{\Gamma(\theta_i)\Gamma(k_i+1)},\quad k_i\geq 0,\quad \sum_{i=1}^{d} k_i=n,
\end{multline}
see Eq.~(35.151) in~\cite{Johnson+Kotz+Balakrishnan} with the convention that $\P[X_i=0]=1$ if $\theta_i=0$, for some $i\in\{1,\dots,d\}$.

Let $\Betadistr{\alpha}{\beta}$ denote the beta distribution with parameters $\alpha,\beta>0$ and having the density $x\mapsto x^{\alpha-1}(1-x)^{\beta-1}/\Beta(\alpha,\beta)$ for $x\in (0,1)$. If $\alpha=0$ (respectively $\beta=0$), we stipulate that $\Betadistr{\alpha}{\beta}$ is the degenerate distribution at $0$ (respectively, $1$). The following result is standard, see~\cite[Chapter 4.5.1]{Johnson+Kotz:UrnModels}.
\begin{proposition}\label{prop:multihoppe_sizes}
The joint distribution of the sizes of the connected components of $\rt_1,\dots, \rt_d$ in
$\MultiHoppe{n}{\theta}{\theta_1/\theta,\dots, \theta_d/\theta}$ (not counting the roots themselves) is $\MDir{n}{\theta_1,\dots,\theta_d}$. In particular, the size of the connected component of $\rt_j$ (not counting $\rt_j$ itself) has the so-called beta-binomial distribution $\Bin{n}{\Betadistr{\theta_j}{\theta-\theta_j}}$, for all $j=1,\dots,d$. 
\end{proposition}

\begin{proposition}\label{prop:root_degree_Hoppe}
For each 
$j\in \{1,\dots, d\}$, let $D_{n;j}$ be the degree of $\rt_j$ in the random forest $\MultiHoppe{n}{\theta}{\theta_1/\theta,\dots, \theta_d/\theta}$. Then,
$$
(D_{n;1},\dots, D_{n;d})~\sim~\MultStirone{n}{\theta}{\theta_1/\theta,\dots,\theta_d/\theta}.
$$
\end{proposition}
\begin{proof}
For $\ell\in \{1,\dots, n\}$ let  $I_\ell$ be a random vector in $\N_0^d$ with distribution~\eqref{eq:rep_multistir_indicator_distr} and assume that $I_1,\dots,I_n$ are mutually independent. We interpret the event $\{I_\ell=\eee_j\}$ as $\{$the node $\ell$ was attached to $\rt_j\}$ and the event $\{I_\ell=0\}$ as $\{\ell$ was attached to one of the nodes $1,\dots, \ell-1\}$. Then, $(D_{n;1},\dots, D_{n;d})$ has the same distribution as $I_1+\dots + I_n$. It remains to apply Proposition~\ref{prop:rep_mult_stir_sum_indicators}.
\end{proof}

\subsubsection{Expected profiles of multinomial Hoppe forests}
The expected number of nodes in an $\RRT{n}$ having distance $k$ to the root is known to be $\frac 1{n!}\stirling{n+1}{k+1}$, for all $k\in \{1,\dots, n\}$; see~\cite[Lemma 6.16]{Drmota_book}
or~\cite[Theorem~1]{smythe_mahmoud_survey_RRT} (actually both references
attribute this formula to~\cite{Dondajewski+Szymanski:1982}). The next theorem generalizes this result by identifying the expected profile of the multinomial Hoppe trees.
Interestingly, even for $d=1$ and arbitrary $\theta>0$, the next formula contains the $\theta$-Stirling numbers.
\begin{theorem}
Consider $\MultiHoppe{n}{\theta}{\theta_1/\theta,\dots, \theta_d/\theta}$ and let $L_{n; j}(k)$ be the number of nodes in the connected component of the root $j\in \{1,\dots, d\}$ which have distance $k$ to the root $j$. Then,
\begin{equation}\label{eq:rrrt_expected_profile}
\E  [L_{n;j}(k)] = \frac{\theta_j}{\theta^{n\uparrow}}\stirling{n}{k}_{\theta},\quad k\in \{1,\dots,n\},\quad j\in \{1,\dots, d\}.
\end{equation}
\end{theorem}
\begin{proof}
Denote the left-hand side of~\eqref{eq:rrrt_expected_profile} by $a_j(n,k):= \E  [L_{n;j}(k)]$. For $n\in\N_0$, let $\mathcal{F}_{n}$ be the $\sigma$-algebra generated by the multinomial Hoppe forest, which starts with $d$ roots, after adding $n$ new nodes. Then $a_j(0,k)=0$ for $k\in\N$ and, for $k\geq 2$,
\begin{multline}\label{eq:hoppe_forest_proof1}
a_j(n,k)=\E  [L_{n;j}(k)]=\E  [\E [L_{n;j}(k)]|\mathcal{F}_{n-1}]=\E \left[L_{n-1;j}(k)+\frac{L_{n-1;j}(k-1)}{n+\theta-1}\right]\\
=a_j(n-1,k)+\frac{a_j(n-1,k-1)}{n+\theta-1},\quad n\in\N,\quad j=1,\dots,d,
\end{multline}
where $1/(n+\theta-1)$ is the probability that
the node $n$ joins
a particular node at
the level $k-1$. For $k=1$, the recursion is slightly different accounting to a special weight of the root $j$:
\begin{equation}\label{eq:hoppe_forest_proof2}
a_j(n,1)=a_j(n-1,1)+\frac{\theta_j}{n+\theta-1},\quad n\in\N,\quad j=1,\dots,d,
\end{equation}
where $\theta_j/(n+\theta-1)$ is the probability of joining the root $j$. Introduce the generating functions
$$
\phi_{n;j}(z):=\theta_j+\sum_{k=1}^{\infty}a_j(n,k)z^k, \quad n\in\mathbb{N}_0, \quad  j=1,\dots,d.
$$
Then~\eqref{eq:hoppe_forest_proof1} and~\eqref{eq:hoppe_forest_proof2} imply
$$
\phi_{0;j}(z)=\theta_j,\quad \phi_{n;j}(z)=\left(1+\frac{z}{n+\theta-1}\right)\phi_{n-1;j}(z),\quad n\in\N.
$$
Thus,
$$
\phi_{n;j}(z)=\theta_j\prod_{k=0}^{n-1}\left(1+\frac{z}{k+\theta}\right)=\theta_j\frac{(\theta+z)^{n\uparrow}}{\theta^{n\uparrow}}=\frac{\theta_j}{\theta^{n\uparrow}}\sum_{k=0}^{n}\stirling{n}{k}_{\theta}z^k,\quad n\in\N_0,
$$
and~\eqref{eq:rrrt_expected_profile} follows.
\end{proof}
\begin{remark}
Observe that the expected number of nodes in the $j$-th tree (not counting the root) is
$$
\sum_{k=1}^n \E  [L_{n;j}(k)] = \frac{\theta_j}{\theta^{n\uparrow}}\sum_{k=1}^n\stirling{n}{k}_{\theta} \overset{\eqref{eq:stirling_r_rising_factorials}}{=} \frac{\theta_j}{\theta^{n\uparrow}} ((1+\theta)^{n\uparrow} - \theta^{n\uparrow}) = n\theta_j/\theta,
$$
as it should be. We have used here that $\stirling{n}{0}_\theta = \theta^{n\uparrow}$.
\end{remark}

\subsubsection{Leaves in multinomial Hoppe forests and generalized Eulerian numbers}
Let $T_n$ be the number of leaves in a random recursive tree $\RRT{n}$ with one root and $n$ other nodes. According to a result of Najock and Heyde~\cite{najock_heyde} (alternatively, see Eq.~(6.15) in~\cite{Drmota_book}):
\begin{equation}\label{eq:leaves_drmota}
\P[T_n=k]=\frac{1}{n!}\EulerNum{n}{k-1},\quad k=1,\dots,n, \quad n\in \N,
\end{equation}
where the second factor is the standard Eulerian number, see~\cite[Chapter 6.2]{graham_etal_book}
and~\cite{petersen_book_eulerian_numbers}. For our purposes the following recursive definition of the Eulerian number serves best:
\begin{equation}\label{eq:Euler_numbers_st_def}
\EulerNum{n}{k}=(n-k)\EulerNum{n-1}{k-1}+(k+1)\EulerNum{n-1}{k},\quad 0\leq k<n,\quad n\in\N
\end{equation}
with the boundary conditions $\EulerNum{n}{0}=1$, $n\in\N_0$, $\EulerNum{n}{n}=0$, $n\in\N$.  In the following we 
generalize~\eqref{eq:leaves_drmota} first to Hoppe trees and then to Hoppe forests. Put $A(r,s):=\EulerNum{r+s+1}{s}$ and note that~\eqref{eq:Euler_numbers_st_def} implies
$$
A(r,s)=(r+1)A(r,s-1)+(s+1)A(r-1,s),\quad r,s\in\N_0,\quad r+s>0,
$$
with $A(0,0)=1$ and $A(r,s)=0$ if $r=-1$ or $s=-1$. This formula suggests the following generalization of Eulerian numbers, proposed by Carlitz and Scoville~\cite{Carlitz+Scoville:1974}. Let $\alpha,\beta$ be arbitrary real parameters. Put
\begin{equation}\label{eq:Euler_numbers_gen_def}
A(r,s|\alpha,\beta)=(r+\beta)A(r,s-1|\alpha,\beta)+(s+\alpha)A(r-1,s|\alpha,\beta),\quad r,s\in\N_0,\quad r+s>0,
\end{equation}
and $A(0,0|\alpha,\beta)=1$ and $A(r,s|\alpha,\beta)=0$ if $r=-1$ or $s=-1$. Note that $A(r,s)=A(r,s|1,1)$. According to Eq.~(1.8) in~\cite{Carlitz+Scoville:1974},
$$
\sum_{r,s=0}^{\infty}A(r,s|\alpha,\beta)\frac{u^r v^s}{(r+s)!}=(1+uF(u,v))^{\alpha}(1+vF(u,v))^{\beta},
$$
where
$$
F(u,v)=\frac{\eee^{u}-\eee^{v}}{u\eee^{v}-v\eee^{u}}.
$$
It turns out that the generalized Eulerian numbers defined by~\eqref{eq:Euler_numbers_gen_def} are an indispensable ingredient of the version of~\eqref{eq:leaves_drmota} for Hoppe trees.

\begin{theorem}[Leaves in Hoppe trees]\label{theo:leaves_hoppe_tree}
Let $T_{n}$ be the number of leaves in a random tree $\Hoppe{n}{\theta}$ with $\theta>0$. Then,
\begin{equation}\label{eq:rrrt_leaves_distribution}
\P[T_{n}=k]= \frac{A(n-k,k|0,\theta)}{\theta^{n\uparrow}},\quad k=1,\dots,n, \quad n\in \N.
\end{equation}
\end{theorem}
\begin{proof}
Denote by $p_{n}(k):= \P[T_{n}=k]$ the left-hand side of~\eqref{eq:rrrt_leaves_distribution} and let $\mathcal{F}_{n}$ denote the $\sigma$-algebra generated by the Hoppe tree upon adding $n$ nodes. Note that the root cannot be a leaf if $n\geq 1$ and it is also convenient not to regard it as a leaf if $n=0$, thus we put $p_{0}(0)=1$. Then,
\begin{multline*}
p_{n}(k)=\E [\P[T_{n}=k|\mathcal{F}_{n-1}]]
=\frac{k}{n+\theta-1}p_{n-1}(k)+\frac{n-k+\theta}{n+\theta-1}p_{n-1}(k-1),\quad k=0,\dots,n,\quad n\in\N,
\end{multline*}
where we stipulate that $p_{n}(k)=0$, for $k>n$, and $p_{n}(-1)=0$. Put $Q_{n;\ell}:=\theta^{(n+\ell)\uparrow}p_{n+\ell}(\ell)$ and note that the above recursion transforms into
$$
Q_{n;\ell}=\ell Q_{n-1;\ell}+(n+\theta) Q_{n;\ell-1},\quad n\in\N_0,\quad \ell\in\N_0,\quad n+\ell>0
$$
with the initial values
$$
Q_{0,0}=1,\quad Q_{-1;\ell}=Q_{n;-1}=0,\quad n,\ell\in\N_0.
$$
Thus, $Q_{n;\ell}=A(n,\ell|0,\theta)$ and~\eqref{eq:rrrt_leaves_distribution} follows.
\end{proof}

\begin{example}
Note that the RRT result~\eqref{eq:leaves_drmota} is a particular case when $\theta=1$, since
$A(n-k,k|0,1)=\EulerNum{n}{k-1}$ which, in turn, is a consequence of $A(r,s|0,1)=\EulerNum{r+s}{s-1}$.
\end{example}

The generalized Eulerian distribution with probability mass function $k\mapsto A(k,n-k|\alpha, \beta)/(\alpha+\beta)^{n\uparrow}$ for $k\in \{0,\dots, n\}$ is investigated
in~\cite{charalambides_eulerian_generalized,Janardan:1988,janardan:1993}  in the context of Morisita's work on mutual repulsive behavior of ant lions~\cite{Morisita:1971}. Among other results, Charalambides~\cite{charalambides_eulerian_generalized} computes
the factorial moments of this distribution and proves
a central limit theorem.

\begin{remark}
The $r$-analogues of Eulerian numbers appearing in~\cite[\S~8.5--8.7]{mezo_book} and~\cite{mezo_recent_developments_stirling} and denoted there by $\EulerNum{n}{k}_{r}$ satisfy the same recurrence relation, namely, $\EulerNum{n}{k}_r = (n-k+r)\EulerNum{n-1}{k-1}_r + (k+1)\EulerNum{n-1}{k}_r$ (see relation~(8.30) on p.~218 in~\cite{mezo_book}) as the numbers  $A(n-k-1,k+1|0,r+1)$ (see~\eqref{eq:Euler_numbers_gen_def}). However, the boundary conditions imposed on these arrays are different; see~\cite[p.~216]{mezo_book}.  Related generalizations of Eulerian numbers can be found in~\cite{janson_euler_frobenius_rounding} and~\cite{maier_triangular_recurrences}.
\end{remark}
\begin{theorem}[Leaves in Hoppe forests]\label{theo:leaves_multihoppe}
Let $T_{n;j}$ be the number of leaves in the connected component of the root $j\in\{1,\dots, d\}$ in the multinomial Hoppe forest $\MultiHoppe{n}{\theta}{\theta_1/\theta,\dots, \theta_d/\theta}$ (if the root $j$ has no descendants, we stipulate that $T_{n,j} = 0$).  Then,
$$
\P[T_{n;j} = k] = \frac 1 {\theta^{n\uparrow}}\sum_{m=k}^n \binom nm (\theta - \theta_j)^{(n-m)\uparrow} A(m-k,k|0,\theta_j).
$$
\end{theorem}
\begin{proof}
The number of nodes in the connected component of $\rt_j$ (not counting the root itself) is $S_{n;j}\sim \Bin{n}{\Betadistr{\theta_j}{\theta-\theta_j}}$, see Proposition~\ref{prop:multihoppe_sizes}.  Conditionally on $S_{n;j} = m$, the component of $\rt_j$ is distributed as $\Hoppe{m}{\theta_j}$. By the total probability formula
and Theorem~\ref{theo:leaves_hoppe_tree}, the number of leaves in the connected component of $\rt_j$ satisfies
\begin{align*}
\P[T_{n;j} = k]
&=
\sum_{m=k}^n \P[S_{n;j}= m] \cdot  \frac{A(m-k,k|0,\theta_j)}{\theta_j^{m\uparrow}}\\
&=
\sum_{m=k}^n \binom nm \frac{\Gamma(m+\theta_j)\Gamma(n-m +\theta-\theta_j)\Gamma(\theta)}{\Gamma(n+\theta)\Gamma(\theta - \theta_j)\Gamma(\theta_j)} \cdot \frac {\Gamma(\theta_j)}{\Gamma(m+\theta_j)}A(m-k,k|0,\theta_j) \\
&=
\frac 1 {\theta^{n\uparrow}}\sum_{m=k}^n \binom nm (\theta - \theta_j)^{(n-m)\uparrow} A(m-k,k|0,\theta_j),
\end{align*}
which proves the claim.
\end{proof}

\begin{corollary}[Subtrees of the $\ell$-th node in a Hoppe tree]\label{ref:subtree_hoppe}
Consider a $\Hoppe{n}{\theta}$-tree and let $\ell\in \{1,\dots, n\}$ be a node. If $T_n^{(\ell)}$ denotes the number of leaves in the subtree rooted at the node $\ell$, then
$$
\P[T_n^{(\ell)} = k]
=
\frac 1 {(\theta+\ell)^{(n-\ell)\uparrow}}  \sum_{m=k}^{n-\ell} \binom {n-\ell} m  (\ell + \theta - 1)^{(n-\ell-m)\uparrow} \EulerNum{m}{k-1}
\qquad
k =1,\dots, n-\ell
$$
and
$\P[T_n^{(\ell)} = 0] = (\ell - 1 + \theta)/(n-1+\theta)$, which is the probability that the subtree only consists
of $\ell$.
\end{corollary}
\begin{proof}
Indeed, after the node $\ell$ has arrived, we can consider it as $\rt_1$ with weight $\theta_1 = 1$ and combine the nodes $1,\dots, \ell-1$ and the old root to form $\rt_2$ with weight $\theta_2 = \ell-1+\theta$. The remaining nodes $\ell+1,\dots, n$ behave as in the corresponding multinomial Hoppe tree with $d=2$. Applying Theorem~\ref{theo:leaves_multihoppe} and recalling that $A(m-k,k|0,1)= \EulerNum{m}{k-1}$ gives the claim.
\end{proof}

In the case of RRT (that is, for $\theta = 1$) our formula is equivalent to the formula of Mahmoud and Smythe~\cite[p.~410]{mahmoud_smythe}. Note, however, that they stipulate a tree consisting of a single root to have one
leaf, whereas we declare it to have no
leaves. Summarizing,
their formula gives the distribution of $\max(T_{n-1}^{(\ell)},1)$.

\section{\texorpdfstring{$r$}{r}-Ewens permutations, random \texorpdfstring{$r$}{r}-recursive trees}\label{sec:r_versions1}
Now we are going to define the $r$-analogues of the random structures introduced above. The general procedure is as follows. Fix parameters $\tau=\theta_1\geq 0$ and $r=\theta_2\geq 0$ such that $\theta=\tau+r>0$. Consider some multinomial random structure with $d=2$ colors having probabilities $p_1 := \tau/(\tau+r)$ and $p_2:=r/(\tau+r)$ and discard all elements having color $2$. The remaining elements of color $1$ form a random structure on some  subset $B\subseteq [n]$ (which is also random and, moreover, can be empty with positive probability). This random structure, defined on a random subset, is the $r$-analogue we are searching for.

\subsection{\texorpdfstring{$r$}{r}-Ewens permutations}

The $r$-Ewens distribution $\rEwens{n}{\tau}{r}$ is informally defined as follows. Consider a random permutation following the $\Ewens{n}{\tau+r}$ distribution. Independently remove each cycle of this permutation with probability $r/(\tau+r)$. The remaining cycles define a random permutation of some subset $B$ of $[n]$,  which we call the $r$-Ewens permutation and whose distribution is $\rEwens{n}{\tau}{r}$. We now give a 
precise definition.
\begin{definition}
An \emph{incomplete permutation} on $[n]$ is a pair $(\rho, B)$, where $B\subseteq [n]$ is a subset of $[n]$, which might be empty, and $\rho:B\to B$ is a bijection. The set $B$ is called the \emph{domain} of $\rho$ and its elements is
colored white. The complement of $B$ is
denoted by $B_0:= [n]\bsl B$ and its elements are
colored red.
\end{definition}
The set of incomplete permutations of $[n]$ is
denoted by $\mathfrak{S}_n^{\mathrm{inc}} = \{(\rho, B): B\subseteq [n], \rho \in \mathfrak{S}_B\}$, where $\mathfrak{S}_B$ is the symmetric group of $B$.

\begin{definition}\label{eq:r-permutation_distribution}
An \emph{$r$-Ewens permutation} with parameters $\tau\geq 0$ and $r\geq 0$ such that $\tau+r>0$ is a random variable $(\bmrho, \bmB)$ taking values in $\mathfrak{S}_n^{\mathrm{inc}}$ with the distribution
\begin{equation}\label{eq:rEwensDistr}
\P[(\bmrho, \bmB)=(\rho, B)] = r^{(n - \# B)\uparrow}\cdot  \frac{\tau^{\# \cycles (\rho)}}{(\tau + r)^{n\uparrow}},\quad
(\rho, B) \in\mathfrak{S}_n^{\mathrm{inc}},
\end{equation}
which is called the $\rEwens{n}{\tau}{r}$ distribution. If $\tau=1$, then we say that $(\bmrho, \bmB)$ is an \emph{$r$-uniform random permutation}.
\end{definition}

\begin{proposition}\label{prop:rEwens_andCRP}
Consider a multinomial Chinese restaurant process with $d=2$ colors having weights $\theta_1= \tau$ and $\theta_2=r$. Let $(\bmrho, \bm B)$ be an incomplete permutation whose cycles are the tables of color $1$ and whose domain $\bm B$ is the set of all customers sitting at tables of color $1$. Let $\bm B_0:= [n]\bsl \bmB$ be the set of all customers sitting at tables of color $2$. Then,  $(\bmrho, \bmB) \sim \rEwens{n}{\tau}{r}$.
\end{proposition}
\begin{proof}
Fix $(\rho, B) \in\mathfrak{S}_n^{\mathrm{inc}}$ and let $\Omega_{\rho,B}$ be the set of colored permutations of $[n]$ (with two colors) such that the elements belonging to cycles with color $2$ are precisely the elements of the set $B_0$ and cycles of color $1$ form permutation $\rho$ of $B$. According to Proposition~\ref{prop:multinomialCRP} we need to check that the sum of the right-hand sides of~\eqref{eq:multinomialEwensDistr} over the set $\Omega_{\rho,B}$ is equal to the right-hand side of~\eqref{eq:rEwensDistr}. Indeed,
$$
\sum_{(\rho^{\prime},\phi)\in\Omega_{\rho,B}}\frac{\tau^{\# \cycles_1(\rho^{\prime}, \phi)}r^{\# \cycles_2(\rho^{\prime},\phi)}}{(t+r)^{n\uparrow}}=\frac{\tau^{\# \cycles(\rho)}}{(t+r)^{n\uparrow}}r^{(n-\#B)\uparrow}\sum_{(\rho^{\prime},\phi)\in\Omega_{\rho,B}}\frac{r^{\# \cycles_2(\rho^{\prime},\phi)}}{r^{(n-\#B)\uparrow}}.
$$
The sum is equal to $1$, since the summands form a probability distribution $\Ewens{n-\#B}{r}$ and the summation is taken over all possible cyclic structures on $[n]\setminus B$.
\end{proof}

\begin{proposition}
Let $A_{n,r}(b_0,b_1,\dots,b_k)$ be the event that $(\bmrho, \bmB)\sim \rEwens{n}{\tau}{r}$ has a red set of size $b_0\in \N_0$ and $k$ white cycles of sizes $b_1,\dots,b_k\in \N$, where $b_0+b_1+\dots+ b_k = n$.  Then,
\begin{equation}\label{eq:r-permutation_cycles_distribution}
\P[A_{n,r}(b_0,b_1,\dots,b_k)]
=\frac{n!}{k!}\frac{\tau^k}{(\tau + r)^{n\uparrow}} \cdot \frac{r^{b_0\uparrow} }{b_0!} \cdot \frac{1}{b_1\cdots  b_k}.
\end{equation}
\end{proposition}
\begin{proof}
Equality~\eqref{eq:r-permutation_cycles_distribution} follows immediately from~\eqref{eq:rEwensDistr} upon noticing that $\binom{n}{b_0,\dots,b_k}\frac{1}{k!}$ counts the number of partitions of $[n]$ into blocks $B_0,B_1,\dots,B_k$ of sizes $b_0\in\N_0,b_1,\dots,b_k\in\N$, respectively, whereas $(b_j-1)!$ is the number of ways to organize a cycle on $b_j$ elements, $j=1,\dots,k$.
\end{proof}

\begin{proposition}\label{prop:rEwens_cycles}
If $(\bmrho, \bmB)$ is an $\rEwens{n}{\tau}{r}$-distributed incomplete permutation, then the number of cycles in $\bmrho$  has the $\rStirone{n}{\tau}{r}$ distribution, that is,
\begin{equation}\label{eq:r-permutation_number_of_cycles_distribution}
\P[\# \cycles (\bmrho) = k] = \stirling{n}{k}_r \frac{\tau^k}{(\tau+r)^{n\uparrow}},\qquad k\in \{0,\dots,n\}.
\end{equation}
\end{proposition}

Before proceeding with a proof of Proposition~\ref{prop:rEwens_cycles} we note that upon multiplying Taylor expansions~\eqref{eq:def_r_stir_first} and~\eqref{eq:def_r_stir_second}, one obtains
\begin{equation}\label{eq:stirling12_r_as_sum_1_over}
\stirling{n}{k}_r = \frac{n!}{k!}\sum_{(b_0,b_1,\dots, b_k)} \frac{1}{b_1\cdots b_k} \cdot \frac{r^{b_0\uparrow}}{b_0!},
\qquad
\stirlingsec{n}{k}_r = \frac{n!}{k!}\sum_{(b_0,b_1,\dots, b_k)} \frac{1}{b_1!\cdots b_k!} \cdot \frac{r^{b_0}}{b_0!},
\end{equation}
where the sums are taken over all $(k+1)$-tuples $(b_0,b_1,\dots, b_k)$ such that $b_0\in\N_0$, $b_1,\dots,b_k\in\N$ and $b_0+b_1+\dots+b_k=n$.

We 
give three elementary proofs of Proposition~\ref{prop:rEwens_cycles} with the aim of clarifying the probabilistic meaning of the first equality in~\eqref{eq:stirling12_r_as_sum_1_over} (the second will be explained later on, see Proposition~\ref{prop:rStirling-Sibuya}) and also of the formulas
\begin{equation}\label{eq:r-stirling-sum-of-stirling}
\stirling{n}{k}_r
=
\sum_{j=k}^n \stirling{n}{j} \binom{j}{k} r^{j-k}
=
\sum_{j=k}^n \binom{n}{j} \stirling{j}{k} r^{(n-j)\uparrow},
\end{equation}
see Eq.~\eqref{eq:stir1_as_poly} in the introduction.

\begin{proof}[First proof]
Formula~\eqref{eq:r-permutation_number_of_cycles_distribution} for the distribution of the number of cycles follows from the first equality in~\eqref{eq:stirling12_r_as_sum_1_over} upon summation of the right-hand sides of~\eqref{eq:r-permutation_cycles_distribution} over all $b_0\in\N_0$ and $b_1,\dots,b_k\in\N$ satisfying $b_0+\dots+b_k=n$.
\end{proof}

\begin{proof}[Second proof]
We use the first equality in~\eqref{eq:r-stirling-sum-of-stirling}.  Since the total number of cycles of both colors has the $\Stirone{n}{\tau+r}$ distribution and the probability for each particular cycle to be of color $1$ is $\tau/(\tau+r)$ independently of the other cycles,
$$
\P[\# \cycles (\bmrho) = k]=\Bin{\Stirone{n}{\tau+r}}{\tau/(\tau+r)}(\{k\}),\quad k\in \{0,\dots,n\}.
$$
Thus, we need to check that
\begin{equation}\label{eq:rEwens_is_mixture_of_Bin}
\rStirone{n}{\tau}{r}=\Bin{\Stirone{n}{\tau+r}}{\tau/(\tau+r)}.
\end{equation}
This equality follows from
\begin{multline*}
\Bin{\Stirone{n}{\tau+r}}{\tau/(\tau+r)}(\{k\})=\sum_{j=k}^{n}\Stirone{n}{\tau+r}(\{j\})\Bin{j}{\tau/(\tau+r)}(\{k\})\\
=\sum_{j=k}^{n}\stirling{n}{j}\frac{(\tau+r)^j}{(\tau+r)^{n\uparrow}}\binom{j}{k}\frac{\tau^k r^{j-k}}{(\tau+r)^j}\overset{\eqref{eq:r-stirling-sum-of-stirling}}{=}\stirling{n}{k}_r \frac{\tau^k}{(\tau+r)^{n\uparrow}}=\rStirone{n}{\tau}{r}(\{k\})
\end{multline*}
for each 
$k\in\{0,1,\dots,n\}$.
\end{proof}

\begin{proof}[Third proof]
In this proof we exploit the second equality in~\eqref{eq:r-stirling-sum-of-stirling}. Using definition~\eqref{eq:rEwensDistr} of the $r$-Ewens distribution we obtain
$$
\P[\# \cycles (\bmrho) = k]=\sum_{j=k}^{n}\P[\# \cycles (\bmrho) = k,\#B=j]\overset{\eqref{eq:rEwensDistr}}{=}\sum_{j=k}^{n}\frac{r^{(n-j)\uparrow}\tau^k}{(\tau+r)^{n\uparrow}}\binom{n}{j}\stirling{j}{k},\quad k\in \{0,\dots,n\},
$$
where $\binom{n}{j}$ counts the number of ways to choose $j$ elements of $[n]$ and $\stirling{j}{k}$ counts the number of ways to organize $k$ cycles on the chosen elements. By the second equality in~\eqref{eq:r-stirling-sum-of-stirling} the right-hand side of the last centered formula
is equal to $\rStirone{n}{\tau}{r}(\{k\})$.
\end{proof}

Similarly to
the classic scenario, the number of cycles in the $r$-Ewens permutation
$(\bmrho,{\bf B})$ can be represented as the sum of independent indicators. This follows from the factorization
$$
\E [s^{\# \cycles (\bmrho)}]=\frac{1}{(r+\tau)^{n\uparrow}}\sum_{k=0}^{n}\stirling{n}{k}_r (\tau s)^k=\frac{(r+\tau s)^{n\uparrow}}{(r+\tau)^{n\uparrow}}=\prod_{j=1}^{n}\left(\frac{r+j-1}{r+\tau+j-1}+\frac{\tau s}{r+\tau+j-1}\right).
$$
This representation immediately yields a central limit theorem for $\# \cycles (\bmrho)$ and, therefore, for the $r$-Stirling distribution of the first kind. Interestingly, the limit theorem does not depend on $r$ at all!

\begin{corollary}
Assume that $\tau,r\geq 0$, $\tau+r>0$ are fixed and let $(\bmrho,{\bf B})$ be the $r$-Ewens permutation. Then,
$$
\frac{\# \cycles (\bmrho)-\tau\log n}{\sqrt{\tau\log n}}~\todistr~\Normal{0}{1}.
$$
\end{corollary}

For the ease of reference we collect various representations for the $r$-Stirling distribution of the first kind derived above in a single
\begin{proposition}\label{prop:r_stirl_repres}
Fix $r\geq 0, \tau >0$. The $r$-Stirling distribution of the first kind $\rStirone{n}{\tau}{r}$ admits
the following representations:
\begin{itemize}
\item[(i)] $\sum_{j=1}^n \Bern{\tau/(\tau + r + j - 1)}\sim \rStirone{n}{\tau}{r}$, 
where the random variables on the left-hand side are independent;
\item[(ii)] $\rStirone{n}{\tau}{r}=\Bin{\Stirone{n}{\tau+r}}{\tau/(\tau+r)}$;
\item[(iii)] $\rStirone{n}{\tau}{r}=\Stirone{\Bin{n}{\Betadistr{\tau}{r}}}{\tau}$, where $\Betadistr{\tau}{r}$ is the beta-distribution.
\end{itemize}
\end{proposition}
\begin{proof}
Part~(i) can be verified by comparing generating functions. Alternatively, it is just the last claim of Proposition~\ref{prop:rep_mult_stir_sum_indicators} with $d=2$, $j=1$, $\theta_1 = \tau$, $\theta_2 = r$. Part (ii) is formula~\eqref{eq:rEwens_is_mixture_of_Bin}. For the proof of part (iii) note that, for $k\in\{0,1,\dots,n\}$,
\begin{align*}
\Stirone{\Bin{n}{\Betadistr{\tau}{r}}}{\tau}(\{k\})&=\sum_{m=0}^{n}\Bin{n}{\Betadistr{\tau}{r}}(\{m\})\Stirone{m}{\tau}(\{k\})\\
&=\sum_{m=0}^{n}\binom{n}{m}\frac{\Gamma(m+\tau)\Gamma(n-m+r)}{\Gamma(n+\tau+r)}\frac{\Gamma(\tau+r)}{\Gamma(\tau)\Gamma(r)}\stirling{m}{k}\frac{\tau^k}{\tau^{m\uparrow}}\\
&=\frac{\Gamma(\tau+r)}{\Gamma(n+\tau+r)}\tau^k\sum_{m=0}^{n}\binom{n}{m}\stirling{m}{k}r^{(n-m)\uparrow}\frac{\Gamma(m+\tau)}{\Gamma(\tau)\tau^{m\uparrow}}\\
&\overset{\eqref{eq:stir1_as_poly}}{=}\frac{\tau^k}{(\tau+r)^{n\uparrow}}\stirling{n}{k}_r=\rStirone{n}{\tau}{r}(\{k\}).
\end{align*}
The proof is complete.
\end{proof}

\begin{remark}
The recent article~\cite{huillet_moehle_bernoulli} investigates properties of the Bernoulli trials with unequal harmonic-like success probabilities as in part (i) of Proposition~\ref{prop:r_stirl_repres}. 
\end{remark}

We close the discussion of $r$-Ewens permutations by noting a combinatorial construction of the $r$-uniform permutations which works for integer $r\in\N_0$ and establishes connections to Broder's Definition~\ref{def:broder}. Call the elements $1,2,\dots,n$ of the set $[n+r]$ {\em white} and elements $n+1,\dots,n+r$ {\em red}. Let $\mathfrak{S}_{n,r}$ be the set of all permutations of $[n+r]$ such that red elements are in different cycles and let $\mathfrak{d}_{{\rm perm}}:\mathfrak{S}_{n,r}\mapsto \mathfrak{S}_n^{\mathrm{inc}}$ be the mapping defined as follows. For $\sigma\in \mathfrak{S}_{n,r}$, the image $\mathfrak{d}_{{\rm perm}}(\sigma)$ is an incomplete permutation $(\rho,B)\in\mathfrak{S}_n^{\mathrm{inc}}$ such that the red set $B_0=[n]\setminus B$ is obtained by merging $r$ cycles of $\sigma$ containing red elements, removing these red elements from the union, and keeping other cycles unchanged.

\begin{proposition}\label{prop:deletion_r_permutation}
Fix $n\in\N$, $r\in\N_0$. Let $\widehat{\sigma}_{n,r}$ be a permutation of $[n+r]$ picked uniformly at random from $\mathfrak{S}_{n,r}$. Then $\mathfrak{d}_{{\rm perm}}(\widehat{\sigma}_{n,r})$ is the $r$-uniform random permutation ($r$-Ewens permutation with $\tau=1$).
\end{proposition}
\begin{proof}
Each 
$\sigma\in \mathfrak{S}_{n,r}$ can be constructed from an incomplete permutation $(\sigma_1,B)\in\mathfrak{S}_n^{{\rm inc}}$ of $[n]$, by distributing $n-\# B$ elements of the block $[n]\setminus B$ among $r$ cycles of $\sigma$ containing $n+1,\dots,n+r$ red elements. Thus,
\begin{multline*}
\#\mathfrak{S}_{n,r}=\sum_{j=0}^{n}\#\{(\sigma_1,B)\in\mathfrak{S}_n^{{\rm inc}}:\# B=j\}r^{(n-j)\uparrow}\\
=\sum_{j=0}^{n}\binom{n}{j} j! r^{(n-j)\uparrow}=\sum_{j=0}^{n}\binom{n}{j} 1^{j\uparrow } r^{(n-j)\uparrow}=(1+r)^{n\uparrow},
\end{multline*}
where the last equality follows from the binomial theorem for raising factorials. By the same reasoning, for each 
incomplete permutation $(\sigma_1,B)\in\mathfrak{S}_n^{{\rm inc}}$,
$$
\#\{\sigma\in \mathfrak{S}_{n,r}:\mathfrak{d}_{{\rm perm}}(\sigma)=(\sigma_1,B)\}=r^{(n-\#B)\uparrow}.
$$
The proof concludes by taking the ratio.
\end{proof}

\subsection{Random \texorpdfstring{$r$}{r}-recursive trees and \texorpdfstring{$r$}{r}-Hoppe trees}
The construction is very similar to the construction of the $r$-Ewens permutations.

\begin{definition}
Fix $n\in\N_0$. An \emph{incomplete recursive tree} with the root $0$ and $n$ additional nodes is a pair $(T,B)$, where $B\subseteq [n]$ and $T$ is a recursive tree with the root $0$ and $\#B$ nodes labeled by the elements of $B$ and colored white. The elements of $[n]\setminus B$ are called red nodes  (these do not belong to the tree).
\end{definition}

Denote by $\mathcal{T}_n$ the set of all incomplete recursive trees with the root $0$ and $n$ additional nodes.

\begin{definition}
Fix parameters $r\geq 0$ and $\tau>0$.
A \emph{random $r$-Hoppe tree} $\rHoppe{n}{\tau}{r}$ with  the root $0$ and $n$ additional nodes is a random element $({\bf  T},{\bf B})$ of $\mathcal{T}_n$ with the distribution
$$
\P[({\bf  T},{\bf B})=(T,B)]=\frac{r^{(n-\# B)\uparrow}\tau^{\deg_T 0}}{(r+\tau)^{n\uparrow}},\quad (T,B)\in\mathcal{T}_n,
$$
where $\deg_T 0$ is the degree of the root $0$ in $T$.
A \emph{random $r$-recursive tree} is a special case corresponding to $\tau=1$.
\end{definition}


Note that this definition is correct in a sense that the quantities on the right-hand side sum up to one. Indeed,
\begin{align*}
\sum_{(T,B)\in\mathcal{T}_n}\frac{r^{(n-\#B)\uparrow}\tau^{\deg 0}}{(r+\tau)^{n\uparrow}}
&=
\sum_{j=0}^{n}\sum_{k=0}^j \frac{r^{(n-j)\uparrow}\tau^{k}}{(r+\tau)^{n\uparrow}}
\sum_{\substack{(T,B)\in\mathcal{T}_n\\ \#B=j,\; \deg 0 =k}} 1
=
\frac{\sum_{j=0}^{n} r^{(n-j)\uparrow} \binom{n}{j} \sum_{k=0}^j \stirling jk \tau^k}{(r+\tau)^{n\uparrow}}
\\
&=
\frac{\sum_{j=0}^{n} \binom{n}{j} r^{(n-j)\uparrow} \tau^{j\uparrow} }{(r+\tau)^{n\uparrow}}
=1,
\end{align*}
where the penultimate equality is a consequence of the fact that there are $\binom{n}{j}$ ways to pick $j$ white nodes and $\stirling{j}{k}$ ways to construct a recursive tree with the root $0$ and $j$ additional white nodes such that the $\deg 0 = k$. The last equality follows from the binomial-type formula for the rising factorial.

The next proposition provides a sequential construction of the $r$-Hoppe trees.
\begin{proposition}\label{prop:r-recursive-tree-Hoppe}
Consider a Hoppe forest $\MultiHoppe{n}{\tau+r}{\tau/(\tau+r),r/(\tau+r)}$ with $d=2$ roots $\rt_1:=0$ and $\rt_2$ having weights $\theta_1=\tau>0$ and $\theta_2=r\geq 0$, and with $n$ additional nodes. Let ${\bf T}$ be the connected component of $\rt_1$ and ${\bf B}$ 
the set of labels of the nodes of ${\bf T}$. Then, $({\bf T},{\bf B})$ is distributed as a random tree $\rHoppe{n}{\tau}{r}$.
\end{proposition}
\begin{proof}
Let $(T,B)$ be a fixed incomplete recursive tree. The event $\{({\bf T},{\bf B})=(T,B)\}$ occurs if, and only if, the nodes with labels in $[n]\setminus B$ were attached to the subtree rooted at $\rt_2$ and the nodes of $B$ formed a particular tree $T$ rooted at $\rt_1$.  By definition of the Hoppe forest, the probability of the former event is
$$
r^{(n-\# B)\uparrow}\prod_{j\in[n]\setminus B}\frac{1}{j-1 + \tau + r},
$$
whereas the probability of the latter (given the former) is
$$
\tau^{\deg_T 0} \prod_{j\in B}\frac{1}{j-1 + \tau + r}.
$$
Thus, $\P[({\bf T},{\bf B})=(T,B)]= r^{(n-\# B)\uparrow}\tau^{\deg_T 0}/(r+\tau)^{n\uparrow}$. The proof is complete.
\end{proof}

Observe that if $r=0$, then in the setting of Proposition~\ref{prop:r-recursive-tree-Hoppe} no nodes are attached to $\rt_2$ (that is, $\P[\#{\bf B}=n]=1$) and  ${\bf T}$ becomes the usual $\Hoppe{n}{\tau}$ tree. If, additionally, $\tau=1$, we recover $\RRT{n}$.


The following is a specialization of the results on multinomial Hoppe forests proved in Section~\ref{sec:multi_hoppe}.

\begin{proposition}
Consider a random $r$-Hoppe tree $({\bf T},{\bf B})\sim \rHoppe{n}{\tau}{r}$.
\begin{itemize}
\item[(a)] The root degree  of ${\bf T}$ satisfies $\deg \rt_1 \sim \rStirone{n}{\tau}{r}$, that is, $\P[\deg \rt_1=k]=\frac{\tau^k}{(\tau +r)^{n\uparrow}}\stirling{n}{k}_r$ for all $k=0,\dots,n$.
\item[(b)] The expected number of nodes in ${\bf T}$ at distance $k$ from the $\rt_1$ is $\tau \stirling{n}{k}_{\tau + r} / (\tau + r)^{n\uparrow}$.
\item[(c)] The number of leaves of $\bf T$, denoted by $T_{n;1}$ has the following distribution:
$$
\P[T_{n;1}=k] = \frac{1}{(\tau + r)^{n\uparrow}}\sum_{m=k}^n \binom{n}{m} r^{(n-m)\uparrow} A(m-k,k|0,\tau),\quad k=1,\dots,n.
$$
\end{itemize}
\end{proposition}

\section{Random \texorpdfstring{$r$}{r}-partitions}\label{sec:r_versions2}



Recall that a \emph{partition} of $[n]$ is a collection $B_1,B_2,\dots,B_k$ of pairwise disjoint, nonempty subsets $B_j\subseteq [n]$ called blocks, whose union is $[n]$. The number $k$ of blocks may vary from $1$ for the trivial one-block partition, to $n$ for the finest partition into singletons. Let $\Pi_n$ be the set of all partitions of $[n]$.

A natural way to generate a partition of $[n]$ is via an urn scheme. Let $N\in\N$ be a fixed integer. Drop $n$ balls labeled by the elements of $[n]$ into $N$ urns uniformly at random and independently. Define a random partition ${\bm \lambda}_{n,N}$ of $[n]$ by declaring two balls to be in the same block if, and only if, they fall into the same urn. As has already been mentioned in the introduction, the number $\#\blocks({\bm \lambda}_{n,N})$ of blocks in the resulting partition ${\bm \lambda}_{n,N}$ has the Stirling--Sibuya  distribution $\StirSibuya{n}{N}$ defined by Eq.~\eqref{def:stirling-sibuya}:
\begin{equation}\label{eq:stirling_sibuya_section_r}
\P[\#\blocks({\bm \lambda}_{n,N})=k]=\stirlingsec{n}{k}\frac{N^{k\downarrow}}{N^n}=\StirSibuya{n}{N}(\{k\}),\quad k\in\{1,\dots,\min(n,N)\}.
\end{equation}

\subsection{Incomplete partitions derived from urn schemes}
Next, we 
introduce an $r$-version of~\eqref{eq:stirling_sibuya_section_r}. To this end, it is necessary to replace partitions by incomplete partitions.
\begin{definition}\label{def:B-partitions}
An \emph{incomplete partition} of $[n]$ is a collection $(B_0,\{B_1,B_2,\dots,B_k\})$ of pairwise disjoint subsets $B_j\subseteq [n]$, whose union is $[n]$ and $B_j\neq \varnothing$ for all $j\in \{1,\dots,k\}$. Note that $B_0$, called the \emph{red block}, is allowed to be empty. The blocks $B_1,\dots, B_k$, called the \emph{white blocks}, are non-empty. Let $\Pi_n^{\mathrm{inc}}$ be the set of all incomplete partitions of $[n]$.
\end{definition}

The already constructed $r$-versions of Ewens permutations and random recursive trees suggest the following way to generate a random incomplete partition leading to an $\rStirSibuya{n}{N}{r}$-distribution. We consider an urn scheme with $N+1$ urns $0,1,\dots,N$. The urn $0$ is red and has frequency $r/(N+r)\geq 0$, whereas urns $1,2,\dots,N$ are white and have frequencies $1/(N+r)$. Drop $n$ balls into the urns and define an incomplete partition ${\bm \lambda}_{n,N,r}$ of $[n]$ by declaring the balls in the urn $0$ to form the red block $B_0$ whereas the balls in those urns $1,\dots,N$ that are non-empty to form white blocks as before.
\begin{proposition}\label{prop:incomplete_partition_urn_scheme}
Let $N\in\N$ be an integer and $r\geq 0$. Then,
\begin{equation}\label{eq:r-partition-distribution}
\P[{\bm \lambda}_{n,N,r}=(B_0,\{B_1,\dots,B_k\})]=\frac{r^{\# B_0}N^{k\downarrow}}{(N+r)^n},\quad (B_0,\{B_1,\dots,B_k\})\in \Pi_n^{\mathrm{inc}}.
\end{equation}
\end{proposition}
\begin{proof}
The formula follows upon noticing that
$$
\frac{r^{\# B_0}N^{k\downarrow}}{(N+r)^n}=\left(\frac{r}{N+r}\right)^{\# B_0}\left(\frac{N^{k\downarrow}}{(N+r)^{n-\# B_0}}\right).
$$
The first factor on the right-hand side is the probability for balls with labels in $B_0$ to fall in the urn $0$. The second factor is the probability for remaining balls to form a partition $B_1,\dots,B_k$ of $[n]\setminus B_0$.
\end{proof}

\begin{remark}
For $r=0$, the right-hand side of~\eqref{eq:r-partition-distribution} vanishes whenever $\# B_0\neq 0$. Thus, neglecting the (empty) red block in ${\bm \lambda}_{n,N,0}$ we obtain the usual partition ${\bm \lambda}_{n,N}$ of $[n]$.
\end{remark}

Let $\type$ be the mapping which sends an incomplete partition $\pi\in\Pi_n^{\mathrm{inc}}$ to the multiset $\mathsf{type}(\pi)$ of block sizes of $\pi$ with a distinguished part $b_0\in\N_0$. Formally, an incomplete partition $\pi=(B_0,\{B_1,\dots,B_k\})$ has type $\mathsf{type}(\pi)=(b_0,\{b_1,\dots,b_k\})$, where $b_0\in\N_0$ and $b_1,\dots,b_k\in\N$ if, and only if, $\# B_0=b_0$ and $\{b_1,\dots,b_k\}$ is the multiset of block sizes of $B_1,\dots,B_k$.

From Proposition~\ref{prop:incomplete_partition_urn_scheme} we immediately conclude the following.

\begin{corollary}\label{cor:incomplete_partition_urn_scheme_type}
Let $N\in\N$ be an integer and $r\geq 0$. Then,
\begin{equation}\label{eq:r-partition_type_distribution}
\P[\type({\bm \lambda}_{n,N,r})=(b_0,\{b_1,\dots,b_k\})]=\binom{n}{b_0,\dots,b_k}\frac{1}{k!}\frac{r^{b_0}N^{k\downarrow}}{(N+r)^n}
\end{equation}
for each 
collection $(b_0,\{b_1,\dots, b_k\})$ of integers such that $b_0\in\N_0$, $b_1,\dots,b_k\in\N$ and $b_0+b_1+\dots+b_k=n$.
\end{corollary}

For an incomplete partition $\pi\in\Pi_n^{\mathrm{inc}}$ let $\#\blocks(\pi)$ denote the number of {\em white} blocks.

\begin{proposition}\label{prop:rStirling-Sibuya}
Let $N,n\in\N$ be integers and $r\geq 0$. The number of white blocks in ${\bm \lambda}_{n,N,r}$ has the $\rStirSibuya{n}{N}{r}$ distribution, that is,
\begin{equation}\label{eq:r-partition-sibuya}
\P[\# \blocks ({\bm \lambda}_{n,N,r}) = k] = \stirlingsec{n}{k}_r \frac{N^{k\downarrow}}{(N+r)^{n}},\qquad k\in \{0,1,\dots,\min(n,N)\}.
\end{equation}
\end{proposition}
Similarly to Proposition~\ref{prop:rEwens_cycles} we 
give three simple proofs of this assertion exploiting various representations of $\stirlingsec{n}{k}_r$.
\begin{proof}[First proof]
This proof uses the second equality in~\eqref{eq:stirling12_r_as_sum_1_over}. One just needs to sum the right-hand sides of~\eqref{eq:r-partition_type_distribution} over all $(k+1)$-tuples $(b_0,b_1,\dots, b_k)$ such that $b_0\in\N_0$, $b_1,\dots,b_k\in\N$ and $b_0+b_1+\dots+b_k=n$. Formula~\eqref{eq:r-partition-sibuya} follows then from~\eqref{eq:stirling12_r_as_sum_1_over}.
\end{proof}
\begin{proof}[Second proof]
This proof hinges on an equality in~\eqref{eq:stir2_as_poly}, namely, $\stirlingsec{n}{k}_r=\sum_{j=k}^n \binom{n}{j} \stirlingsec{j}{k} r^{n-j}$. Conditioning on the number of balls that do not fall in the urn $0$ and using~\eqref{eq:stirling_sibuya_section_r}, we obtain
\begin{multline*}
\P[\# \blocks ({\bm \lambda}_{n,N,r}) = k]=\StirSibuya{\Bin{n}{{N/(N+r)}}}{N}(\{k\})\\
=\sum_{j=k}^{n}\binom{n}{j}\left(\frac{r}{N+r}\right)^{n-j}\left(\frac{N}{N+r}\right)^{j}\stirlingsec{j}{k}\frac{N^{k\downarrow}}{N^j}
\overset{\eqref{eq:stir2_as_poly}}{=}\stirlingsec{n}{k}_r \frac{N^{k\downarrow}}{(N+r)^{n}}
\end{multline*}
for each 
$k\in\{0,1,\dots,\min(n,N)\}$.
\end{proof}
\begin{proof}[Third proof]
In this proof we suppose that $r\in \N$ is integer and exploit the formula  $\stirlingsec{n}{k}_r=\sum_{j=k}^{n}\stirlingsec{n}{j}\binom{j}{k}r^{(j-k)\downarrow}$, see~\eqref{eq:stir2_as_poly}.
Replace the red urn with frequency $r/(N+r)$ by $r$ urns with frequency $1/(N+r)$, so that we allocate $n$ balls in $N+r$ equiprobable urns ($N$ white ones and $r$ red ones). In total, there are $(N+r)^n$ possible allocations. Then,
\begin{align*}
\P[\# \blocks ({\bm \lambda}_{n,N,r}) = k]
&=
\sum_{j=k}^n \P[k \text{ white urns and }  j-k \text{ red urns are non-empty}]
\\
&=
\frac{1}{(N+r)^n}\sum_{j=k}^n \stirlingsec{n}{j} \binom jk N^{k\downarrow} r^{(j-k)\downarrow}
=
\stirlingsec{n}{k}_r \frac{N^{k\downarrow}}{(N+r)^{n}},
\end{align*}
where $\stirlingsec{n}{j}$ is the number of ways to decompose $n$ balls into $j$ blocks, $\binom jk$ is the number of ways to choose $k$ blocks to be placed into white urns, and, finally, $N^{k\downarrow}$, respectively $r^{(j-k)\downarrow}$,  is the number of ways to choose white, respectively red, urns to be filled.
\end{proof}
\begin{remark}
Proposition~\ref{prop:rStirling-Sibuya} can also be found 
in~\cite[Section 5]{nishimura_sibuya}.
\end{remark}

\subsection{Random uniform \texorpdfstring{$r$}{r}-partitions and random Gibbs \texorpdfstring{$r$}{r}-partitions}

The cardinality of the set $\Pi_n$ of all partitions of $[n]$ is the Bell number $\Bell_n$. Let ${\bm \pi}$ be a random uniform  partition of $[n]$, that is, a random element with values in $\Pi_n$ and such that $\P[{\bm \pi} = \pi] = 1/\Bell_n$ for each 
partition $\pi\in \Pi_n$. Stam~\cite{stam_random_partition} discovered a randomization of the number of urns $N$ in the above scheme which leads to a random partition with the uniform distribution on $\Pi_n$. Let $M$ be a random variable with
$$
\P[M=m] = \frac{1}{\eee \cdot \Bell_n} \frac{m^n}{m!}, \qquad m\in \N.
$$
The numbers on the right-hand side sum up to $1$ by the Dobi\'nski formula. Stam showed that the distribution of ${\bm \lambda}_{n,M}$ is uniform on $\Pi_n$, where $M$ is assumed independent of everything else. The proof is rather simple. From~\eqref{eq:r-partition-distribution} applied with $r=0$ we conclude that, for each 
$\{B_1,\dots,B_k\}\in\Pi_n$,
$$
\P[{\bm\lambda}_{n,M}=\{B_1,\dots,B_k\}]=\frac{1}{\eee \cdot \Bell_n}\sum_{m=k}^{\infty} \frac{m^n}{m!}\frac{m^{k\downarrow}}{m^n}=\frac{1}{\eee \cdot \Bell_n}\sum_{m=k}^{\infty} \frac{1}{(m-k)!}=\frac{1}{\Bell_n}.
$$
Moreover, Stam showed that the number of empty boxes is Poisson distributed with parameter $1$ and 
independent of the partition.

We now focus on a two-fold generalization of Stam's construction. This leads 
us to the $r$-Stirling distributions of the second kind $\rStirtwo{n}{\theta}{r}$.

\begin{definition}\label{def:Gibbs-partitions}
Fix some $\theta>0$ and $r\geq 0$. A random variable ${\bm \pi}_{\theta,r}$ with values in $\Pi_n^{\mathrm{inc}}$ is called a {\em random Gibbs $r$-partition} if it has the distribution
$$
\P[{\bm \pi}_{\theta,r}=(B_0,\{B_1,\dots, B_k\})]=\frac{\theta^k r^{\# B_0}}{T_{n,r}(\theta)},\quad (B_0,\{B_1,\dots,B_k\})\in \Pi_n^{\mathrm{inc}},
$$
where we recall the definition
$$
T_{n,r}(\theta) = \sum_{k=0}^{n}\stirlingsec{n}{k}_r \theta^k
$$
of the $r$-Touchard polynomial. This distribution is denoted by $\GibbsPart{n}{\theta}{r}$. If $\theta=1$, an incomplete partition ${\bm \pi}_{1,r}$ is called {\em uniform $r$-partition}. The number $T_{n,r}(1)=:\Bell_{n,r}$ is called the $n$-th \emph{$r$-Bell number}, see~\cite[\S~8.3]{mezo_book} and~\cite{mezoe_r_bell}.
\end{definition}
\begin{remark}
In the literature, see, for example,~\cite{Pitman_book} and~\cite{Gnedin+Pitman:2005}, it is more common to use the term `Gibbs partition' in a more general sense when the weight of a block also depends on its size. In Definition~\ref{def:Gibbs-partitions} we use this terminology in a narrow sense, when all white blocks have the same (but arbitrary) weight $\theta>0$.
\end{remark}

Note that the above distribution resembles~\eqref{eq:r-partition-distribution} with $N$ replaced by $\theta$. The change of the factor $\theta^{k\downarrow}$ to $\theta^k$ results in a more sophisticated normalizing factor $T_{n,r}(\theta)$ in place of $(\theta+r)^n$ in~\eqref{eq:r-partition-distribution}.

The correctness of the definition follows from the chain of equalities (there are $\binom{n}{j}$ choices for a red block $B_0$ containing $n-j$ elements and $\stirlingsec{n}{k}$ choices to partition the remaining $j$ elements into $k$ white blocks $B_1,\dots, B_k$):
\begin{align*}
\sum_{\pi\in \Pi_n^{\mathrm{inc}}} \P[{\bm \pi}_{\theta,r}=\pi]
&=
\frac{1}{T_{n,r}(\theta)} \sum_{j=0}^{n}\sum_{k=0}^{j}  \theta^k  r^{n-j}\binom{n}{j}\stirlingsec{j}{k}
=
\frac{1}{T_{n,r}(\theta)} \sum_{k=0}^{n} \theta^k \sum_{j=k}^{n} r^{n-j} \binom{n}{j}\stirlingsec{j}{k}
\\
&\overset{\eqref{eq:stir2_as_poly}}{=}
\frac{1}{T_{n,r}(\theta)} \sum_{k=0}^{n}\stirlingsec{n}{k}_r \theta^k = 1.
\end{align*}


\begin{proposition}
Let ${\bm \pi}_{\theta,r}\sim \GibbsPart{n}{\theta}{r}$. Then
\begin{equation}\label{eq:r-gibbs-partition_type_distribution}
\P[\type({\bm \pi}_{\theta,r})=(b_0,\{b_1,\dots,b_k\})]
=
\frac{r^{b_0}\theta^k}{T_{n,r}(\theta)}\binom{n}{b_0,\dots,b_k}\frac{1}{k!}
=
\frac{n!}{k!}\frac{r^{b_0}\theta^k}{b_0!\cdots b_k!}\frac{1}{T_{n,r}(\theta)}.
\end{equation}
Furthermore, the number of white blocks follows the $r$-Stirling distribution of the second kind:
\begin{equation}\label{eq:r-partition_number_blocks_distribution}
\P[\#\blocks({\bm \pi}_{\theta,r}) = k]=\frac{\theta^k}{T_{n,r}(\theta)}\stirlingsec{n}{k}_r,\quad k=0,\dots,n.
\end{equation}
\end{proposition}
\begin{proof}
Equality~\eqref{eq:r-gibbs-partition_type_distribution} is obvious. Formula~\eqref{eq:r-partition_number_blocks_distribution} follows from the second equality in~\eqref{eq:stirling12_r_as_sum_1_over} upon summing the right-hand side 
of~\eqref{eq:r-partition_type_distribution} over all $b_0\in\N_0$ and $b_1,\dots,b_k\in\N$ satisfying $b_0+\dots+b_k=n$.
\end{proof}

Next, in the spirit of Stam's construction, we show that ${\bm \pi}_{\theta,r}$ can be realized by randomizing $N$ in the incomplete partition ${\bm \lambda}_{n,N,r}$ defined at 
the beginning of this section.
Consider a random variable $M_{\theta,r}$ with the following distribution:
\begin{equation}\label{eq:r_dobinski_distr_for_stam}
\P[M_{\theta,r} = m] = \frac{1}{\eee^{\theta} T_{n,r}(\theta)}\frac{(r+m)^n \theta^m}{m!},
\qquad
m\in \N_0.
\end{equation}

We first check that this is indeed a probability distribution. The next lemma is an $r$-analogue of the Dobi\'nski formula, see~\cite[Theorem~5.1]{mezoe_r_bell}. For completeness, we provide a proof.

\begin{lemma}
For each 
$\theta>0$ and $r\geq 0$,
$$
\sum_{m=0}^{\infty}\frac{(r+m)^n \theta^m}{m!}=\eee^{\theta} T_{n,r}(\theta).
$$
\end{lemma}
\begin{proof}
Using the second formula in~\eqref{eq:stirling_r_rising_factorials} yields
\begin{multline*}
\sum_{m=0}^{\infty}\frac{(r+m)^n \theta^m}{m!}=\sum_{m=0}^{\infty}\frac{\theta^m}{m!}\sum_{k=0}^{n}\stirlingsec{n}{k}_r m^{k\downarrow}=\sum_{k=0}^{n}\stirlingsec{n}{k}_r \sum_{m=0}^{\infty}\frac{\theta^m}{m!}m^{k\downarrow}\\
=\sum_{k=0}^{n}\stirlingsec{n}{k}_r \theta^k \sum_{m=k}^{\infty}\frac{\theta^{m-k}}{(m-k)!}=\left(\sum_{k=0}^{n}\stirlingsec{n}{k}_r \theta^k\right) \left(\sum_{m=0}^{\infty}\frac{\theta^{m}}{m!}\right)=T_{n,r}(\theta)\eee^{\theta} .
\end{multline*}
\end{proof}

Here is a version of Stam's result including the claim that the number of empty urns is independent of the incomplete partition and has a Poisson distribution. Naturally, $M_{\theta,r}$ is taken independent of everything else.

\begin{proposition}
For each 
$\theta>0$ and $r\geq 0$,
$$
\left({\bm \lambda}_{n,M_{\theta,r},r},M_{\theta,r}-\#\blocks({\bm \lambda}_{n,M_{\theta,r},r})\right)~\sim~\GibbsPart{n}{\theta}{r}\otimes \Poi{\theta}.
$$
\end{proposition}
\begin{proof}
Fix $v\in\N_0$ and an incomplete partition $(B_0,\{B_1 \dots, B_k\})\in\Pi_n^{{\rm inc}}$. Then, by~\eqref{eq:r_dobinski_distr_for_stam} and Proposition~\ref{prop:incomplete_partition_urn_scheme},
\begin{align*}
\lefteqn{
\P[{\bm \lambda}_{n,M_{\theta,r},r} = (B_0,\{B_1 \dots, B_k\}), M_{\theta,r} = k+v]
}\\
&=
\P[M_{\theta,r} = k+v] \cdot \P[{\bm \lambda}_{n,M_{\theta,r},r} = (B_0,\{B_1 \dots, B_k\})|M_{\theta,r} = k+v]\\
&=
\P[M_{\theta,r} = k+v] \cdot \P[{\bm \lambda}_{n,k+v,r} = (B_0,\{B_1 \dots, B_k\})]\\
&=
\frac  1 {\eee^{\theta} T_{n,r}(\theta)}\frac{(r+k+v)^n \theta^{k+v}}{(k+v)!} \cdot
\frac{r^{\# B_0}(k+v)^{k\downarrow}}{(k+v+r)^{n}}\\
&=
\frac{\theta^k r^{\# B_0}}{T_{n,r}(\theta)} \cdot \eee^{-\theta}\frac{\theta^{v}}{v!}\\
&=
\GibbsPart{n}{\theta}{r}((B_0,\{B_1 \dots, B_k\}))\cdot \Poi{\theta}(\{v\}),
\end{align*}
where in the last line we have used Definition~\ref{def:Gibbs-partitions}.
\end{proof}

As we did for $r$-permutations, we conclude the discussion of $r$-partitions with a remark on Broder's definition of the $r$-Stirling numbers of the second kind given in the introduction. If $r\in\N_0$ is an integer, then the $r$-uniform partition can be constructed combinatorially as follows. Consider the set $[n+r]$ in which elements $1,2,\dots,n$ are {\em white} and $n+1,\dots,n+r$ are {\em red}. Let $\Pi_{n,r}$ be the set of all partitions of $[n+r]$ such that red elements are in different blocks and let $\mathfrak{d}_{{\rm part}}:\Pi_{n,r}\mapsto \Pi_n^{{\rm inc}}$ be the mapping defined as follows. For $\pi\in \Pi_{n,r}$, the image $\mathfrak{d}_{{\rm part}}(\pi)$ is an incomplete partition $(\widehat{B}_0,\{\widehat{B}_1,\dots,\widehat{B}_k\})$ of $[n]$ such that the red block $\widehat{B}_0$ is obtained by merging $r$ blocks of $\pi$ containing red elements, removing these red elements from the union, and keeping white blocks $\widehat{B}_1,\dots,\widehat{B}_k$ unchanged. Let $\widehat{\pi}_{n,r}$ be a partition of $[n+r]$ picked uniformly at random from $\Pi_{n,r}$.
\begin{proposition}\label{prop:deletion_r_partition}
For each 
$n\in\N$ and $r\in\N_0$,
$$
\mathfrak{d}_{{\rm part}}(\widehat{\pi}_{n,r})~\sim~\GibbsPart{n}{1}{r},
$$
that is, $\mathfrak{d}_{{\rm part}}(\widehat{\pi}_{n,r})$ is the uniform $r$-partition.
\end{proposition}
\begin{proof}
Each 
$\pi\in \Pi_{n,r}$ can be constructed from an incomplete partition $\pi_1:=(B_0,\{B_1,\dots,B_k\})\in\Pi_n^{{\rm inc}}$ of $[n]$ by distributing $b_0=\# B_0$ elements of the block $B_0$ among $r$ blocks of $\pi$ containing $r$ red elements (there is an $r$-fold choice, for each $x\in B_0$). Thus,
$$
\#\Pi_{n,r}=\sum_{j=0}^{n}\#\{\pi_1\in\Pi_n^{{\rm inc}}:\# B_0=j\}r^j=\sum_{j=0}^{n}\binom{n}{j}\Bell_{n-j} r^j=\Bell_{n,r}=T_{n,r}(1),
$$
where the penultimate inequality follows from~\eqref{eq:stir2_as_poly}. Furthermore, for each 
incomplete partition $\pi_1\in \Pi_n^{{\rm inc}}$ with red block $B_0$ of size $b_0$,
$$
\#\{\pi\in \Pi_{n,r}:\mathfrak{d}_{{\rm part}}(\pi)=\pi_1\}=r^{b_0}.
$$
The proof concludes by taking the ratio.
\end{proof}

\section{Random \texorpdfstring{$r$}{r}-compositions}\label{sec:r_versions3}
A \emph{composition} of $n$ into $k$ summands is a tuple $(b_1,\dots, b_k)$ with $b_1,\dots, b_k \in \N$ and $b_1+\dots+b_k = n$. Using stars-and-bars argument, one can check that the number of compositions of $n$ into $k$ summands is $\binom{n-1}{k-1}$. Picking one of these compositions at random, we obtain a \emph{uniform random composition} of $n$ into $k$ summands. In the following we 
introduce an $r$-analogue of this notion.

\begin{definition}\label{def:B-compositions}
An \emph{incomplete composition} of $n$ is a tuple $(b_0,b_1,\dots, b_k)$ such that $b_0 \in \N_0$, $b_1,\dots, b_k\in \N$ and $b_0 + b_1 + \dots + b_k = n$. Note that $b_0$ is allowed to be $0$, while $b_1,\dots, b_k$ are not.
\end{definition}

The number of such incomplete compositions is $\binom{n}{k}$ by stars-and-bars argument. Incomplete compositions are closely related to Weyl chambers of type $B$.  Consider Weyl chambers of types $A$ and $B$ which are polyhedral cones in $\R^n$ given by $A^{(n)}:= \{(x_1,\dots,x_n)\in\mathbb{R}^n:\,x_1\leq \dots \leq x_n\}$ and $B^{(n)} := \{(x_1,\dots,x_n)\in\mathbb{R}^n:\,0 \leq x_1 \leq \dots \leq x_n\}$. Then, the $k$-dimensional faces of the cone $A^{(n)}$ are in bijective correspondence with compositions of $n$ into $k$ summands and have the form
$$
\{x_{1}= \dots = x_{b_1} \leq x_{b_1+1} = \dots = x_{b_1+b_2} \leq \dots \leq x_{b_1+\dots+ b_{k-1}+1} = \dots = x_n\}.
$$
Similarly, $k$-dimensional faces of $B^{(n)}$ are in bijective correspondence with the incomplete compositions
$(b_0,b_1,\dots, b_k)$ and have the form
\begin{multline*}
\{0 = x_1 = \dots = x_{b_0} \leq x_{b_0+1}= \dots = x_{b_0+b_1} \leq x_{b_0+b_1+1} = \dots\\ = x_{b_0+b_1+b_2} \leq \dots \leq x_{b_0+ b_1+\dots+ b_{k-1}+1} = \dots = x_n\}.
\end{multline*}

\begin{definition}\label{def:r_composition}
Fix $r\geq 0$. A \emph{random $r$-composition} of $n$ is a random vector $(b_0^{(n,k)}, b_1^{(n,k)}, \dots, b_k^{(n,k)})$ such that, for each 
incomplete composition $(b_0, b_1,\dots, b_k)$ of $n$,
\begin{equation}\label{eq:def_r_composition}
\P[b_0^{(n,k)} = b_0, b_1^{(n,k)} = b_1,\dots, b_k^{(n,k)} = b_k]
=
\frac{r^{b_0\uparrow}/b_0!}{\binom{n+r-1}{k+r-1}}
=
\frac{\binom {b_0 + r - 1}{b_0}}{\binom{n+r-1}{k+r-1}}.
\end{equation}

\end{definition}

\begin{example}
For $r=0$, we understand the right-hand side as $1/ \binom{n-1}{k-1}$ for $b_0=0$ and as $0$ for $b_0 \geq 1$. So, $b_0^{(n,k)} = 0$ with probability $1$ and, discarding the $0$-th entry, we recover the uniform distribution on the set of compositions of $n$.
\end{example}

\begin{example}
For $r= 1$, 
each incomplete composition $(b_0, b_1,\dots, b_k)$ is equally likely, with the corresponding probability being 
$1/ \binom{n}{k}$. Thus, we recover the uniform distribution on the set of all incomplete compositions which are in bijective correspondence with the $k$-faces of the Weyl chamber $B^{(n)}$. Note in passing that it would be possible to give a similar interpretation for the $r$-uniform \emph{partitions}  with $r=1/2$ by considering uniformly distributed elements of the subspace lattice generated by the type $B$ reflection arrangement, see~\cite[\S~6.3, 6.7]{aguiar_mahajan_book_hyperplane_arrangements} for a description of the correspondence between (incomplete) partitions and flats in reflection arrangements.
\end{example}

With the help of the next lemma we will show that Definition~\ref{def:r_composition} indeed defines a probability distribution.
\begin{lemma}\label{lem:identity_bin_coeff_r_lah}
Let $r\geq 0$. Then, for each 
$m\in \N$ and $\ell \in \{0,1,\dots, m\}$,
\begin{equation}\label{eq:identity_bin_coeff_for_r_lah}
\sum_{b_0=0}^{m-\ell} \binom{b_0 + r - 1}{b_0} \binom{m-b_0}{\ell} = \sum_{b_0=0}^{m-\ell} \binom{b_0 + r - 1}{r-1} \binom{m-b_0}{\ell}= \binom{m+r}{\ell + r}.
\end{equation}
\end{lemma}
\begin{proof}
The first equality is trivial and the second one is Eq.~(5.26) in~\cite{graham_etal_book}, provided that $r\in\N$. The latter constraint is superfluous (indeed, both sides are polynomials in $r$). For the reader's convenience we give below yet another proof of~\eqref{eq:identity_bin_coeff_for_r_lah} valid for all real $r\geq 0$. We use the Taylor expansions
$$
\sum_{b_0=0}^{\infty} \binom{b_0 + r - 1}{b_0} x^{b_0} = (1-x)^{-r}
\quad
\text{ and }
\quad
\sum_{j=\ell}^\infty \binom j\ell x^j = \frac{x^\ell}{(1-x)^{\ell+1}}.
$$
Multiplying these Taylor series and evaluating the coefficient of $x^m$, we arrive at 
the left-hand side of~\eqref{eq:identity_bin_coeff_for_r_lah}. On the other hand, the same coefficient can be computed as follows:
$$
[x^m] \left(\frac{(1-x)^{-r}  x^\ell}{(1-x)^{\ell+1}}\right) = [x^{m-\ell}] (1-x)^{-\ell - r - 1} = [x^{m-\ell}] \sum_{j=0}^\infty \frac{x^j}{j!} (\ell + r + 1)^{j\uparrow} = \frac{(\ell + r + 1)^{j\uparrow}}{(m-\ell)!} = \binom{m+r}{\ell+r}.
$$
Comparing these results completes the proof.
\end{proof}
\begin{remark}
Now we can show that~\eqref{eq:def_r_composition} indeed defines a probability distribution on the set of incomplete compositions of $n$. The number of  incomplete compositions with a given value of $b_0\in \{0,\dots, n-k\}$ is $\binom{n-b_0-1}{k-1}$. Each such an incomplete composition has probability given by the right-hand side of~\eqref{eq:def_r_composition}. The sum of all such probabilities is
$$
\frac {1}{\binom{n+r-1}{k+r-1}} \sum_{b_0=0}^{n-k} \binom{n-b_0-1}{k-1} \binom {b_0 + r - 1}{b_0} =  1
$$
by Lemma~\ref{lem:identity_bin_coeff_r_lah} with $m=n-1$ and $\ell = k-1$.

\end{remark}

\subsection{Couplings of random \texorpdfstring{$r$}{r}-compositions and the Dirichlet multinomial distribution}\label{subsec:couplings}

In this subsection we aim at constructing 
an $r$-composition via a certain urn scheme, in a way similar to that 
used for uniform $r$-partitions in Section~\ref{sec:r_versions2}. As a result, we provide 
several couplings of $r$-compositions which are consistent either in $n$ or in $k$. It
turns out that the random $r$-composition can be constructed using an urn scheme with random frequencies.

Formula~\eqref{eq:def_r_composition} implies that the random vector $(b_0^{(n,k)}, b_1^{(n,k)}-1, \dots, b_k^{(n,k)}-1)$ has the Dirichlet multinomial distribution
\begin{equation}\label{eq:mdir_rcompositions}
\MDir{n-k}{r,\underbrace{1,\dots,1}_{k\text{ times}}},
\end{equation}
see definition~\eqref{eq:MDir_def}. It is 
known that the probability measure $\MDir{n}{\alpha_0,\alpha_1,\dots,\alpha_k}$ can be regarded as a mixture of multinomial distributions. More precisely,
$$
\MDir{n}{\alpha_0,\alpha_1,\dots,\alpha_k}=\MultNP{n}{\Dir{\alpha_0,\alpha_1,\dots,\alpha_k}},
$$
where the Dirichlet distribution $\Dir{\alpha_0,\alpha_1,\dots,\alpha_k}$ is defined by the density
\begin{equation}\label{eq:dirichlet}
(x_0,\dots,x_k)\mapsto \frac{\Gamma(\alpha_0)\Gamma(\alpha_1)\cdots\Gamma(\alpha_k)}{\Gamma(\alpha_0+\alpha_1+\dots+\alpha_k)}x_0^{\alpha_0-1}x_1^{\alpha_1-1}\cdots x_k^{\alpha_k-1},\quad \sum_{i=0}^{k}x_i=1,\quad x_i>0.
\end{equation}

The above interpretation suggests a coupling of random $r$-compositions which is consistent in $n\geq k$ if $k$ is fixed. The coupling works as follows. Take a {\it random} partition $\mathcal{P}$ of $[0,1)$ into $k+1$ intervals (urns)
$$
I_{k,0}:=[0,P_0),\quad I_{k,1}:=[P_0,P_0+P_1),\quad \dots\quad I_{k,k}:=[P_0+\dots+P_{k-1},1),
$$
where $(P_0,P_1,\dots,P_k)$ follows $\Dir{r,1,\dots,1}$ distribution. In a role of balls take an independent of $\mathcal{P}$ sample $U_1,U_2,\dots,U_n,\dots$ from the uniform distribution on $[0,1]$. Let $Z_{n,i}$, $i=0,\dots,k$ be the number of balls among $U_1,\dots,U_n$ in the urn $I_{k,i}$, that is,
$$
Z_{n,i}:=\#\{1\leq j\leq n: U_j\in I_{k,i}\},\quad 0\leq i\leq k.
$$
The following trivially holds true.
\begin{proposition}\label{prop:coupling1}
Fix $k\in\N$. For each 
$n\geq k$, the random vectors $(Z_{n-k,0},Z_{n-k,1},\dots,Z_{n-k,k})$ and $(b_0^{(n,k)}, b_1^{(n,k)}-1, \dots, b_k^{(n,k)}-1)$ have the same Dirichlet multinomial distribution~\eqref{eq:mdir_rcompositions}.
\end{proposition}

The aggregation property of the Dirichlet distribution asserts that if
$$
(X_0,X_1,\dots,X_k)~\sim~\Dir{\alpha_0,\alpha_1,\dots,\alpha_k},
$$
then
$$
(X_0+X_1,X_2,\dots,X_k)~\sim~\Dir{\alpha_0+\alpha_1,\alpha_2,\dots,\alpha_k}.
$$
This property suggests another coupling of $r$-compositions which is consistent in both parameters $n$ and $k$ but works only when $r$ is an integer. This coupling hinges on a simple observation that the collection of gaps between consecutive order statistics from the uniform distribution on $[0,1]$ has a symmetric Dirichlet distribution $\Dir{1,1,\dots,1}$. More precisely, let $V_1,V_2,\dots$ be a sample from the uniform distribution on $(0,1)$. For each 
$M\in\N$, let
$$
0<V_{M,1}<V_{M,2}<\dots<V_{M,M}<1
$$
be the order statistics of the first $M$ points of the 
sample. Then
$$
(V_{M,1},V_{M,2}-V_{M,1},\dots,V_{M,M}-V_{M,M-1},1-V_{M,M})~\sim~\Dir{\underbrace{1,\dots,1}_{M+1\text{ times}}}.
$$
By the aggregation property, if $r$ is an integer
$$
(V_{k+r,r},V_{k+r,r+1}-V_{k+r,r},\dots,V_{k+r,k+r}-V_{k+r,k+r-1},1-V_{k+r,k+r})~\sim~\Dir{r,\underbrace{1,\dots,1}_{k+1\text{ times}}}.
$$
Thus, a family $(I_{k,0},I_{k,1},\dots,I_{k,k})$, $k\in\N$ of nested partitions of $[0,1)$ such that the lengths of $(I_{k,0},I_{k,1},\dots,I_{k,k})$ have the 
$\Dir{r,1,\dots,1}$ distribution can be sequentially constructed as follows. Start with a sample $V_1,\dots,V_{r}$ of size $r$ from the uniform distribution on $[0,1]$. This sample yields a partition of $[0,1)$ into two intervals $I_{1,0}:=[0,V_{r,r})$ and $I_{1,1}:=[V_{r,r},1)$, whose lengths follow the $\Dir{r,1}$ distribution. Add a new point $V_{r+1}$. If $V_{r+1}$ lands into 
$I_{1,1}$, then it splits $I_{1,1}$ 
into two subintervals $I_{2,1}:=[V_{r,r},V_{r+1})$ and $I_{2,2}:=[V_{r+1},1)$. In this case, put also $I_{2,0}:=[0,V_{r,r})$. On the other hand, if $V_{r+1}$ lands into 
$I_{1,0}$, the splitting is $I_{2,0}:=[0,V_{r+1,r})$ and $I_{2,1}:=[V_{r+1,r},V_{r,r})$. In this case, put also $I_{2,2}:=I_{1,1}$. More generally, if the partition $(I_{k,0},I_{k,1},\dots,I_{k,k})$ has already been constructed, adding a point $V_{r+k}$ yields a partition $(I_{k+1,0},I_{k+1,1},\dots,I_{k+1,k+1})$ depending on whether $V_{r+k}\in I_{k,0}$ or $V_{r+k}\notin I_{k,0}$. Formally, the construction is given by the equations
\begin{multline}\label{eq:nested_intervals}
I_{k+1,0}:=[0,V_{r+k,r}),\quad I_{k+1,j}:=[V_{r+k,r+j-1},V_{r+k,r+j}),\quad j=1,\dots,k,\\I_{k+1,k+1}=[V_{r+k,r+k},1).
\end{multline}
We used the word `nested' to highlight that the partition $(I_{k+1,0},I_{k+1,1},\dots,I_{k+1,k+1})$ is obtained from the partition $(I_{k,0},I_{k,1},\dots,I_{k,k})$ by splitting one of the intervals $I_{k,j}$ into two subintervals. By taking samples $(V_i)$ and $(U_i)$ independent and appealing to Proposition~\ref{prop:coupling1} we obtain (in the case $r\in\N$) a coupling of $r$-compositions with the following properties:
\begin{itemize}[leftmargin=*]
\item an $r$-partition $(b_0^{(n,k+1)},b_1^{(n,k+1)},\dots,b_{k+1}^{(n,k+1)})$ is derived from an $r$-partition $(b_0^{(n,k)},b_1^{(n,k)},\dots,b_{k}^{(n,k)})$ by splitting one of blocks $b_j^{(n,k)}$ into two parts;
\item an $r$-partition $(b_0^{(n+1,k)},b_1^{(n+1,k)},\dots,b_{k}^{(n+1,k)})$ is derived from $(b_0^{(n,k)},b_1^{(n,k)},\dots,b_{k}^{(n,k)})$ by increasing one of blocks by one without changing other blocks.
\end{itemize}

Proposition~\ref{prop:coupling1} is a de Finetti-type result, which provides 
a realization of $(b_0^{(n,k)}, b_1^{(n,k)}, \dots, b_k^{(n,k)})$ via a randomized multinomial allocation scheme. Given next is a Markovian construction that gives a realization of the process $(b_0^{(n,k)}, b_1^{(n,k)}, \dots, b_k^{(n,k)})_{n\geq k}$, with $k$ fixed, 
as consecutive values of a certain Markov chain.
\begin{proposition}\label{prop:coupling2}
Fix $k\in\N$. Let $(X(l))_{l\in\N_0}$ be a Markov chain on $[0,\infty)\times \N^k$ with the initial state
$X(0)=(r,1,\dots,1)$ and transition probabilities
\begin{multline*}
\P[X(l+1)=(i_0,i_1,\dots,i_k)+\eee_j|X(l)=(i_0,i_1,\dots,i_k)]=\frac{i_j}{i_0+i_1+\dots+i_k},\\
j=0,\dots,k,\quad l\in\N_0,
\end{multline*}
where $\eee_j$ is the $j$-th unit vector in $\R^{k+1}$. For each 
fixed $n\geq k$, $X(n-k)-r\eee_0$ has the same distribution as $(b_0^{(n,k)}, b_1^{(n,k)}, \dots, b_k^{(n,k)})$.
\end{proposition}
The Markov chain constructed in Proposition~\ref{prop:coupling2} is nothing else but a P\'{o}lya's urn. Initially, the urn contains $r$ balls of color $0$ and one ball of each of other colors $1,2,\dots,k$. At 
each step, a ball is drawn uniformly at random from the urn, its color is observed, and the ball is returned to the urn with an additional ball of the same color. The vector $X(n)$ encodes the colors of balls in the urn after $n$ steps.

For later needs, we formulate a lemma.

\begin{lemma}\label{lem:assoc}
Fix $n,k\in\N$. The random vector $(X_0^{(n)},\dots, X_k^{(n)})$ with the $\MDir{n}{\alpha_0,\alpha_1,\dots,\alpha_k}$ distribution is negatively associated, that is, for each pair of disjoint subsets $A_1,A_2$ of $\{0,1,\dots, k\}$ and all nondecreasing in each coordinate functions $f:\R^{|A_1|}\to\R$ and $g:\R^{|A_2|}\to\R$ $$\E[f(X_i^{(n)},\,i\in A_1)g(X_j^{(n)},\,j\in A_2)]\leq \E[f(X_i^{(n)},\,i\in A_1)]\E[g(X_j^{(n)},\,j\in A_2)].$$
\end{lemma}
\begin{proof}
This claim is given on p.~292 in~\cite{JoagDev+Proschan:1983} without proof. We 
now give a proof. Let $Y_0$, $Y_1,\dots, Y_k$ be independent random variables such that, for $i\in\{0,\dots, k\}$, $Y_i$ has a negative binomial distribution with parameters $\alpha_i$ and $\beta\in (0,1)$, that is,
$$
\P[Y_i=k]=\NBin{\alpha_i}{\beta}(\{k\})=\frac{\alpha_i^{k\uparrow}}{k!}\beta^{\alpha_i}(1-\beta)^k,\quad k\in\N_0.
$$
We claim that the distribution of $(X_0^{(n)},\dots, X_k^{(n)})$ coincides with the conditional distribution of $(Y_0,\dots, Y_k)$ given $\sum_{i=0}^k Y_i=n$.  Indeed, for each 
$(b_0,b_1,\dots,b_k)\in \N_0^{k+1}$ summing up to  $n$,
\begin{multline*}
\P\left[Y_0=b_0,Y_1=b_1,\dots,Y_k=b_k\Big|\sum_{i=0}^k Y_i=n\right]=\frac{\P[Y_0=b_0]\P[Y_1=b_1]\cdots\P[Y_k=b_k]}{\P[\sum_{i=0}^k Y_i=n]}\\
=\left(\prod_{i=0}^{k}\NBin{\alpha_i}{\beta}(\{b_i\})\right)/\NBin{\alpha_0+\dots+\alpha_k}{\beta}(\{n\}).
\end{multline*}
After cancellations this becomes the right-hand side of~\eqref{eq:MDir_def} with $d=k+1$, $k_i=b_i$ and $\theta_i=\alpha_i$, for $i=0,\dots,k$. In view of the above representation we may use \cite[Theorem 2.6]{JoagDev+Proschan:1983}, 
which states it is sufficient to show that for each subset $A$ of $\{0,1,\dots, k\}$ and all nondecreasing in each coordinate functions $f:\R^{|A|}\to\R$ the sequence $$n\mapsto \E\Big[f(Y_i,\,i\in A)\Big|\sum_{j=0}^k Y_j=n\Big]=\E[f(X_i^{(n)},\,i\in A)]$$ is nondecreasing. This property is secured by the aforementioned coupling, which enables us to think of $X_i^{(n)}$ as the occupancy count of the box $i$ corresponding to $(P_0,P_1,\dots,P_k)$ having the $\Dir{\alpha_0,\alpha_1,\dots,\alpha_k}$ distribution. Throwing a new ball does not decrease an occupancy count, that is, $X_i^{(n)}\leq X_i^{(n+1)}$ a.s.
\end{proof}

Lemma~\ref{lem:assoc} in combination with Proposition~\ref{prop:coupling1} immediately implies a result to be used later 
in the proof of Lemma~\ref{lem:meanvar}.
\begin{cor}\label{cor:assoc}
The random $r$-composition $(b_0^{(n,k)},b_1^{(n,k)},\dots, b_{k}^{(n,k)})$ is negatively associated.
\end{cor}

\subsection{Marginal distributions of \texorpdfstring{$r$}{r}-compositions and their limits}
\begin{proposition}\label{prop:r_composition_marginal_distr_exact}
The marginal distributions of the random $r$-composition $(b_0^{(n,k)}, b_1^{(n,k)}, \dots, b_k^{(n,k)})$ are given by
\begin{align}
\P\left[b_0^{(n,k)} = b_0\right]
&=
\frac{\binom{n-b_0-1}{k-1}\binom{b_0+r-1}{b_0}}{\binom{n+r-1}{k+r-1}},
\qquad
b_0 \in \{0,1,\dots, n-k\},
\label{eq:r_composition_marginal_distr_b_0}
\\
\P\left[b_j^{(n,k)} = b_j\right]
&=
\frac{\binom{n-b_j+r-1}{k+r-2}}{\binom{n+r-1}{k+r-1}},
\qquad
b_j\in \{1,\dots, n-k+1\}.
\label{eq:r_composition_marginal_distr_b_j}
\end{align}
Moreover, for $0<i<j\leq k$, the bivariate distributions are given by
\begin{align}
\P\left[b_i^{(n,k)} = b_i, b_j^{(n,k)} = b_j\right]
&=
\frac{\binom{n+r-b_i-b_j-1}{k+r-3}}{\binom{n+r-1}{k+r-1}},\quad b_i, b_j\in \{1,\dots, n-k+1\},\quad b_i+b_j\leq n-k+2.
\label{eq:r_composition_joint distr_b_j}
\end{align}
\end{proposition}
\begin{proof}
For~\eqref{eq:r_composition_marginal_distr_b_0}, just observe that the number of incomplete compositions with fixed $b_0$ is $\binom{n-b_0-1}{k-1}$ and the probability of each such an incomplete composition is given by the right-hand side of~\eqref{eq:def_r_composition}. For~\eqref{eq:r_composition_marginal_distr_b_j}, we can take $j=1$ without loss of generality. In the case $k=1$, \eqref{eq:r_composition_marginal_distr_b_0} secures
$$
\P\left[b_1^{(n,k)} = b_1\right]=\P\left[b_0^{(n,k)} = n-b_1\right]=\frac{\binom{n-b_1+r-1}{n-b_1}}{\binom{n+r-1}{k+r-1}}=\frac{\binom{n-b_1+r-1}{r-1}}{\binom{n+r-1}{k+r-1}}.
$$
Assume now that $k\geq 2$. Observe that the number of incomplete compositions with fixed $b_0$ and $b_1$ is given by $\binom{n-b_0-b_1-1}{k-2}$. Since the probability of each such an incomplete composition is given by the right-hand side of~\eqref{eq:def_r_composition}, we obtain
$$
\P\left[b_1^{(n,k)} = b_1\right]
=
\sum_{b_0=0}^{n-b_1- k+1} \binom{n-b_0-b_1-1}{k-2} \frac{\binom {b_0 + r - 1}{b_0}}{\binom{n+r-1}{k+r-1}}
=
\frac{\binom{n-b_1-1+r}{k-2+r}}{\binom{n+r-1}{k+r-1}},
$$
where we evaluated the sum  using Lemma~\ref{lem:identity_bin_coeff_r_lah} with $m:= n-b_1-1$ and $\ell := k-2$.

Alternatively, one can use a 
known property of the Dirichlet distribution $\Dir{\alpha_0,\alpha_1,\dots,\alpha_k}$, which asserts that its $i$-th component has the beta distribution $
\Betadistr{\alpha_i}{\alpha-\alpha_i}$, where $\alpha:=\sum_{j=0}^k\alpha_j$. Then, by Proposition~\ref{prop:coupling1},
\begin{equation}\label{eq:b0_and_bj_via_betabinomial}
b_0^{(n,k)}~\sim~\Bin{n-k}{\Betadistr{r}{k}}\quad\text{and}\quad b_j^{(n,k)}~\sim~1+\Bin{n-k}{\Betadistr{1}{k+r-1}}.
\end{equation}
Formulas~\eqref{eq:r_composition_marginal_distr_b_0} and~\eqref{eq:r_composition_marginal_distr_b_j} follow by simple manipulations with beta-integrals.

For the proof of~\eqref{eq:r_composition_joint distr_b_j}, note that by the aggregation property of the Dirichlet distribution in combination with 
Proposition~\ref{prop:coupling1}, 
$$
\left(b_0^{(n,k)}+\sum_{\ell\notin \{0,i,j\}}(b_{\ell}^{(n,k)}-1),b_i^{(n,k)}-1,b_j^{(n,k)}-1\right)~\sim~\MDir{n-k}{\Dir{r+k-2,1,1}}
$$
for each pair of indices $0<i<j\leq k$. Thus,
$$
\E [t^{b_i^{(n,k)}-1}s^{b_j^{(n,k)}-1}]=\frac{\Gamma(r+k)}{\Gamma(r+k-2)}\iiint (x+ty+sz)^{n-k}x^{r+k-3}{\rm d}x{\rm d}y{\rm d}z,\quad t,s\in\C,
$$
where the integration is taken over the simplex $x+y+z=1$, $x,y,z\geq 0$. By expanding $(x+ty+sz)^{n-k}$ with the aid of the multinomial theorem and equating the coefficients, we conclude that
$$
\P\left[b_i^{(n,k)} = b_i, b_j^{(n,k)} = b_j\right]=\frac{\Gamma(r+k)}{\Gamma(r+k-2)}\frac{(n-k)!}{(n-k+2-b_i-b_j)!}\frac{\Gamma(n+r-b_i-b_j)}{\Gamma(n+r)}.
$$
This yields~\eqref{eq:r_composition_joint distr_b_j} after elementary manipulations.
\end{proof}

\begin{corollary}\label{cor:without_b0_uniform}
The conditional distribution of $(b_1^{(n,k)},\dots,b_k^{(n,k)})$ given $b_0^{(n,k)}$ is uniform on the set of all (usual) compositions of $n-b_0^{(n,k)}$ into $k$ blocks.
\end{corollary}

\begin{remark}\label{rem:expectation_r_compositions}
For $r=0$, $\P[b_j^{(n,k)} = b_j] =\binom{n-b_j-1}{k-2}/\binom{n-1}{k-1}$ for all $j\in \{1,\dots, k\}$. Now observe that the formula for the marginal distribution of $b_j^{(n,k)}$ for general $r\geq 0$ can be obtained from the $r=0$ case if we replace $n$ and $k$ by $n+r$ and $k+r$, respectively. This observation has several implications given below.
\end{remark}

\begin{proposition}
For $r\geq 0$,
\begin{equation}\label{eq:r_composition_expectation}
\E [b_j^{(n,k)}] = \frac{n+r}{k+r},
\qquad
 j \in \{1,\dots, k\},
\qquad
\E [b_0^{(n,k)}] = \frac{r(n-k)}{k+r}.
\end{equation}
\end{proposition}
\begin{proof}
For $r=0$ it is clear, by exchangeability, that $\E [b_j^{(n,k)}] = n/k$ for all $j\in \{1,\dots, k\}$. According to Remark~\ref{rem:expectation_r_compositions}, for general $r\geq 0$, we have to replace $n$ and $k$ by $n+r$ and $k+r$, which gives the first formula in~\eqref{eq:r_composition_expectation}. The second formula follows from $b_0^{(n,k)} = n - b_1^{(n,k)} - \dots -  b_k^{(n,k)}$.
\end{proof}

A strong law of large numbers for the multinomial distribution immediately implies a limit theorem for $r$-compositions with $k$ being fixed.

\begin{proposition}\label{prop:r_composition_marginal_convergence_fixed_k}
Let $k\in\N$ be a fixed integer and $r\geq 0$. Then
$$
\frac{1}{n}(b_0^{(n,k)}, b_1^{(n,k)}, \dots, b_k^{(n,k)})\todistr \Dir{r,\underbrace{1,\dots,1}_{k\text{ times}}}.
$$
\end{proposition}

\begin{proposition}\label{prop:r_composition_marginal_distr_asymptotic}
Fix some $r\geq 0$ and let $n\to\infty$, $k\to\infty$ such that $k/n \to \alpha$ for some $\alpha \in (0,1]$. Then, for a random $r$-composition  $(b_0^{(n,k)}, b_1^{(n,k)}, \dots, b_k^{(n,k)})$,
\begin{align*}
\lim_{n\to\infty} \P\left[b_0^{(n,k)} = b_0\right]
&=
\binom{b_0+r-1}{b_0} (1-\alpha)^{b_0}\alpha^{r},
\qquad
b_0 \in \{0,1,\dots\},\\
\lim_{n\to\infty} \P\left[b_j^{(n,k)} 
= b_1\right]
&=
\alpha (1-\alpha)^{b_1-1},
\qquad
b_1\in \{1,2,\dots\},~~j\in\{1,\dots, k\}.
\end{align*}
That is to say, $b_0^{(n,k)}$ converges weakly to the 
negative binomial distribution $\NBin{r}{\alpha}$, while $b_j^{(n,k)}$ 
converges weakly to the geometric distribution $\Geo{\alpha}$.
\end{proposition}
\begin{proof}
Follows from Proposition~\ref{prop:r_composition_marginal_distr_exact} after letting $n\to\infty$.
\end{proof}

\begin{proposition}\label{prop:r_composition_marginal_distr_asymptotic_alpha_0}
Fix some $r\geq 0$ and let $n\to\infty$, $k\to\infty$ such that $k/n \to 0$.  Then, for a random $r$-composition  $(b_0^{(n,k)}, b_1^{(n,k)}, \dots, b_k^{(n,k)})$,
\begin{equation}\label{eq:r_composition_marginal_distr_asymptotic_alpha_0}
\frac{b_0^{(n,k)}}{n/k}~\todistrk~\Gammadistr{r}{1}\quad\text{and}\quad \frac{b_1^{(n,k)}}{n/k}~\todistrk~\Exp{1},
\end{equation}
where $\Gammadistr{r}{1}$ is the gamma distribution with parameters $r$ and $1$ and $\Exp{1}$ is the standard exponential distribution.
\end{proposition}
\begin{proof}
The easiest way to prove this is to use representations~\eqref{eq:b0_and_bj_via_betabinomial}. The first claim in~\eqref{eq:r_composition_marginal_distr_asymptotic_alpha_0} follows from the next three facts. First, with $(\gamma(t))_{t\geq 0}$ being the standard gamma-subordinator,
$$
\frac{\gamma(r)}{\gamma(r+k)}~\sim~\Betadistr{r}{k},
$$
see~\cite[Section 9.1]{Kingman}.
Second, for each 
fixed $\lambda>0$,
$$
\frac{k}{n}\Bin{n-k}{\lambda/k}~\overset{\P}{\to}\lambda,\quad k,n\to\infty,\quad k/n\to 0,
$$
as can be 
checked by using Chebyshev's inequality. Third, by the strong law of large numbers
$$
\frac{\gamma(r+k)}{k}~\overset{a.s.}{\underset{k\to\infty}\longrightarrow}~1.
$$
It remains to note that $\gamma(r)~\sim~\Gammadistr{r}{1}$.

The second claim in~\eqref{eq:r_composition_marginal_distr_asymptotic_alpha_0} follows by the same reasoning with the help of 
the second relation in~\eqref{eq:b0_and_bj_via_betabinomial}. The proof is complete.
\end{proof}

\section{The \texorpdfstring{$(r,s)$}{(r,s)}-Lah distribution}\label{sec:Lah_distribution}

\subsection{Motivation}
In this section we investigate 
the so-called $(r,s)$-Lah distribution which is defined in terms of the $r$-Stirling numbers of both kinds. Special cases of this distribution appear 
in~\cite{van_der_hofstad_etal_shortest_path_trees} in the context of random recursive trees (and the closely related shortest path trees on the complete graph with exponentially distributed weights), and in~\cite{kabluchko_marynych_lah,kabluchko_steigenberger_r_lah}. In the last two articles, 
limit theorems for this distribution are 
applied to obtain 
asymptotics of the face numbers of random walk convex (and positive) hulls~\cite{godlandschlaefli:2022,godland:2022,KVZ17,kabluchko:2016}. All these papers investigate the $(r,s)$-Lah distribution 
for certain special values of parameters, for example, the distribution appearing in~\cite{van_der_hofstad_etal_shortest_path_trees} corresponds to $(r,s) = (1,0)$, the distributions in~\cite{kabluchko_marynych_lah,KVZ17,kabluchko:2016} and~\cite{godland:2022} correspond to $(0,0)$ and $(1/2,1/2)$, respectively, while~\cite{kabluchko_steigenberger_r_lah} is concerned with 
the case $r=s$.

Our purpose is to introduce a general distributional family unifying all these special cases. We 
relate these general $(r,s)$-Lah distributions to the 
$(r+s)$-versions of the classic combinatorial probability structures 
defined in the previous sections. 
In particular, we 
derive several stochastic representations of the $(r,s)$-Lah distribution in terms of random 
$(r+s)$-compositions and (multinomial) Hoppe trees. In Section~\ref{sec:Lah_limits} we use 
these representations when proving 
limit theorems for the $(r,s)$-Lah distribution.

\subsection{The \texorpdfstring{$r$}{r}-Lah numbers}
We start by reviewing some properties of the $r$-Lah numbers $L(n,k)_r$; see~\cite{nyul:2015} as well as~\cite{belbachir:2013,cheon:2012,shattuck:2016} for more information. For $n\in \N_0$, $k\in \{0,\dots, n\}$ and $r\in \R$, these numbers may be defined~\cite[Theorem~3.10]{nyul:2015} by their exponential generating function
\begin{equation}\label{eq:r_Lah_exp_gener}
\sum_{n=k}^\infty L(n,k)_r \frac {x^n}{n!} = \frac 1 {k!} \left(\frac x {1-x}\right)^k (1-x)^{-2r}.
\end{equation}
We extend this definition by putting $L(n,k)_r:=0$ for $n\in \N_0$ and $k\notin\{0,\dots,n\}$.
The Taylor expansion of the right-hand side of~\eqref{eq:r_Lah_exp_gener} gives~\cite[Theorem~3.7]{nyul:2015}
\begin{align}\label{eq:r_Lah_def1}
L(n,k)_r
=
\binom{n+2r-1}{k+2r-1} \frac{n!}{k!}.
\end{align}
As a consequence,  $L(n,k)_r$ is a polynomial of $r$, for fixed $n$ and $k$. The following formula~\cite[Theorem~3.2]{nyul:2015} complements~\eqref{eq:stirling_r_rising_factorials}:
\begin{equation}\label{eq:lah_r_rising_falling_factorials}
(x+r)^{n\uparrow} =  \sum_{k=0}^n L(n,k)_r (x-r)^{k\downarrow}.
\end{equation}
Using~\eqref{eq:stirling_r_rising_factorials} and~\eqref{eq:lah_r_rising_falling_factorials} it can be shown~\cite[Theorem~3.11]{nyul:2015} that
\begin{align}\label{eq:r_Lah_def2}
L(n,k)_r
=
\sum_{j=k}^n \stirling{n}{j}_r \stirlingsec{j}{k}_r.
\end{align}
More generally, from the same reference it is known  that
\begin{equation}\label{eq:r_lah_stirling_r_s}
L(n,k)_{\frac{r+s}{2}}
= \sum_{j=k}^n \stirling{n}{j}_r \stirlingsec{j}{k}_s.
\end{equation}
The cited paper~\cite{nyul:2015} only considers nonnegative integer $r$ and requires that $r$ and $s$ have the same parity. In fact, since both sides of the identity are polynomials in $r$ and $s$, it holds for arbitrary $r,s\in \R$.

From Eq.~\eqref{eq:def_r_composition} in the definition of $r$-compositions we conclude that
$$
L(n,k)_r = \frac{n!}{k!}\sum_{\substack{\ell_0\in \N_0, \ell_1,\dots, \ell_k\in \N\\ \ell_0+\ell_1 + \dots + \ell_k = n}} \frac{(2r)^{\ell_0\uparrow}}{\ell_0!},
$$
which already indicates a connection of the $r$-Lah 
numbers with the random $2r$-compositions.

\subsection{Definition of the \texorpdfstring{$(r,s)$}{(r,s)}-Lah distribution}

\begin{definition} \label{DfnrLah}
An \emph{$(r,s)$-Lah distribution} $\Lah{n}{k}{r}{s}$ with parameters $n \in \N$, $k\in \{0,\dots, n\}$ and $r,s \in [0,\infty)$ (where the case $k=r=s=0$ is always excluded) is a discrete probability measure defined by
\begin{align}\label{eq:Lah_distr_def}
\Lah{n}{k}{r}{s}(\{j\}) = \frac{1}{L(n,k)_{\frac{r+s}{2}}} \stirling{n}{j}_r \stirlingsec{j}{k}_s, \qquad  j \in \lbrace k , k+1, \dotsc, n \rbrace.
\end{align}
\end{definition}

\begin{definition}\label{DfnrParamLah}
We call  $(n,k,r,s)$ an \emph{admissible quadruple} of parameters for the $(r,s)$-Lah distribution if  $n \in \N$, $k\in \{0,\dots, n\}$, $r,s \in [0,\infty)$ and $\max\{k,r,s\} > 0$.
\end{definition}
We exclude the case $k=r=s=0$ in which $L(n,0)_0 = 0$, $n\in \N$,  since  $(-1)! = \infty$. Equation~\eqref{eq:Lah_distr_def} indeed defines a probability distribution due to 
identity~\eqref{eq:r_lah_stirling_r_s}.

Now we 
review several particular cases of the $(r,s)$-Lah distribution that have already appeared in the literature.
\begin{example}\label{ex:lah_distr_simple_special_cases}
If $k=0$ and $s\geq 0$, then in view of 
$\stirlingsec{j}{0}_s = s^j$, the distribution takes the form
$$
\Lah{n}{0}{r}{s}(\{j\}) = \frac{1}{(r+s)^{n\uparrow}} \stirling{n}{j}_r s^j, \qquad  j \in \lbrace 0 , 1, \dots, n \rbrace.
$$
Thus, $\Lah{n}{0}{r}{s}$ is the $r$-Stirling distribution of the first kind $\rStirone{n}{s}{r}$. In particular, if also $s=0$, then $\Lah{n}{0}{r}{0}$ is the degenerate at $0$ distribution for each 
$r>0$.
\end{example}

\begin{example}\label{ex:lah_distr_simple_special_cases_k=1}
Let now $k=1$ and $s=0$. Note that, for $j\in \{1,\dots,n\}$,
$$
\Lah{n}{1}{r}{0}(\{j\})\overset{\eqref{eq:r_Lah_def1}}{=}\frac{1}{n!\binom{n+r-1}{r}}\stirling{n}{j}_r\stirlingsec{j}{1}_0\overset{\eqref{eq:r_stirling_def_finite_difference}}{=}\frac{1}{n!\binom{n+r-1}{r}}\stirling{n}{j}_r=\frac{n+r}{n}\frac{1}{(r+1)^{n\uparrow}}\stirling{n}{j}_r.
$$
Thus, if $Z\sim \rStirone{n}{1}{r}$, then $\Lah{n}{1}{r}{0}$ is the conditional distribution of $Z$ given $\{Z>0\}$.
\end{example}

\begin{example}
Motivated by~\cite{KVZ17,kabluchko:2016} the $(r,s)$-Lah distribution with $r=s=0$ was investigated in \cite{kabluchko_marynych_lah}. This distribution
has a combinatorial interpretation in terms of fragmented permutations. A fragmented permutation of $n$ elements with $k$ fragments is a partition of the set $\{1,\dots,n\}$ into $k$ nonempty blocks together with the structure of permutation on each 
block. The number of fragmented permutations with $k$ blocks is the Lah number $L(n,k)_0$. Consider a random fragmented permutation with $k$ blocks sampled uniformly from the set of all such fragmented permutations. Then, the number of cycles of this fragmented permutation has the $(r,s)$-Lah distribution with parameters $r=s=0$.
\end{example}

\begin{example}\label{ex:hofstad}
Consider a complete graph on $n+1$ nodes in which the edges have i.i.d.\ unit exponential weights. The union of shortest paths from some node of this graph to the $k$ uniformly chosen nodes is a random tree. The number of edges in this tree, denoted by $H_{n+1}(k)$, is analyzed 
in~\cite{van_der_hofstad_etal_shortest_path_trees}. According to Theorem~2.1 of that paper 
$$
\P[H_{n+1}(k) = j] = \frac{k!}{n! \binom {n}{k}} \stirling{n+1}{j+1}_0 \stirlingsec{j}{k}_0,
\qquad
j\in \{k,k+1,\dots, n\}.
$$
Note that $\stirling{n}{j}_1 = \stirling{n+1}{j+1}_0$. This means that $H_{n+1}(k)$ has the same distribution as $\Lah{n}{k}{1}{0}$. We will return to this example in Section~\ref{subsec:lah_distr_rep_using_rrt}.
\end{example}

\subsection{Generating polynomial of the \texorpdfstring{$(r,s)$}{(r,s)}-Lah distribution and a recursion for the probability mass function}
The next result identifies the generating polynomial of the $(r,s)$-Lah distribution (up to a normalizing factor). It is a generalization of~\cite[Lemma~3.3 or Lemma~3.6]{kabluchko_steigenberger_r_lah}.
\begin{lemma}\label{lem:r_lah_generating_poly}
For any $n\in \N$, $k\in \{0,\dots, n\}$, $r,s \geq 0$ and $t\in\C$,
$$
\sum_{j=k}^{n}   \stirling{n}{j}_r \stirlingsec{j}{k}_s t^j = \frac{n!}{k!} [x^n] \left(\left((1-x)^{-t}-1\right)^k (1-x)^{-(st+r)}\right).
$$
\end{lemma}
\begin{proof}
We compute the Taylor series using the exponential generating functions~\eqref{eq:def_r_stir_first} and~\eqref{eq:def_r_stir_second}:
\begin{align*}
&\hspace{-1cm}\frac{1}{k!} \left((1-x)^{-t}-1\right)^k (1-x)^{-(st+r)}
=
(1-x)^{-r} \left( \frac{1}{k!} \left( \eee^{- t \log (1-x) }-1 \right)^k \eee^{-s t \log (1-x) } \right)
\notag \\
&\stackrel{\eqref{eq:def_r_stir_second}}{=}
(1-x)^{-r} \sum_{j=k}^{\infty} \left(\stirlingsec{j}{k}_s  \frac{\bigl(- t \log (1-x) \bigr)^j}{j!} \right)= \sum_{j=k}^{\infty} \stirlingsec{j}{k}_s t^j \left( \frac{ \left(-\log (1-x)\right)^j}{j!} (1-x)^{-r} \right)
\notag \\
&\stackrel{\eqref{eq:def_r_stir_first}}{=}
\sum_{j=k}^{\infty} \stirlingsec{j}{k}_s t^j \left( \sum_{n=j}^{\infty} \stirling{n}{j}_r \frac{x^n}{n!} \right) =
\sum_{j=k}^{\infty} \sum_{n=j}^{\infty} \frac{1}{n!} \stirling{n}{j}_r \stirlingsec{j}{k}_s t^j x^n.
\notag
\end{align*}
Taking the coefficient of $x^n$ gives the claimed identity.
\end{proof}

\begin{corollary}\label{cor:r_stirling_identity_general_orthogonality}
Fix some  $n\in \N$ and  $k\in \{0,\dots, n\}$. Let $r,s\in \N_0$ and $m\in \N_0$ satisfy 
$m < (n+r)/(k+s)$. Then,
$$
\sum_{j=k}^{n}   \stirling{n}{j}_r \stirlingsec{j}{k}_s (-m)^j = 0.
$$
\end{corollary}
\begin{proof}
We apply Lemma~\ref{lem:r_lah_generating_poly} with $t:=-m$.
The degree of the polynomial $(1-x)^{m}-1)^k (1-x)^{sm-r}$ is $km + sm-r < n$ under our assumption on $m$. Hence, 
the coefficient of $x^n$ vanishes.
\end{proof}

\begin{corollary}
Fix some  $n\in \N$, $k\in \{0,\dots, n\}$ and $r\in \R$.  Let $m\in \N_0$ satisfy 
$m < n/k$. Then,
$$
\sum_{j=k}^{n}   \stirling{n}{j}_r \stirlingsec{j}{k}_r (-m)^j = 0.
$$
\end{corollary}
\begin{proof}
If $r\in\N$ is sufficiently large, 
then $m < (n+r)/(k+r)$ and we can apply Corollary~\ref{cor:r_stirling_identity_general_orthogonality} with $r=s$.  Since $\stirling{n}{j}_r$ and $\stirlingsec{j}{k}_r$ are polynomials in $r$, the identity extends to all real $r$.
\end{proof}

Now we 
derive recurrence relations for the probability mass function and the generating polynomial of $\Lah{n}{k}{r}{s}$. The triangles of generalized Stirling numbers satisfy the recursions~\cite[p.~1661]{nyul:2015}:
$$
\stirling{n}{j}_r=(n+r-1)\stirling{n-1}{j}_r+\stirling{n-1}{j-1}_r,\quad \stirlingsec{j}{k}_s=\stirlingsec{j-1}{k-1}_s+(k+s)\stirlingsec{j-1}{k}_s.
$$
Upon multiplying these identities we conclude that 
\begin{equation}\label{eq:trivariate_rec}
\stirling{n}{j}_r\stirlingsec{j}{k}_s=(n+r-1)\stirling{n-1}{j}_r\stirlingsec{j}{k}_s+\stirling{n-1}{j-1}_r\stirlingsec{j-1}{k-1}_s+(k+s)\stirling{n-1}{j-1}_r\stirlingsec{j-1}{k}_s.
\end{equation}
After simple manipulations with binomial coefficients we obtain 
from~\eqref{eq:trivariate_rec} the following trivariate recursion for the probability mass function of the $(r,s)$-Lah distribution.
Put
$$
p_{n,k}(j)=\P[\Lah{n}{k}{r}{s}=j].
$$
\begin{proposition}
For $j\in\{k,k+1,\dots,n\}$, the following trivariate recursive formula holds true
\begin{multline}\label{eq:trivariate_rec1}
p_{n,k}(j)=\frac{(n+r-1)(n-k)}{n(n+r+s-1)}p_{n-1,k}(j)+\frac{(k+s)(n-k)}{n(n+r+s-1)}p_{n-1,k}(j-1)\\
+\frac{k(k+r+s-1)}{n(n+r+s-1)}p_{n-1,k-1}(j-1).
\end{multline}
\end{proposition}
Multiplying both sides of~\eqref{eq:trivariate_rec1} by $t^j$ and summing over $j$ yields 
a bivariate recursion for the generating functions
$$
G_{n,k}(t):=\sum_{j=k}^{n}p_{n,k}(j) t^j = \E [t^{\Lah{n}{k}{r}{s}}],
$$
which is,
\begin{equation}\label{eq:lah_distr_gen_funct_recurrence_rel}
G_{n,k}(t)=G_{n-1,k}(t)\left(\frac{(n+r-1)(n-k)}{n(n+r+s-1)}+\frac{(k+s)(n-k)}{n(n+r+s-1)}t\right)+\frac{k(k+r+s-1)}{n(n+r+s-1)}tG_{n-1,k-1}(t).
\end{equation}

\subsection{Stochastic representation in terms of random recursive trees} \label{subsec:lah_distr_rep_using_rrt}
It is 
known, see~\cite[Theorem O1]{devroye_records} 
or \cite{smythe_mahmoud_survey_RRT}, that in $\RRT{n}$ both the degree of the root and the distance from the root to node $n$ are $\Stirone{n}{1}$-distributed. The next theorem unifies and generalizes both claims.

\begin{theorem}\label{theo:lah_distr_rep_using_rrt}
Consider a random tree $\Hoppe{n}{r+s}$ with root $0$ of weight $r+s$ and nodes $1,\dots, n$. Think of the root as being divided into two components with weights $r$ and $s$ which means that if some node is attached to the root, it chooses one of these components with probabilities $r/(r+s)$ and $s/(r+s)$. Let $\mathfrak{A}_{n,k}$ be a random subset of $\{1,\dots, n\}$ consisting of $k$ nodes sampled uniformly without replacement. Let $\mathfrak{C}_{n,k}$ be the set of  all nodes which are attached directly to the weight $s$ component of the root.
If $\mathfrak{T}_{n,k}$ denotes the minimal subtree rooted at $0$ and containing all nodes from $\mathfrak{A}_{n,k}\cup \mathfrak{C}_{n,k}$,  then the number of edges in $\mathfrak{T}_{n,k}$ is distributed according to $\Lah{n}{k}{r}{s}$.
\end{theorem}
Before proving the theorem we 
consider some special cases.

\begin{corollary}\label{cor:hoppe_tree_random_subtree_lah}
The minimal subtree of a random tree $\Hoppe{n}{r}$ spanned by a random uniform subset of $k$ nodes (and the root) has $\Lah{n}{k}{r}{0}$ edges.
\end{corollary}
\begin{proof}
Choose $s=0$ in Theorem~\ref{theo:lah_distr_rep_using_rrt}  (meaning that the set $\mathfrak C_{n,k}$ is empty).
\end{proof}
In the case 
$r=1$ and $s=0$, Theorem~\ref{theo:lah_distr_rep_using_rrt} recovers Proposition~3.1 in~\cite{van_der_hofstad_etal_shortest_path_trees}. Note that in~\cite{van_der_hofstad_etal_shortest_path_trees} the results are stated in terms of the so-called shortest path trees; these can be reduced to the RRT's, as explained in~\cite[Sections~16.2, 16.3]{van_mieghem_book}, \cite[Sections~1,2]{van_der_hofstad_first_passage} and~\cite[Section~6.2]{van_mieghem_scaling_hopcount_report}; see also~\cite{van_der_hofstad_etal_shortest_path_trees,van2001stochastic}.

In the case where 
$r>0$ is arbitrary and $k=1$, 
recalling a characterization of  $\Lah{n}{1}{r}{0}$ given in  Example~\ref{ex:lah_distr_simple_special_cases_k=1} we 
arrive at the following

\begin{corollary}
The distance from the root of a random tree $\Hoppe{n}{r}$ to a node $V_n$ picked uniformly at random from the set of nodes $\{1,\dots, n\}$ has the same distribution as $Z\sim \rStirone{n}{1}{r}$ conditioned on $Z\neq 0$.
\end{corollary}
In the next corollary we provide a simple representation of the \emph{standard} $(r = s = 0)$ Lah distribution.

\begin{corollary}\label{cor:Lah_standard_in_RRT}
Consider a random recursive tree with root denoted for the time being by $1$ and nodes $2,\dots, n$. Then, the number of \textbf{nodes} 
in a subtree generated by the root and a uniform random subset of $k$ nodes from $\{1,\dots, n\}$ (note that $1$ may be in the subset) 
is distributed according to $\Lah{n}{k}{0}{0}$.
\end{corollary}
\begin{proof}
We consider a $\Hoppe{n}{r}$-tree with  $r$ tending to zero and apply Corollary~\ref{cor:hoppe_tree_random_subtree_lah}. Note that node $1$ is always attached to $0$, while nodes $2,\dots, n$ are attached to $0$ with probability tending to zero as $r\to 0$. Hence, if we remove $0$ and the edge between $0$ and $1$, then with probability tending to one we obtain the random recursive tree described in Corollary~\ref{cor:Lah_standard_in_RRT}. It remains to apply Corollary~\ref{cor:hoppe_tree_random_subtree_lah} and observe that the number of nodes in the RRT is the number of edges plus the removed edge between $0$ and $1$.
\end{proof}

In the case $k=1$, Corollary \ref{cor:Lah_standard_in_RRT} 
recovers  a  known result on the insertion depth in the RRT; see~\cite[Theorem~2]{smythe_mahmoud_survey_RRT} or~\cite{szymanski_max_degree}.

\begin{example}\label{example:hoppe_lag_boundary_k=0}
The case $k=0$ of Theorem~\ref{theo:lah_distr_rep_using_rrt} is straightforward. In this scenario $\mathfrak{A}_{n,k}$ is empty and $\mathfrak{C}_{n,k}$ has the mixed binomial distribution $\Bin{\Stirone{n}{r+s}}{s/(r+s)}$, where $\Stirone{n}{r+s}$ is the degree distribution of the root in $\Hoppe{n}{r+s}$. By part (ii) of Proposition~\ref{prop:r_stirl_repres} $\Bin{\Stirone{n}{r+s}}{s/(r+s)}=\rStirone{n}{s}{r}$ and by Example~\ref{ex:lah_distr_simple_special_cases} $\rStirone{n}{s}{r}=\Lah{n}{0}{r}{s}$ 
in full accordance with the $k=0$ case of Theorem~\ref{theo:lah_distr_rep_using_rrt}. 
\end{example}

\begin{proof}[Proof of Theorem~\ref{theo:lah_distr_rep_using_rrt}.]
Replacing $n$ by $n+1$, we can write 
recurrence relation~\eqref{eq:lah_distr_gen_funct_recurrence_rel} for $G_{n,k}(t)= \E [t^{\Lah{n}{k}{r}{s}}]$ in the form
\begin{align}
G_{n+1,k}(t)
&
=
\left(1 - \frac{k}{n+1}\right) \left(G_{n,k}(t) \frac{n+r}{n+r+s} + G_{n,k}(t)\frac{st}{n+r+s}\right) \label{eq:lah_distr_gen_funct_recurr_1}
\\
&+
t \cdot \frac{k}{n+1} \left(G_{n,k-1}(t) \frac{k-1+r+s}{n+r+s} + G_{n,k}(t) \frac{n-k+1}{n+r+s}\right).  \label{eq:lah_distr_gen_funct_recurr_2}
\end{align}
We show that the generating function of the number of edges in $\mathfrak{T}_{n,k}$ satisfies the same recurrence relation as $G_{n,k}$. Our argument  generalizes the one in~\cite{van_der_hofstad_etal_shortest_path_trees}. Consider the tree $\Hoppe{n+1}{r+s}$. With probability $k/(n+1)$, node $n+1$ belongs to $\mathfrak{A}_{n+1,k}$. Denote this event by $\mathfrak{Q}$. Given $\mathfrak{Q}$, there two possibilities:

\vspace{2mm}
\noindent
\emph{Case 1.} If the node $n+1$ is attached to a node from the set $\mathfrak{A}_{n,k-1}$ or to the root with weight $r+s$, then the tree $\mathfrak{T}_{n+1,k}$ differs from $\mathfrak{T}_{n,k-1}$ by one edge containing the node $n+1$. This additional edge accounts for a factor $t$ in~\eqref{eq:lah_distr_gen_funct_recurr_2}. The conditional (given $\mathfrak{Q}$) probability of the event occurring in Case 1 is $(k-1+r+s)/(n+r+s)$.

\vspace{2mm}
\noindent
\emph{Case 2.} If the node $n+1$ is attached to some node $v$ from the set $\{1,\dots, n\}\setminus \mathfrak{A}_{n,k}$, then the tree $\mathfrak{T}_{n+1,k}$ can be constructed as follows. Construct the minimal tree containing the root and the set $\mathfrak{A}_{n,k-1} \cup\{v\} \cup \mathfrak{C}_{n,k}$ and add to this tree the edge connecting $n+1$ and $v$. Then, 
$\mathfrak{T}_{n+1,k}$ differs by one edge from the tree having the same law as $\mathfrak{T}_{n,k}$, which again gives a factor $t$. The conditional (given $\mathfrak{Q}$) probability of the event occurring in Case 2 is $(n-k+1)/(n+r+s)$. Altogether, these considerations explain the term~\eqref{eq:lah_distr_gen_funct_recurr_2}.

\vspace{2mm}

Now, it is possible that the node $n+1$ does not belong to $\mathfrak{A}_{n+1,k}$, which happens with probability $\P[\mathfrak{Q}^c]=1 - k/(n+1)$. Given $\mathfrak{Q}^c$, there are again two possibilities:

\vspace{2mm}
\noindent
\emph{Case 3.} If the node $n+1$ is attached to the weight $s$ component of the root, then it belongs to $\mathfrak{C}_{n,k}$ and the tree $\mathfrak{T}_{n+1,k}$ differs from $\mathfrak{T}_{n,k}$ by one edge. The conditional probability (given $\mathfrak{Q}^c$) of this event is $s/(n+r+s)$.

\vspace{2mm}
\noindent
\emph{Case 4:} If the node $n+1$ is attached to any node in $\{1,\dots, n\}$ or to the weight $r$ component of the root, then $\mathfrak{T}_{n+1,k} = \mathfrak{T}_{n,k}$. The conditional probability (given $\mathfrak{Q}^c$) of this event is $(n+r)/(n+r+s)$. Altogether, Cases 3 and 4 explain~\eqref{eq:lah_distr_gen_funct_recurr_1}.

It remains to note that on the boundary $k=0$ the generating function of the number of edges in $\mathfrak{T}_{n,0}$ also coincides with $G_{n,0}(t)=\E [t^{\Lah{n}{0}{r}{s}}]$ as we have seen in Example~\ref{example:hoppe_lag_boundary_k=0}.
\end{proof}

\subsection{Stochastic representation in terms of random compositions}

The next proposition provides a representation of the $(r,s)$-Lah distribution as a sum of conditionally independent random variables. It will be the starting point in our asymptotic analysis of the $(r,s)$-Lah distribution.
\begin{proposition}\label{prop:r_lah_representation}
Let $(b_0^{(n,k)}, b_1^{(n,k)}, \dots, b_k^{(n,k)})$ be a random $(r+s)$-composition of $n$. Further, let $(Z^{(n,0)}_{r,s}, Z^{(n,1)}_{r,s}, \dots, Z^{(n,k)}_{r,s})$ be a random vector such that, conditionally on the event $\{b_0^{(n,k)} = b_0, b_1^{(n,k)}= b_1, \dots, b_k^{(n,k)}= b_k\}$, where $(b_0,b_1,\dots, b_k)$ is an arbitrary fixed incomplete composition of $n$, the random variables
$$
Z^{(n,0)}_{r,s}\sim  \rStirone{b_0}{s}{r},
\quad
Z^{(n,1)}_{r,s}\sim \Stirone{b_1}{1}, \quad \dots,\quad  Z^{(n,k)}_{r,s}\sim \Stirone{b_k}{1}
$$
are independent, that is,
$$
\P\left[ Z^{(n,j)}_{r,s} = \ell \,\Big|\, b_0^{(n,k)} = b_0, b_1^{(n,k)}= b_1, \dots, b_k^{(n,k)}= b_k\right] = \P\left[\sum_{m=1}^{b_j} \Bern{1/m} = \ell\right]
$$
for all $j = 1,\dots, k$ and 
$$
\P\left[ Z^{(n,0)}_{r,s} = \ell \,\Big|\, b_0^{(n,k)} = b_0, b_1^{(n,k)}= b_1, \dots, b_k^{(n,k)}= b_k\right] = \P\left[\sum_{m=1}^{b_0} \Bern{s/(r+s+m-1)} = \ell\right].
$$
Here, all Bernoulli random variables are assumed 
independent. Then,
$$
Z^{(n,0)}_{r,s} + Z^{(n,1)}_{r,s} + \dots + Z^{(n,k)}_{r,s}~\sim~\Lah{n}{k}{r}{s}.
$$
\end{proposition}

\begin{example}
Taking $k=0$ gives $\Lah{n}{0}{r}{s}=\rStirone{n}{s}{r}$.
\end{example}

\begin{example}
For $r=s=0$, Proposition~\ref{prop:r_lah_representation} reduces to the representation of the Lah distribution given in~\cite[Proposition 2.3]{kabluchko_marynych_lah}. 
\end{example}

\begin{example}[Cumulative degree of roots in the multinomial Hoppe tree]
Consider a random forest $\MultiHoppe{n-k}{\theta}{(r+s)/\theta, 1/\theta,\dots, 1/\theta}$ with $\rt_0$ having weight $r+s$ and 
$k$ roots $\rt_1,\dots, \rt_k$ having weight $1$ each. The total weight of all roots is $\theta = k+r+s$. According to Proposition~\ref{prop:coupling1} and Proposition~\ref{prop:multihoppe_sizes}, the vector encoding the sizes of connected components (not counting $\rt_0$ but counting $\rt_1,\dots, \rt_k$) has the same distribution as a random $(r+s)$-composition of $n$. Let $D_{n;j}$ be the degree of $\rt_j$, for 
$j=0,\dots, k$. Then, Proposition~\ref{prop:r_lah_representation} implies that
$$
D_{n;1} + \dots +  D_{n;k} + \Bin{D_{n;0}}{s/(r+s)} \sim \Lah{n}{k}{r}{s}.
$$
For $r=0$ the right-hand side simplifies to $D_{n;0} + D_{n;1} + \dots + D_{n,k}$.
\end{example}

\begin{example}
For $r=s=1/2$, the conditional distribution (as in Proposition \ref{prop:r_lah_representation}) of  $Z^{(n,j)}_{r,s}$ with $j\in \{1,\dots, k\}$ is the same as $\sum_{m=1}^{b_j} \Bern{1/m}$, which is the distribution of the number of records in $b_j$ i.i.d.\ observations from a continuous distribution, while the conditional distribution of $Z^{(n,0)}_{r,s}$ is the same as $\sum_{m=1}^{b_0} \Bern{1/(2m)}$. The distribution of the random $1$-composition is uniform on the set of all incomplete compositions of $n$. 
Motivated by the analysis of positive hulls of random walks carried out in~\cite{godland:2022} (see also~\cite{godlandschlaefli:2022}) the case $r=s$ has already been investigated in~\cite{kabluchko_steigenberger_r_lah}.
\end{example}

\begin{proof}[Proof of Proposition~\ref{prop:r_lah_representation}]
Recall from Lemma~\ref{lem:r_lah_generating_poly} and~\eqref{eq:r_Lah_def1} that
\begin{equation}
G_{n,k}(t)=\E [t^{\Lah{n}{k}{r}{s}}] =
\frac{1}{\binom{n+r+s-1}{k+r+s-1}} [x^n]\left(\bigl((1-x)^{-t}-1\bigr)^k (1-x)^{-(st+r)}\right).
\end{equation}
From the definition of the random variables $Z^{(n,0)}_{r,s}, Z^{(n,1)}_{r,s}, \dots, Z^{(n,k)}_{r,s}$ it follows that
$$
\E\left[ t^{Z^{(n,j)}_{r,s}}\,\Big|\, b_0^{(n,k)} = b_0, b_1^{(n,k)}= b_1, \dots, b_k^{(n,k)}= b_k\right] = \frac{t^{b_j\uparrow}}{b_j!},
\qquad
j\in \{1,\dots, k\} 
$$
and 
$$
\E\left[ t^{Z^{(n,0)}_{r,s}}\,\Big|\, b_0^{(n,k)} = b_0, b_1^{(n,k)}= b_1, \dots, b_k^{(n,k)}= b_k\right] = \frac{(st+r)^{b_0 \uparrow}}{(r+s)^{b_0 \uparrow}}.
$$
By conditional independence, this gives
\begin{equation}\label{eq:proof_repr_cond_expect}
\E\left[ t^{Z^{(n,0)}_{r,s} + Z^{(n,1)}_{r,s}+ \dots + Z^{(n,k)}_{r,s}}\,\Big|\, b_0^{(n,k)} = b_0, b_1^{(n,k)}= b_1, \dots, b_k^{(n,k)}= b_k\right]
=
\frac{t^{b_1\uparrow}}{b_1!} \cdot \dots \cdot\frac{t^{b_k\uparrow}}{b_k!}\cdot  \frac{(st+r)^{b_0 \uparrow}}{(r+s)^{b_0 \uparrow}}.
\end{equation}
Also, by the definition of the random $(r+s)$-composition (Definition \ref{def:r_composition}),
\begin{equation}\label{eq:proof_repr_probab_comp}
\P[b_0^{(n,k)} = b_0, b_1^{(n,k)}= b_1, \dots, b_k^{(n,k)}= b_k] = \frac{(r+s)^{b_0\uparrow}/b_0!}{\binom{n+r+s-1}{k+r+s-1}}.
\end{equation}
Using the formulas
$$
(1-x)^{-t}-1 = \sum_{b=1}^\infty \frac {x^b}{b!} t^{b\uparrow}
\quad
\text{ and }
\quad
(1-x)^{-(st+r)} = \sum_{b_0=0}^\infty \frac {x^{b_0}}{b_0!} (st + r)^{b_0\uparrow},
$$
we obtain
\begin{align*}
G_{n,k}(t)
&=
\frac{1}{\binom{n+r+s-1}{k+r+s-1}} \sum_{(b_0,b_1,\dots, b_k)} \frac{t^{b_1\uparrow}}{b_1!} \dots \frac{t^{b_k\uparrow}}{b_k!} \frac{(st+r)^{b_0 \uparrow}}{b_0!}\\
&= \sum_{(b_0,b_1,\dots, b_k)} \frac{t^{b_1\uparrow}}{b_1!}\cdot \dots \cdot\frac{t^{b_k\uparrow}}{b_k!} \frac{(st+r)^{b_0 \uparrow}}{(r+s)^{b_0 \uparrow}} \cdot \frac{(r+s)^{b_0\uparrow}/b_0!}{\binom{n+r+s-1}{k+r+s-1}},
\end{align*}
where both sums are taken over all incomplete compositions of $n$.  Recognizing on the right-hand side the conditional expectation of $t^{Z^{(n,0)}_{r,s} + Z^{(n,1)}_{r,s}+ \dots + Z^{(n,k)}_{r,s}}$ given $\{b_0^{(n,k)} = b_0, b_1^{(n,k)}= b_1, \dots, b_k^{(n,k)}= b_k\}$, see~\eqref{eq:proof_repr_cond_expect}, and the probability of this event, see~\eqref{eq:proof_repr_probab_comp},  we 
apply the total expectation formula to obtain 
$$
G_{n,k}(t)=\E [t^{\Lah{n}{k}{r}{s}}] = \E [t^{Z^{(n,0)}_{r,s} + Z^{(n,1)}_{r,s}+ \dots + Z^{(n,k)}_{r,s}}].
$$
The proof is complete. 
\end{proof}

Using Corollary~\ref{cor:without_b0_uniform} and recalling the representation of the standard ($r=s=0$) Lah distribution given in~\cite[Proposition 2.3]{kabluchko_marynych_lah} 
we obtain the following.

\begin{corollary}\label{cor:reduction_to_r=s=0}
Let $(b_0^{(n,k)},b_1^{(n,k)},\dots,b_k^{(n,k)})$ be a random $(r+s)$-composition and $(\Lah{m}{k}{0}{0})_{m\geq k}$ 
a sequence of random variables having the standard Lah distributions, which is independent of $(Z^{(n,0)}_{r,s},b_0^{(n,k)})$, where $Z^{(n,0)}_{r,s}$ is as defined in Proposition~\ref{prop:r_lah_representation}. Then,
$$
\Lah{n}{k}{r}{s}~~~\text{has the same distribution as}~~~
Z^{(n,0)}_{r,s}+\Lah{n-b_0^{(n,k)}}{k}{0}{0}.
$$

\end{corollary}

Corollary~\ref{cor:reduction_to_r=s=0} is designed to derive 
limit theorems for $\Lah{n}{k}{r}{s}$ from the known limit theorems for $\Lah{n}{k}{0}{0}$.

\subsection{Expectation and variance}
We employ the 
notation 
$$
L 
(n,k,r,s):=\E [\Lah{n}{k}{r}{s}].
$$
The following explicit formula can be found in 
Theorem 3.1 of~\cite{kabluchko_steigenberger_r_lah}
\begin{equation}\label{eq:expectation_r=s=0}
L 
(n,k,0,0) = \E [\Lah{n}{k}{0}{0}] = \frac{nk}{n-(k-1)} (H_{n} - H_{k-1}),
\end{equation}
where $H_n:=1+1/2+\dots+1/n$ is the $n$-th harmonic number. For $r,s\geq 0$, put $H_0^{(r,s)}:=0$ and, for $n\in\N$, $H_n^{(r,s)}:=\sum_{m=1}^n s(r+s+m-1)^{-1}$. Plainly, $H_n^{(r,0)}=0$ for all $n\in\N_0$. Corollary~\ref{cor:reduction_to_r=s=0} immediately yields
$$
L
(n,k,r,s)=\E [H_{b_0^{(n,k)}}^{(r,s)}]+\E [L(n-b_0^{(n,k)},k,0,0)],
$$
where $(b_0^{(n,k)},\dots,b_k^{(n,k)})$ is a random $(r+s)$-composition. However, this formula is of little use for practical purposes. Note that Proposition~\ref{prop:r_lah_representation} implies that 
$$
L(n,k,r,s)=\E [H_{b_0^{(n,k)}}^{(r,s)}]+k\E [H_{b_j^{(n,k)}}].
$$
We do 
not derive exact formulas for either $L(n,k,r,s)$ or the variance, although it can be done relatively straightforwardly by generalizing the proofs given in~\cite[Theorems~3.1, 3.7]{kabluchko_steigenberger_r_lah}. Instead, our focus is 
on asymptotic estimates, which are 
crucial for proving limit theorems for the Lah distributions in subsequent sections.
\begin{lemma}
For $r,s\geq 0$, the following estimate holds true
\begin{equation}\label{eq:estimateH}
\Big|H_n^{(r,s)}-s\log\Big(\frac{r+s+(n-1)_+}{r+s}\Big)\Big|\leq \frac{s}{r+s},\quad n\in\N_0.
\end{equation}
\end{lemma}
\begin{proof}
For $n\in\N_0$, 
$$
H_n^{(r,s)}=\sum_{m=1}^n \frac{s}{r+s+m-1}\leq 
\frac{s}{r+s}+\int_0^{(n-1)_+}\frac{s\,{\rm d}x}{r+s+x} 
=\frac{s}{r+s}+s\log\Big(\frac{r+s+(n-1)_+}{r+s}\Big).
$$
This together with an analogous estimate from below completes the proof of~\eqref{eq:estimateH}.
\end{proof}

\begin{lemma}\label{lem:meanvar}
Fix some $r,s\geq 0$ and let $(b_0^{(n,k)},\dots,b_k^{(n,k)})$ be a random $(r+s)$-composition. Assume that $n,k\to\infty$ such that $k/n \to 0$. Then, for each 
fixed $a\geq 1$,
\begin{equation}\label{lem:meanvar_a}
\E\Big[\big(H_{b_0^{(n,k)}}^{(r,s)}\big)^a+\sum_{j=1}^k \big(H_{b_j^{(n,k)}}\big)^a \Big]~\simeq~\E\Big[\sum_{j=1}^k \big(H_{b_j^{(n,k)}}\big)^a \Big]~\simeq~k(\log (n/k))^a,\quad n,k\to\infty.
\end{equation}
In particular, 
\begin{equation}\label{lem:meanvar_a1}
L(n,k,r,s)\simeq \E\Big[\sum_{j=1}^k H_{b_j^{(n,k)}}\Big]\simeq k\log (n/k),\quad n\to\infty.
\end{equation}
Furthermore,
\begin{equation}\label{lem:meanvar_b}
{\rm Var}\,\Big[\sum_{j=1}^k H_{b_j^{(n,k)}}\Big]=O(k),\quad n\to\infty
\end{equation}
and
\begin{equation}\label{lem:meanvar_c}
{\rm Var}\,\Big[H_{b_0^{(n,k)}}^{(r,s)}+\sum_{j=1}^k H_{b_j^{(n,k)}}\Big]=O(k),\quad n\to\infty.
\end{equation}
\end{lemma}
\begin{proof}
When dealing with $H_{b_0^{(n,k)}}^{(r,s)}$ we always assume that $s>0$, for $H_{b_0^{(n,k)}}^{(r,0)}=0$ a.s.

For a proof of~\eqref{lem:meanvar_a} it suffices to check that $\E [(H_{b_1^{(n,k)}})^a]\simeq (\log (n/k))^a$ and $\E [(H_{b_0^{(n,k)}}^{(r,s)})^a]=O(\log (n/k))$ as $n,k\to\infty$. In view of~\eqref{eq:r_composition_marginal_distr_asymptotic_alpha_0},
$$
\frac{\log b_1^{(n,k)}}{\log(n/k)}~\overset{\P}{\to}~1,\quad n,k\to\infty.
$$
This in combination with $|H_n-\log n|\leq C$ for all $n\in \N$ and a constant $C>0$ entails
$$
\frac{(H_{b_1^{(n,k)}})^a}{(\log(n/k))^a}~\overset{\P}{\to}~1,\quad n,k\to\infty
$$
and thereupon by Fatou's lemma ${\liminf}_{n,k\to\infty}(\E [(H_{b_1^{(n,k)}})^a]/(\log(n/k))^a)\geq 1$. To prove the converse inequality for the upper limit, note that $x\mapsto (\log x + C)^{a}$ is concave on $[1,\infty)$ for each 
$C>a-1$. In view of this we take 
$C$ sufficiently large and write with the help of Jensen's inequality
\begin{multline*}
\E [(H_{b_1^{(n,k)}})^a]\leq \E [(\log b_1^{(n,k)}+C)^a]\leq (\log \E [b_1^{(n,k)}]+C)^a\\
=(\log \E [(n-b_0^{(n,k)})/k]+C)^a\leq (\log (n/k)+C)^a.
\end{multline*}
Using~\eqref{eq:estimateH}, we conclude that
\begin{equation*}
(H_{n}^{(r,s)})^a\leq 2^{a-1}\left(\left(\frac{s}{r+s}\right)^a+s^a\log^a\Big(\frac{r+s+(n-1)_+}{r+s}\Big)\right).
\end{equation*}
Further, Jensen's inequality and formula~\eqref{eq:r_composition_expectation} with $r$ replaced by $r+s$ yields
\begin{align}
\E [(H_{b_0^{(n,k)}}^{(r,s)})^a]-2^{a-1}\left(\frac{s}{r+s}\right)^a &\leq 2^{a-1}s^a\E\left[\Big(\log\Big(\frac{r+s+(b_0^{(n,k)}-1)_+}{r+s}\Big)+C\Big)^a\right]\notag\\
&\leq 2^{a-1}s^a\left[\Big(\log\Big(\frac{r+s+\E [b_0^{(n,k)}]}{r+s}\Big)+C\Big)^a\right]\notag\\
&=2^{a-1}s^a\left[\Big(\log\Big(1+\frac{n-k}{k+r+s}\Big)+C\Big)^a\right].\label{eq:Z_n_0_mean_estimate}
\end{align}
The right-hand side is $O((\log (n/k))^a)$. The proof of~\eqref{lem:meanvar_a} is complete.

We now turn to the proof of~\eqref{lem:meanvar_b} and~\eqref{lem:meanvar_c}. The functions $x\mapsto H_{\lfloor x\rfloor}^{(r,s)}$ and $x\mapsto H_{\lfloor x\rfloor}$ are nondecreasing on $[0,\infty)$ and $[1,\infty)$, respectively. An application of Corollary~\ref{cor:assoc}, with $r$ replaced by $r+s$, then yields, for $j\in\{1,\dots, k\}$,
$$
\E\big[H_{b_0^{(n,k)}}^{(r,s)}H_{b_j^{(n,k)}}\big]\leq \E \big[H_{b_0^{(n,k)}}^{(r,s)}\big]\E \big[H_{b_j^{(n,k)}}\big]
$$
and, for $i,j\in\{1,\dots,k\}$, $i\neq j$,
$$
\E\big[H_{b_i^{(n,k)}}H_{b_j^{(n,k)}}\big]\leq \E \big[H_{b_i^{(n,k)}}\big]\E \big[H_{b_j^{(n,k)}}\big].
$$
Thus, both~\eqref{lem:meanvar_b} and~\eqref{lem:meanvar_c} follow provided we can prove
\begin{equation}\label{eq:uniform}
{\rm Var}\,H_{b_0^{(n,k)}}^{(r,s)}=O(1)\quad\text{and}\quad {\rm Var}\,H_{b_1^{(n,k)}}=O(1),\quad n,k\to\infty.
\end{equation}
To check this, we show that, for each 
$p\geq 1$,
\begin{multline}\label{eq:p_moment_bounded}
\sup_{n,k\in \N, k\leq n} \E \left|\log \big(r+s+(b_0^{(n,k)}-1)_+\big)- \log(n/k)\right|^p <\infty\quad\text{and}\\
\sup_{n,k\in \N, k\leq n} \E \left|\log b_1^{(n,k)}- \log(n/k)\right|^p <\infty.
\end{multline}
Our argument 
is a slight adaptation of the proof of formula (5.33) in~\cite{kabluchko_marynych_lah}. We start by proving the second inequality in~\eqref{eq:p_moment_bounded}. In view of $|x|=x_++x_-$ for $x\in\mathbb{R}$ and $(x+y)^p\leq 2^{p-1}(x^p+y^p)$ for $x,y\geq 0$ it suffices to show that
$$
\sup_{n,k\in\N,k\leq n}\E[\log_+((k/n)b_1^{(n,k)})]^p<\infty\quad\text{and}\quad\sup_{n,k\in\N, k\leq n}\E[\log_-((k/n)b_1^{(n,k)})]^p<\infty.
$$
The former inequality follows from $[\log_+ x]^p=O(x)$ as $x\to\infty$ and the estimate $(k/n)\E[b_1^{(n,k)}]=(n-\E[b_0^{(n,k)}])/n\leq 1$. To prove the latter inequality, we first conclude with the help of~\eqref{eq:r_composition_marginal_distr_b_j} that
$$
\frac{\P\big[b_1^{(n,k)} = j\big]}{\P\big[b_1^{(n,k)} =j+1\big]}=1+\frac{k+r+s-2}{n-j-k+1}>1
$$
for all $k> (2-r-s)_{+}$ and all $j\in\{1,\dots, n-k\}$. Thus, the sequence $j\mapsto \P\big[b_1^{(n,k)} = j\big]$ is strictly decreasing. This entails
\begin{multline*}
\E [\log_-((k/n)b_1^{(n,k)})]^p=\sum_{j=1}^{\lfloor n/k\rfloor}[\log_-(jk/n)]^p\P[b_1^{(n,k)}=j]\leq \P[b_1^{(n,k)}=1]\sum_{j=1}^{\lfloor n/k \rfloor}[\log_-(jk/n)]^p\\=\frac{k+q-1}{n+q-1}\sum_{j=1}^{\lfloor n/k\rfloor}
[\log_-(jk/n)]^p~\to~\int_0^1 |\log x|^p{\rm d}x<\infty,\quad n,k\to\infty.
\end{multline*}
The proof of the second inequality in~\eqref{eq:p_moment_bounded} is complete.

Turning to the first inequality in~\eqref{eq:p_moment_bounded}, we conclude that it is enough to show that
\begin{multline*}
\sup_{n,k\in\N,k\leq n}\E[\log_+((k/n)(r+s+(b_0^{(n,k)}-1)_+)]^p<\infty\quad \text{and}\\
\sup_{n,k\in\N, k\leq n}\E[\log_-((k/n)(r+s+(b_0^{(n,k)}-1)_+))]^p<\infty.
\end{multline*}
The first of these is secured by
\begin{multline*}
(k/n)\E[r+s+(b_0^{(n,k)}-1)_+]\leq r+s+ (k/n)\E[b_0^{(n,k)}]\\
=r+s+(k/n)((r+s)(n-k)/(k+r+s))\leq 4(r+s).
\end{multline*}
We have used~\eqref{eq:r_composition_expectation} for the equality. As for the second, we write
\begin{multline*}
\E[\log_-((k/n)(r+s+(b_0^{(n,k)}-1)_+))]^p=[\log_-((r+s)(k/n))]^p\P[(b_0^{(n,k)}-1)_+=0]\\+\sum_{j=1}^{\lfloor n/k-(r+s)\rfloor}[\log_-((k/n)(r+s+j))]^p\P[b_0^{(n,k)}=j+1]=:I_1(n,k)+I_2(n,k).
\end{multline*}
In view of~\eqref{eq:r_composition_marginal_distr_b_0}, $\P[(b_0^{(n,k)}-1)_+=0]\simeq (r+s+1)(k/n)^{r+s}$ as $n,k\to\infty$. This implies that the first summand $I_1$ vanishes. Now we analyze $I_2$. Since
$$
\frac{(n-j-1)!}{(n-j-k)!}\frac{(n-j-1-k)!}{(n-j-2)!}=\frac{n-j-1}{n-j-k}>1
$$
for all $k\geq 2$ and all $j,n\in\N$ satisfying $n-j-k\geq 1$, we conclude that the sequence $j\mapsto \binom{n-j-1}{k-1}$ is decreasing. Thus, for $j\geq 1$,
$$
\frac{\binom{n-j-2}{k-1}}{\binom{n+r+s-1}{k+r+s-1}}\leq \frac{\binom{n-3}{k-1}}{\binom{n+r+s-1}{k+r+s-1}}~\simeq~\Big(\frac{k}{n}\Big)^{r+s},\quad n,k\to\infty.
$$
By a standard estimate for the gamma function, there exists a constant $c_{r,s}>0$ such that, for all $j\geq 1$, $$\binom{r+s+j}{1+j}\leq c_{r,s} j^{r+s-1}.
$$
Combining fragments together and recalling~\eqref{eq:r_composition_marginal_distr_b_0} we infer
$$
I_2(n,k)=O\Big(\frac{k}{n}\Big)\sum_{j=1}^{\lfloor n/k-(r+s)\rfloor}[\log_-((k/n)(r+s+j))]^p \Big(\frac{kj}{n}\Big)^{r+s-1}=O(1),\quad n,k\to\infty.
$$
The last equality is justified by a Riemann approximation $$\lim_{n,k\to\infty}\frac{k}{n}\sum_{j=1}^{\lfloor n/k-(r+s)\rfloor}[\log_-((k/n)(r+s+j))]^p \Big(\frac{kj}{n}\Big)^{r+s-1}=\int_0^1 |\log x|^p x^{r+s-1}{\rm d}x<\infty.
$$
This completes the proof of~\eqref{eq:p_moment_bounded}.

We 
only show how to obtain the first relation in~\eqref{eq:uniform}, the argument for the second is analogous. In view of~\eqref{eq:estimateH} and the first inequality in~\eqref{eq:p_moment_bounded}, $$
\big|\E [H_{b_0^{(n,k)}}^{(r,s)}]- s\log(n/k)\big|=O(1)\quad\text{and}\quad \E\big[H_{b_0^{(n,k)}}^{(r,s)}- s\log (n/k)\big]^2=O(1),\quad n,k\to\infty.
$$
The latter entails
$$
\E\big[H_{b_0^{(n,k)}}^{(r,s)}\big]^2=2 s\log(n/k)\E [H_{b_0^{(n,k)}}^{(r,s)}]-(s\log(n/k))^2+O(1),\quad n,k\to\infty,
$$
whence
$$
{\rm Var}\,[H_{b_0^{(n,k)}}^{(r,s)}]=\E\big[H_{b_0^{(n,k)}}^{(r,s)}\big]^2-\big(\E\big[H_{b_0^{(n,k)}}^{(r,s)}\big]\big)^2=-(\E H_{b_0^{(n,k)}}^{(r,s)}- s\log (n/k))^2+O(1)=O(1),\quad n,k\to\infty.$$
The proof is complete.
\end{proof}
\begin{remark}\label{rem:expect}
A perusal of the proof of~\eqref{lem:meanvar_a} reveals that, for fixed $k$,
$$
\E\Big[H_{b_0^{(n,k)}}\Big]~=~s\log n+O(1),\quad n\to\infty,
$$
whence
$$
L(n,k,r,s)~=~(k+s)\log n +O(1),\quad n\to\infty.
$$
\end{remark}

\section{Limit theorems for the \texorpdfstring{$(r,s)$}{(r,s)}-Lah distribution}\label{sec:Lah_limits}
In this concluding section we prove central limit theorems for the $(r,s)$-Lah distribution. Not only do these theorems generalize the special cases $r=s=0$ and $r=s$ investigated 
in~\cite{kabluchko_marynych_lah} and~\cite{kabluchko_steigenberger_r_lah}, but they also treat asymptotic regimes not covered in those 
papers. Let  $\Normal{0}{\sigma^2}$ denote a normal random variable with zero mean 
and variance $\sigma^2>0$. Throughout this section $(b_0^{(n,k)},b_1^{(n,k)},\dots,b_k^{(n,k)})$ denotes the random $(r+s)$-composition.
\begin{theorem}[CLT in the constant $k$ regime]
\label{thm:clt_lah_r_k_const}
Let $k\in\N_0$ and $r,s \in [0, \infty)$ be fixed and such that $\max\{k,s\} > 0$. Then,
\begin{align*}
\frac{\Lah{n}{k}{r}{s} - (k+s)\log n}{\sqrt{(k+s)\log n}} \todistr \Normal{0}{1}.
\end{align*}
\end{theorem}
\begin{proof}
We assume that $r+s>0$. The proof hinges on Corollary~\ref{cor:reduction_to_r=s=0} and a known CLT in the case of constant $k$ and $r=s=0$, see Theorem 4.2 in~\cite{kabluchko_marynych_lah}. A classic functional limit theorem for independent Bernoulli variables together with Theorem 4.2 in~\cite{kabluchko_marynych_lah} imply
\begin{multline}\label{eq:thm:clt_lah_r_k_const_proof1}
\left(\left(\frac{\sum_{m=1}^{\lfloor n^t\rfloor}\Bern{s/(r+s+m-1)}-ts\log n}{\sqrt{s\log n}}\right)_{t\geq 0},\frac{\Lah{n-b_0^{(n,k)}}{k}{0}{0}-k\log (n-b_0^{(n,k)})}{\sqrt{k\log (n-b_0^{(n,k)})}}\right)\\
\todistr~\left((B(t))_{t\geq 0},\Normal{0}{1}\right)
\end{multline}
with the components on both sides being independent and $(B(t))_{t\geq 0}$ being a 
standard Brownian motion. According to 
Proposition~\ref{prop:r_composition_marginal_convergence_fixed_k}, $b_0^{(n,k)}/n \todistr {\rm Beta}\,[r+s; k]$, whence 
\begin{equation}\label{eq:thm:clt_lah_r_k_const_proof2}
\frac{\log_{+} b_0^{(n,k)}}{\log n}~\overset{\P}{\to}~1\quad \text{and}\quad\frac{\log (n-b_0^{(n,k)})-\log n}{\sqrt{\log n}}~\overset{\P}{\to}~0,\quad n\to\infty.
\end{equation}
Using continuity of the composition operation and Slutsky's lemma we conclude from~\eqref{eq:thm:clt_lah_r_k_const_proof1} and~\eqref{eq:thm:clt_lah_r_k_const_proof2} that
\begin{equation}\label{eq:thm:clt_lah_r_k_const_proof3}
\left(\frac{Z^{(n,0)}_{r,s}-s\log_{+} b_0^{(n,k)}}{\sqrt{\log n}}, 
\frac{\Lah{n-b_0^{(n,k)}}{k}{0}{0}-k\log n}{\sqrt{\log n}}\right)~\todistr~\left(\sqrt{s}B(1),\sqrt{k}\Normal{0}{1}\right).
\end{equation}
By adding the coordinates and using that, see Proposition~\ref{prop:r_composition_marginal_convergence_fixed_k},
$$
\frac{\log_{+} b_0^{(n,k)}-\log n}{\sqrt{\log n}}~\overset{\P}{\to}~0,\quad n\to\infty,
$$
we obtain the desired statement. The proof is complete.
\end{proof}

Theorem~\ref{thm:clt_lah_r_k_const} was previously known in case $r=s>0$  from~\cite{kabluchko_steigenberger_r_lah} and in case $r=s=0$ from~\cite{kabluchko_marynych_lah}. The next theorem is new even in the standard scenario $r=s=0$.

\begin{theorem}[CLT in the intermediate regime]\label{thm:clt_lah_r_interm}
Fix some $r,s\geq 0$ and let $n\to\infty$, $k\to\infty$ such that $k/n \to 0$. Then,
\begin{align*}
\frac{\Lah{n}{k}{r}{s} - L(n,k,r,s)}{\sqrt{k\log (n/k)}}~\todistrk~\Normal{0}{1}.
\end{align*}
\end{theorem}
\begin{proof}[Proof of Theorem~\ref{thm:clt_lah_r_interm}]
According to Proposition~\ref{prop:r_lah_representation}, it is sufficient to show that $$\frac{Z_{r,s}^{(n,0)}+\dots+Z_{r,s}^{(n,k)}-L(n,k,r,s)}{\sqrt{k\log (n/k)}}~\todistrk~\Normal{0}{1}. 
$$
For each $n\in\N$ and $k\in\N$, denote by $\mathcal{G}_{n,k}$ the $\sigma$-algebra generated by $b_0^{(n,k)}$, $b_1^{(n,k)},\dots, b_k^{(n,k)}$. We first prove that, conditionally on $\mathcal{G}_{n,k}$,
\begin{equation}\label{eq:conditional}
\frac{\sum_{j=0}^k \big(Z_{r,s}^{(n,j)}-\E \big[Z_{r,s}^{(n,j)}\big|\mathcal{G}_{n,k}\big]\big)}{\sqrt{{\rm Var}\,[\sum_{j=0}^k Z_{r,s}^{(n,j)}|\mathcal{G}_{n,k}]}}~\todistrk~\Normal{0}{1}. 
\end{equation}
Conditionally on $\mathcal{G}_{n,k}$, the left-hand side is a sum of independent centered random variables with finite third absolute moments. In view of this, we will prove the last limit relation by an application of the Berry--Esseen inequality:
\begin{multline*}
I_{n,k}:=\sup_{x\in\R}\left|\P\left[\frac{\sum_{j=0}^k \big(Z_{r,s}^{(n,j)}-\E \big[Z_{r,s}^{(n,j)}\big|\mathcal{G}_{n,k}\big]\big)}{\sqrt{{\rm Var}\,[\sum_{j=0}^k Z_{r,s}^{(n,j)}|\mathcal{G}_{n,k}]}} \leq x\right]-\P[\Normal{0}{1}\leq x]\right|\\
\leq A\frac{\sum_{j=0}^k \E [|Z_{r,s}^{(n,j)}-\E \big[Z_{r,s}^{(n,j)}\big|\mathcal{G}_{n,k}\big]|^3|\mathcal{G}_{n,k}]}{({\rm Var}\,[\sum_{j=0}^k Z_{r,s}^{(n,j)}|\mathcal{G}_{n,k}])^{3/2}}
\end{multline*}
for a deterministic constant $A>0$ which does not depend on $n$, nor $k$. To complete the proof of~\eqref{eq:conditional} it is enough to show that
\begin{equation}\label{eq:inter1}
\frac{\sum_{j=0}^k \E [|Z_{r,s}^{(n,j)}-\E \big[Z_{r,s}^{(n,j)}\big|\mathcal{G}_{n,k}\big]|^3|\mathcal{G}_{n,k}]}{({\rm Var}\,[\sum_{j=0}^k Z_{r,s}^{(n,j)}|\mathcal{G}_{n,k}])^{3/2}}~\overset{\P}{\to}~0,\quad n,k\to\infty.
\end{equation}
As a preparation, we prove that
\begin{equation}\label{eq:inter}
\frac{{\rm Var}\,\big[\sum_{j=0}^k Z_{r,s}^{(n,j)}\big|\mathcal{G}_{n,k}\big]}{k\log(n/k)}~\overset{\P}{\to}~1,\quad n,k\to\infty.
\end{equation}
In what follows we assume that $s>0$, for otherwise $Z_{r,s}^{(n,0)}=0$ a.s. In view of \eqref{eq:estimateH}, 
\begin{multline*}
{\rm Var}\,\big[Z_{r,s}^{(n,0)}\big|\mathcal{G}_{n,k}\big]=\sum_{m=1}^{b_0^{(n,k)}}\frac{s}{r+s+m-1}\Big(1-\frac{s}{r+s+m-1}\Big)\leq H_{b_0^{(n,k)}}^{(r,s)}\\\leq 
s\log \Big(\frac{r+s+(b_0^{(n,k)}-1)_+}{r+s}\Big)+\frac{s}{r+s}.
\end{multline*}
As a consequence of~\eqref{eq:r_composition_marginal_distr_asymptotic_alpha_0}, $\log (b_0^{(n,k)}-1)_+/\log (n/k)\overset{\P}{\to} 1$ as $n,k\to\infty$. This shows that $$\frac{{\rm Var}\,\big[Z_{r,s}^{(n,0)}\big|\mathcal{G}_{n,k}\big]}{k\log(n/k)}~\overset{\P}{\to}~0,\quad n,k\to\infty.$$ Since, for $j\in\N$, $${\rm Var}\,\big[Z_{r,s}^{(n,j)}\big|\mathcal{G}_{n,k}\big]=\sum_{m=1}^{b_j^{(n,k)}}\frac{1}{m}\Big(1-\frac{1}{m}\Big)\in [H_{b_j^{(n,k)}}-\pi^2/6, H_{b_j^{(n,k)}}],$$ relation \eqref{eq:inter} is equivalent to $$\frac{\sum_{j=1}^k H_{b_j^{(n,k)}}}{k\log(n/k)}~\overset{\P}{\to}~1,\quad n,k\to\infty.$$ The latter is secured by 
Lemma \ref{lem:meanvar} and Chebyshev's inequality. Thus, for the proof of~\eqref{eq:inter1} it remains to check
\begin{equation}\label{eq:inter11}
\frac{\sum_{j=0}^k \E [|Z_{r,s}^{(n,j)}-\E \big[Z_{r,s}^{(n,j)}\big|\mathcal{G}_{n,k}\big]|^3|\mathcal{G}_{n,k}]}{(k\log (n/k))^{3/2}}~\overset{\P}{\to}~0,\quad n,k\to\infty.
\end{equation}
Invoking once again Proposition~\ref{prop:r_lah_representation} and using Rosenthal's inequality, see Theorem 9.1 on p.~152 in~\cite{Gut:2005}, we obtain for an absolute constant $B>0$ which does not depend on $n$ and $k$
\begin{multline*}
\E\big[\big|Z_{r,s}^{(n,j)}-\E \big[Z_{r,s}^{(n,j)}\big|\mathcal{G}_{n,k}\big]\big|^3\big|\mathcal{G}_{n,k}\big]\leq B \Big(\Big(\sum_{m=1}^{b_j^{(n,k)}}\E\left[\Bern{m^{-1}}-m^{-1}\right]^2\Big)^{3/2}\\+\sum_{m=1}^{b_j^{(n,k)}}\E\left[\big|\Bern{m^{-1}}-m^{-1}|\right]^3\Big)\leq B\big(\big(H_{b_j^{(n,k)}}\big)^{3/2}+H_{b_j^{(n,k)}}\big)\quad\text{a.s.}
\end{multline*}
for $j\geq 1$ and analogously
\begin{equation*}
\E\big[\big|Z_{r,s}^{(n,0)}-\E \big[Z_{r,s}^{(n,0)}\big|\mathcal{G}_{n,k}\big]\big|^3\big|\mathcal{G}_{n,k}\big]\leq B\big(\big(H_{b_0^{(n,k)}}^{(r,s)}\big)^{3/2}+H_{b_0^{(n,k)}}^{(r,s)}\big)\quad\text{a.s.}
\end{equation*}
Therefore,~\eqref{eq:inter11} is secured by Eq.~\eqref{lem:meanvar_a} with $a=3/2$ and Markov's inequality.

We have proved that \eqref{eq:conditional} holds conditionally on $\mathcal{G}_{n,k}$. By Lebesgue's dominated convergence theorem, it also holds unconditionally. Furthermore, in view of \eqref{eq:inter},
\begin{equation*}
\frac{\sum_{j=0}^k \big(Z_{r,s}^{(n,j)}-\E \big[Z_{r,s}^{(n,j)}\big|\mathcal{G}_{n,k}\big]\big)}{\sqrt{k\log(n/k)}}~\todistrk~\Normal{0}{1}.
\end{equation*}
Observe that
$$
L(n,k,r,s)= \E\Big[H^{(r,s)}_{b_0^{(n,k)}}+\sum_{j=1}^k H_{b_j^{(n,k)}}\Big],\quad \sum_{j=0}^k \E \big[Z_{r,s}^{(n,j)}\big|\mathcal{G}_{n,k}\big]=H_{b_0^{(n,k)}}^{(r,s)}+\sum_{j=1}^k H_{b_j^{(n,k)}}.
$$
The proof finishes by an appeal to Eq.~\eqref{lem:meanvar_c} in Lemma~\ref{lem:meanvar}, which ensures that $$\frac{\sum_{j=0}^k \E \big[Z_{r,s}^{(n,j)}\big|\mathcal{G}_{n,k}\big]-L(n,k,r,s)}{\sqrt{k\log (n/k)}}~\overset{\P}{\to}~0,\quad n,k\to\infty.$$
The proof is complete.
\end{proof}

\begin{theorem}[CLT in the central regime]\label{thm:clt_central_regime}
Fix some $r,s\geq 0$ and let $n\to\infty$, $k\to\infty$ such that $k/n \to \alpha$ for some $\alpha \in (0,1)$. Then,
$$
\frac{\Lah{n}{k}{r}{s} - L(n,k,0,0)}{\sqrt n}~\todistrk~\Normal{0}{-\left(\frac{\alpha}{1-\alpha} + \frac{\alpha(\alpha+1)\log \alpha}{(1-\alpha)^2} + \frac{\alpha^2 \log^2 \alpha}{(1-\alpha)^3}\right)}.
$$
\end{theorem}
\begin{proof}
Our proof relies on Corollary~\ref{cor:reduction_to_r=s=0} in conjunction with the fact that in the central regime $b_0^{(n,k)}$ converges in distribution to the negative binomial distribution; see Proposition~\ref{prop:r_composition_marginal_distr_asymptotic}.

We first check that
\begin{equation}\label{eq:thm:clt_central_regime_proof1}
\frac{\Lah{n}{k}{0}{0}-\Lah{n-b_0^{(n,k)}}{k}{0}{0}}{\sqrt{n}}~\overset{\P}{\to}~0,\quad n, k\to\infty.
\end{equation}
In view of Proposition 2.4 in~\cite{kabluchko_marynych_lah} (or by using the couplings from Section~\ref{subsec:couplings}) and conditional Markov's inequality it is sufficient to show that
\begin{equation}\label{eq:thm:clt_central_regime_proof11}
\frac{L(n,k,0,0)-L(n-b_0^{(n,k)},k,0,0)}{\sqrt{n}}~\overset{\P}{\to}~0,\quad n, k\to\infty.
\end{equation}
Indeed,~\eqref{eq:thm:clt_central_regime_proof11} implies that~\eqref{eq:thm:clt_central_regime_proof1} holds conditionally on $b_0^{(n,k)}$ and, by the Lebesgue dominated convergence theorem, also unconditionally.

Recall that $L(n,k,0,0)$ is given explicitly in Eq.~\eqref{eq:expectation_r=s=0}. Convergence~\eqref{eq:thm:clt_central_regime_proof11} follows immediately from the estimates, with $a\in\N_0$ fixed,
\begin{align*}\label{eq:thm:clt_central_regime_proof2}
0\leq L(n,k,0,0)-L(n-a,k,0,0)&\leq \frac{k}{n-k+1}\left(n(H_n-H_{k-1})-(n-a)(H_{n-a}-H_{k-1})\right)\\
&\leq\frac{k}{n-k+1}\left(2a H_n+n(H_n-H_{n-a})\right)\\
&\leq\frac{k}{n-k+1}\left(2a H_n+n\min\left(H_n,\frac{a}{n-a}\right)\right),
\end{align*}
and the fact that $b_0^{(n,k)}$ converges in distribution. The latter fact also ensures that
$$
\frac{Z^{(n,0)}_{r,s}}{\sqrt{n}}~\overset{\P}{\to}~0,\quad n\to\infty,
$$
see Proposition~\ref{prop:r_lah_representation}. Combining this with~\eqref{eq:thm:clt_central_regime_proof1} and appealing to Corollary~\ref{cor:reduction_to_r=s=0} shows that the claim of the theorem is equivalent to
$$
\frac{\Lah{n}{k}{0}{0} - L(n,k,0,0)}{\sqrt n}~\todistrk~\Normal{0}{-\left(\frac{\alpha}{1-\alpha} + \frac{\alpha(\alpha+1)\log \alpha}{(1-\alpha)^2} + \frac{\alpha^2 \log^2 \alpha}{(1-\alpha)^3}\right)}.
$$
But this is secured by Theorem 5.1 in~\cite{kabluchko_marynych_lah}. The proof is complete.
\end{proof}

The next result extends Theorem~\ref{thm:clt_central_regime} to the border case $\alpha=1$ in which the limiting variance in Theorem~\ref{thm:clt_central_regime} vanishes.

\begin{theorem}\label{thm:clt large k}
Fix some $r,s\geq 0$ and let $n\to\infty$, $k\to\infty$ such that $k/n \to 1$ and $n-k\to\infty$. Then,
\begin{align*}
\frac{\Lah{n}{k}{r}{s} - L(n,k,r,s)}{\sqrt{n-k}}~\todistrnk~\Normal{0}{1/4}.
\end{align*}
\end{theorem}
Before delving into a 
proof we briefly explain the idea. In the regime $k/n\to 1$, the number of blocks in the incomplete composition $(b_0^{(n,k)},b_1^{(n,k)},\dots,b_k^{(n,k)})$ is of the same magnitude as $n$. As we will see, this means that the variability of the number of blocks of sizes $>2$ or of size $1$ is negligible in the sense 
that they only contribute to $\Lah{n}{k}{r}{s}\sim Z^{(n,0)}_{r,s} + Z^{(n,1)}_{r,s} + \dots + Z^{(n,k)}_{r,s}$ 
through the expectations. The principal contribution to the variance of $\Lah{n}{k}{r}{s}$ comes from the blocks of size $2$, and the number of these blocks 
turns out to be sufficiently close to $n-k$ in probability. Independently for each block of size $2$ the corresponding sum $Z^{(n,j)}_{r,s}$ takes values $1$ or $2$ with probability $1/2$ resulting in the limiting variance $1/4$. 
We formalize and summarize the above heuristics in a lemma.

\begin{lemma}\label{lem:aux}
For $m\in\N$, $m\leq n-k+1$ and a random $(r+s)$-composition $(b_0^{(n,k)},b_1^{(n,k)},\dots,b_k^{(n,k)})$, let $\rho_m(n,k)$ be the number of its (usual)
blocks of size $m$, that is,
$$
\rho_m(n,k)=\sum_{j=1}^k \1[b_j^{(n,k)}=m].
$$
Let $n\to\infty$, $k\to\infty$ such that $k/n \to 1$ and $n-k\to\infty$. Then

\noindent (a) ${\rm Var}\,[\rho_1(n,k)]~\simeq~n^{-1}(n-k)^2=o(n-k)$ as $n,k\to\infty$.

\noindent (b) ${\rm Var}\,[\rho_2(n,k)]~\simeq~4n^{-1}(n-k)^2=o(n-k)$ and $\rho_2(n,k)/(n-k)~\overset{\P}{\to} 1$ as $n,k\to\infty$.

\noindent (c) $\sum_{m=3}^{n-k+1} H_m \E[\rho_m(n,k)]=O(n^{-1}(n-k)^2)=o(n-k)$ as $n,k\to\infty$.

\noindent (d) $\E\Big[\sum_{m=3}^{n-k+1} H_m(\rho_m(n,k)-\E[\rho_m(n,k)])\Big]^2=O(n^{-1}(n-k)^2)=o(n-k)$ as $n,k\to\infty$.
\end{lemma}
\begin{proof}
According to~\eqref{eq:r_composition_marginal_distr_b_j}, for $m\in\N$,
$$
\E [\rho_m(n,k)]=k\P[b_1^{(n,k)}=m]=k\frac{\binom{n-m+r+s-1}{k+r+s-2}}{\binom{n+r+s-1}{k+r+s-1}}.
$$
According to~\eqref{eq:r_composition_joint distr_b_j}, for $1\leq i<j\leq k$,
$$
\P[b_i^{(n,k)}=m, b_j^{(n,k)}=m]=\frac{\binom{n+r+s-2m-1}{k+r+s-3}}{\binom{n+r+s-1}{k+r+s-1}}.
$$
This yields
\begin{align*}
{\rm Var}\, [\rho_m(n,k)]&=\E [\rho_m(n,k)]+2\sum_{1\leq i<j\leq k}\P[b_i^{(n,k)}=m, b_j^{(n,k)}=m]-(\E [\rho_m(n,k)])^2\\
&=\E [\rho_m(n,k)]+k(k-1)\P[b_1^{(n,k)}=m, b_2^{(n,k)}=m]-(\E [\rho_m(n,k)])^2\\
&=k\frac{\binom{n-m+r+s-1}{k+r+s-2}}{\binom{n+r+s-1}{k+r+s-1}}+k(k-1)\frac{\binom{n+r+s-2m-1}{k+r+s-3}}{\binom{n+r+s-1}{k+r+s-1}}-\left(k\frac{\binom{n-m+r+s-1}{k+r+s-2}}{\binom{n+r+s-1}{k+r+s-1}}\right)^2.
\end{align*}
We will make a repeated use of these formulas without further reference and in conjunction with the identities
$$
\binom{\alpha}{\beta}=\frac{\alpha}{\beta}\binom{\alpha-1}{\beta-1}\quad\text{and}\quad \binom{\alpha}{\beta}=\frac{\alpha-\beta+1}{\beta}\binom{\alpha}{\beta-1}.
$$

\noindent (a) Plugging $m=1$ yields
\begin{multline*}
{\rm Var}\,[\rho_1(n,k)]=\frac{k(k+r+s-1)}{n+r+s-1}+\frac{k(k-1)(k+r+s-1)(k+r+s-2)}{(n+r+s-1)(n+r+s-2)}-\frac{k^2(k+r+s-1)^2}{(n+r+s-1)^2}\\
=\frac{k(n-k)(n-k+r+s-1)(k+r+s-1)}{(n+r+s-2)(n+r+s-1)^2}~\simeq~\frac{(n-k)^2}{n}=o(n-k),\quad n,k\to\infty.
\end{multline*}

\noindent (b) Plugging $m=2$ yields
\begin{align*}
{\rm Var}\, [\rho_2(n,k)]&= k\frac{(k+r+s-1)(n-k)}{(n+r+s-2)(n+r+s-1)}\\
&+ k(k-1)\frac{(k+r+s-1)(k+r+s-2)(n-k-1)(n-k)}{(n+r+s-4)(n+r+s-3)(n+r+s-2)(n+r+s-1)}\\
&-\left(k\frac{(k+r+s-1)(n-k)}{(n+r+s-2)(n+r+s-1)}\right)^2\\
&\simeq~4n^{-1}(n-k)^2=o(n-k),\quad n,k\to\infty.
\end{align*}
Noting that 
$$\E [\rho_2(n,k)]=
k\frac{(k+r+s-1)(n-k)}{(n+r+s-2)(n+r+s-1)}~\sim~n-k,\quad n,k\to\infty,$$ the claim about the convergence in probability follows with the help of Chebyshev's inequality.

\noindent (c) For $m\geq 3$,
$$
\E[\rho_m(n,k)]=\frac{k(k+r+s-1)}{n+r+s-1}\prod_{i=0}^{m-2}\frac{n-k-i}{n+r+s-2-i}\leq k\Big(\frac{n-k}{n+r+s-2}\Big)^{m-1}.
$$
It remains to note that
$$
k\sum_{m=3}^{n-k+1} H_m\Big(\frac{n-k}{n+r+s-2}\Big)^{m-1}~\sim~H_3\frac{(n-k)^2}{n},\quad n,k\to\infty,
$$
by the Lebesgue 
dominated convergence theorem.

\noindent (d) By the triangle inequality in $L_2$, subadditivity of $x\mapsto x^{1/2}$ on $[0,\infty)$ and the inequality $(x+y)^2\leq 2x^2+2y^2$ for $x,y\geq 0$
\begin{align*}
\E\Big[\sum_{m=3}^{n-k+1} H_s(\rho_m(n,k)-\E[\rho_m(n,k)])\Big]^2&\leq \Big(\sum_{m=3}^{n-k+1}H_m ({\rm Var}\,[\rho_m(n,k)])^{1/2}\Big)^2\\
&\hspace{-4cm}\leq 2 \Big(\sum_{m=3}^{n-k+1} H_m (k(k-1)\P\{b_1^{(n,k)}=m, b_2^{(n,k)}=m\}-(\E [\rho_m(n,k)])^2)^{1/2}\Big)^2\\
&\hspace{-4cm}+2\Big(\sum_{m=3}^{n-k+1} H_m (\E[\rho_m(n,k)])^{1/2}\Big)^2.
\end{align*}
As in part (c) we infer $\Big(\sum_{m=3}^{n-k+1} H_m (\E[\rho_m(n,k)])^{1/2}\Big)^{2}=O(n^{-1}(n-k)^2)=o(n-k)$ as $n,k\to\infty$. Further, for $m\geq 3$
\begin{align*}
&k(k-1)\P[b_1^{(n,k)}=m, b_2^{(n,k)}=m]-(\E [\rho_m(n,k)])^2\\
&=\frac{k(k+r+s-1)}{n+r+s-1}\prod_{j=0}^{m-2}\frac{n-k-j}{n+r+s-2-j}\\
&\times\left(\frac{(k-1)(k+r+s-2)}{n+r+s-2m}\prod_{j=0}^{m-2}\frac{n-k-m+1-j}{n+r+s-m-1-j}-\frac{k(k+r+s-1)}{n+r+s-1}\prod_{j=0}^{m-2}\frac{n-k-j}{n+r+s-2-j}\right)\\
&\leq \frac{k(k+r+s-1)}{n+r+s-1}\Big(\frac{n-k}{n+r+s-2}\Big)^{m-1}\Big(\frac{n-k-m+2}{n+r+s-m}\Big)^{m-1}\\
&\hspace{5cm}\times\Big(\frac{(k-1)(k+r+s-2)}{n+r+s-2m}-\frac{k(k+r+s-1)}{n+r+s-1}\Big)\\
&\leq \frac{k^2(k+r+s-1)^2}{n+r+s-1}\Big(\frac{n-k}{n+r+s-2}\Big)^{2m-2}\left(\frac{1}{n+r+s-2m}-\frac{1}{n+r+s-1}\right)\\
&= \frac{(2m-1)k^2}{n+r+s-2m}\frac{(k+r+s-1)^2}{(n+r+s-1)^2}\Big(\frac{n-k}{n+r+s-2}\Big)^{2m-2}\\
&\leq \frac{(2m-1)k^2}{n+r+s-2m}\Big(\frac{n-k}{n+r+s-2}\Big)^{2m-2}.
\end{align*}
Observe that $n/(n-2m)\to 1$ as $n,k\to\infty$ uniformly in $m\in\N$, $m\leq n-k+1$. This in combination with Lebesgue's dominated convergence theorem enables us to conclude that
\begin{multline*}
k^2\Big(\sum_{m=3}^{n-k+1}H_m \Big(\frac{2m-1}{n+r+s-2m}\Big)^{1/2}\Big(\frac{n-k}{n+r+s-2}\Big)^{m-1}\Big)^2\\
~\sim~5H_3^2\frac{(n-k)^4}{n^3}=o\Big(\frac{(n-k)^2}{n}\Big),\quad n,k\to\infty.
\end{multline*}
\end{proof}
\begin{proof}[Proof of Theorem \ref{thm:clt large k}]
We first show that
\begin{equation}\label{eq:rs0}
\frac{\sum_{m=1}^{n-k+1} \sum_{j=1}^k Z_{r,s}^{(n,j)}\1[b_j^{(n,k)}=m]-\E\Big[\sum_{m=1}^{n-k+1} \sum_{j=1}^k Z_{r,s}^{(n,j)}\1[b_j^{(n,k)}=m]\Big] }{2^{-1}\sqrt{n-k}}~\todistrnk~\Normal{0}{1}.
\end{equation}
By Lemma~\ref{lem:aux}(a) and Chebyshev's inequality,
$$
\frac{\sum_{j=1}^k Z_{r,s}^{(n,j)}\1[b_j^{(n,k)}=1]-\E\Big[\sum_{j=1}^k Z_{r,s}^{(n,j)}\1[b_j^{(n,k)}=1]\Big] }{\sqrt{n-k}}=\frac{\rho_1(n,k)-\E[\rho_1(n,k)]}{\sqrt{n-k}}~\overset{\P}{\to}~0,\quad n,k\to\infty.
$$
The $m=2$ term of the sum gives a principal contribution. Indeed, write
\begin{multline*}
\frac{\sum_{j=1}^k Z_{r,s}^{(n,j)}\1[b_j^{(n,k)}=2]-\E\Big[\sum_{j=1}^k Z_{r,s}^{(n,j)}\1[b_j^{(n,k)}=2]\Big]}{2^{-1}\sqrt{n-k}}\\=\frac{\sum_{j=1}^k (2Z_{r,s}^{(n,j)}-3)\1[b_j^{(n,k)}=2]}{\sqrt{\rho_2(n,k)}}\sqrt{\frac{\rho_2(n,k)}{n-k}}+3\frac{\rho_2(n,k)-\E [\rho_2(n,k)]}{\sqrt{n-k}}.
\end{multline*}
By Lemma~\ref{lem:aux}(b) and Chebyshev's inequality, the second term converges to $0$ in probability. Conditionally on $\mathcal{G}_{n,k}$, $\sum_{j=1}^k (2Z_{r,s}^{(n,j)}-3)\1[b_j^{(n,k)}=2]$ has the same distribution as the sum of $\rho_2(n,k)$ independent random variables taking values $\pm 1$ with probability $1/2$. Hence, by the central limit theorem for random walks the first factor converges in distribution to $\Normal{0}{1}$ conditionally on $\mathcal{G}_{n,k}$, hence also unconditionally. By Lemma~\ref{lem:aux} (b),
$$
\frac{\rho_2(n,k)}{n-k}~\overset{\P}{\to}~1,\quad n,k\to\infty.
$$
This secures the distributional convergence of the first term and the sum of two terms to $\Normal{0}{1}$.

Now we prove 
that the contributions of the counts $\rho_m(n,k)$, $3\leq m\leq n-k+1$ are negligible. Write
\begin{multline*}
\sum_{m=3}^{n-k+1} \sum_{j=1}^k Z_{r,s}^{(n,j)}\1[b_j^{(n,k)}=m]-\E\Big[\sum_{m=3}^{n-k+1} \sum_{j=1}^k Z_{r,s}^{(n,j)}\1[b_j^{(n,k)}=m]\Big]\\=\Big(\sum_{m=3}^{n-k+1} \sum_{j=1}^k Z_{r,s}^{(n,j)}\1[b_j^{(n,k)}=m]-\sum_{m=3}^{n-k+1}H_m \rho_m(n,k)\Big)+\Big(\sum_{m=3}^{n-k+1} H_m (\rho_m(n,k)-\E[\rho_m(n,k)])\Big)\\=:I(n,k)+J(n,k).
\end{multline*}
We first calculate
$$
\E[(I(n,k))^2|\mathcal{G}_{n,k}]=\sum_{m=3}^{n-k+1}\Big( \sum_{\ell=1}^m \ell^{-1}(1-\ell^{-1})\Big)\rho_m(n,k)\leq \sum_{m=3}^{n-k+1} H_m \rho_m(n,k).
$$
Using now Lemma~\ref{lem:aux}(c) we obtain $\E [I(n,k)^2]=o(n-k)$ as $n,k\to\infty$. 
Finally, by Lemma \ref{lem:aux}(d) $\E [J(n,k)^2]=o(n-k)$ as $n,k\to\infty$. The proof of~\eqref{eq:rs0} is complete.

In view of~\eqref{eq:rs0} and according to Proposition~\ref{prop:r_lah_representation}, it remains to show that $\E[Z_{r,s}^{(n,0)}]=O(1)$. But this 
follows immediately from~\eqref{eq:Z_n_0_mean_estimate}.
\end{proof}

Theorems~\ref{thm:clt_lah_r_k_const},~\ref{thm:clt_lah_r_interm},~\ref{thm:clt_central_regime} and~\ref{thm:clt large k} prove Conjecture 2.4 in~\cite{van_der_hofstad_etal_shortest_path_trees} in a more general form.  Hsien-Kuei Hwang has kindly informed us that the asymptotic normality of the distribution appearing in~\cite{van_der_hofstad_etal_shortest_path_trees} was proved by him   (unpublished notes) using singularity analysis and the saddle point method.

\subsection*{Acknowledgement}
We are grateful to Hsien-Kuei Hwang for pointing out the article~\cite{van_der_hofstad_etal_shortest_path_trees}.
ZK has been supported by the German Research Foundation under Germany's Excellence Strategy EXC 2044 - 390685587, Mathematics M\"unster: Dynamics - Geometry - Structure and by the DFG priority program SPP 2265 Random Geometric Systems. AM was supported by the Alexander von Humboldt Foundation.

\bibliographystyle{plain}
\bibliography{refs}

\begin{thebibliography}{10}

\bibitem{aguiar_mahajan_book_hyperplane_arrangements}
M.~Aguiar and S.~Mahajan.
\newblock {\em Topics in hyperplane arrangements}.
\newblock American Mathematical Society, Providence, RI, 2017.

\bibitem{Arratia+Barbour+Tavare:2003}
R.~Arratia, A.~D. Barbour, and S.~Tavar\'{e}.
\newblock {\em Logarithmic combinatorial structures: a probabilistic approach}.
\newblock European Mathematical Society (EMS), Z\"{u}rich, 2003.

\bibitem{belbachir:2013}
H.~Belbachir and A.~Belkhir.
\newblock Cross recurrence relations for r-{L}ah numbers.
\newblock {\em Ars Combinatoria}, 110:199--203, 2013.

\bibitem{belkhir_multivar_lah}
A.~Belkhir.
\newblock The multivariate {L}ah and {S}tirling numbers.
\newblock {\em J. Integer Seq.}, 23(4):Art. 20.4.5, 13, 2020.

\bibitem{broder:1984}
A.~Broder.
\newblock The $r$-{S}tirling numbers.
\newblock {\em Discrete Mathematics}, 49(3):241--259, 1984.

\bibitem{Carlitz:1980-1}
L.~Carlitz.
\newblock Weighted {S}tirling numbers of the first and second kind. {I}.
\newblock {\em Fibonacci Quart.}, 18(2):147--162, 1980.

\bibitem{Carlitz:1980-2}
L.~Carlitz.
\newblock Weighted {S}tirling numbers of the first and second kind. {II}.
\newblock {\em Fibonacci Quart.}, 18(3):242--257, 1980.

\bibitem{Carlitz+Scoville:1974}
L.~Carlitz and R.~Scoville.
\newblock Generalized {E}ulerian numbers: combinatorial applications.
\newblock {\em J. Reine Angew. Math.}, 265:110--137, 1974.

\bibitem{charalambides_eulerian_generalized}
Ch.~A. Charalambides.
\newblock On a generalized {E}ulerian distribution.
\newblock {\em Ann. Inst. Statist. Math.}, 43(1):197--206, 1991.

\bibitem{charalambides_book_enum_combin}
Ch.~A. Charalambides.
\newblock {\em Enumerative combinatorics}.
\newblock Chapman \& Hall/CRC, 2002.

\bibitem{charalambides_book_combinatorial_methods}
Ch.~A. Charalambides.
\newblock {\em Combinatorial methods in discrete distributions}.
\newblock Wiley Ser. Probab. Stat. John Wiley \& Sons, 2005.

\bibitem{charalambides_koutras_diff_gen_factorials}
Ch.~A. Charalambides and M.~Koutras.
\newblock On the differences of the generalized factorials at an arbitrary
  point and their combinatorial applications.
\newblock {\em Discrete Math.}, 47:183--201, 1983.

\bibitem{charalambides_singh_stirling_review}
Ch.~A. Charalambides and J.~Singh.
\newblock A review of the {S}tirling numbers, their generalizations and
  statistical applications.
\newblock {\em Comm. Statist. Theory Methods}, 17(8):2533--2595, 1988.

\bibitem{cheon:2012}
G.-S. Cheon and J.-H. Jung.
\newblock $r$-{W}hitney numbers of {D}owling lattices.
\newblock {\em Discrete Mathematics}, 312:2337--2348, 2012.

\bibitem{crane_ewens_formula}
H.~Crane.
\newblock The ubiquitous {E}wens sampling formula.
\newblock {\em Statist. Sci.}, 31(1):1--19, 2016.

\bibitem{devroye_records}
L.~Devroye.
\newblock Applications of the theory of records in the study of random trees.
\newblock {\em Acta Inform.}, 26(1-2):123--130, 1988.

\bibitem{Dondajewski+Szymanski:1982}
M.~Dondajewski and J.~Szyma\'{n}ski.
\newblock On the distribution of vertex-degrees in a strata of a random
  recursive tree.
\newblock {\em Bull. Acad. Polon. Sci. S\'{e}r. Sci. Math.}, 30(5--6):205--209,
  1982.

\bibitem{Drmota_book}
M.~Drmota.
\newblock {\em Random trees: An interplay between combinatorics and
  probability}.
\newblock Springer-Verlag, Vienna, 2009.

\bibitem{Gnedin+Pitman:2005}
A.~Gnedin and J.~Pitman.
\newblock Exchangeable {G}ibbs partitions and {S}tirling triangles.
\newblock {\em Zap. Nauchn. Sem. S.-Peterburg. Otdel. Mat. Inst. Steklov.
  (POMI)}, 325:83--102, 244--245, 2005.

\bibitem{godlandschlaefli:2022}
T.~Godland and Z.~Kabluchko.
\newblock Angle sums of {S}chl{\"a}fli orthoschemes.
\newblock {\em Discrete and Computational Geometry}, 68:125--164, 2022.

\bibitem{godland:2022}
T.~Godland and Z.~Kabluchko.
\newblock Positive hulls of random walks and bridges.
\newblock {\em Stochastic Processes and their Applications}, 147:327--362,
  2022.

\bibitem{graham_etal_book}
R.~L. Graham, D.~E. Knuth, and O.~Patashnik.
\newblock {\em Concrete mathematics}.
\newblock Addison-Wesley Publishing Company, second edition, 1994.

\bibitem{Gut:2005}
A.~Gut.
\newblock {\em Probability: a graduate course}.
\newblock Springer, New York, 2005.

\bibitem{van_der_hofstad_etal_shortest_path_trees}
R.~van~der Hofstad, G.~Hooghiemstra, and P.~Van~Mieghem.
\newblock Size and weight of shortest path trees with exponential link weights.
\newblock {\em Combin. Probab. Comput.}, 15(6):903--926, 2006.

\bibitem{hoppe_urns_ewens}
F.~Hoppe.
\newblock P\'{o}lya-like urns and the {E}wens' sampling formula.
\newblock {\em J. Math. Biol.}, 20(1):91--94, 1984.

\bibitem{Hoppe:1986}
F.~Hoppe.
\newblock Size-biased filtering of {P}oisson-{D}irichlet samples with an
  application to partition structures in genetics.
\newblock {\em J. Appl. Probab.}, 23(4):1008--1012, 1986.

\bibitem{hoppe_sampling}
F.~Hoppe.
\newblock The sampling theory of neutral alleles and an urn model in population
  genetics.
\newblock {\em J. Math. Biol.}, 25(2):123--159, 1987.

\bibitem{huillet_occupancy_problems_generalized_stirling}
T.~Huillet.
\newblock Occupancy problems related to the generalized {S}tirling numbers.
\newblock {\em J. Stat. Phys.}, 191(1):Paper No. 5, 2024.

\bibitem{huillet_moehle_bernoulli}
T.~Huillet and M.~M{\"o}hle.
\newblock On {B}ernoulli trials with unequal harmonic success probabilities.
\newblock {\em Metrika}, 2023.
\newblock to appear.

\bibitem{Janardan:1988}
K.~G. Janardan.
\newblock Relationship between {M}orisita's model for estimating the
  environmental density and the generalized {E}ulerian numbers.
\newblock {\em Ann. Inst. Statist. Math.}, 40(3):439--450, 1988.

\bibitem{janardan:1993}
K.~G. Janardan.
\newblock Some properties of the generalized {E}ulerian distribution.
\newblock {\em J. Statist. Plann. Inference}, 34(2):159--169, 1993.

\bibitem{janson_euler_frobenius_rounding}
S.~Janson.
\newblock Euler-{F}robenius numbers and rounding.
\newblock {\em Online J. Anal. Comb.}, 8:34, 2013.

\bibitem{JoagDev+Proschan:1983}
K.~Joag-Dev and F.~Proschan.
\newblock Negative association of random variables, with applications.
\newblock {\em Ann. Statist.}, 11(1):286--295, 1983.

\bibitem{Johnson+Kotz:UrnModels}
N.~L. Johnson and S.~Kotz.
\newblock {\em Urn models and their application: an approach to modern discrete
  probability theory}.
\newblock John Wiley \& Sons, Inc., New York, 1977.

\bibitem{Johnson+Kotz+Balakrishnan}
N.~L. Johnson, S.~Kotz, and N.~Balakrishnan.
\newblock {\em Discrete multivariate distributions}.
\newblock John Wiley \& Sons, Inc., New York, 1997.

\bibitem{kabluchko_marynych_lah}
Z.~Kabluchko and A.~Marynych.
\newblock Lah distribution: Stirling numbers, records on compositions, and
  convex hulls of high-dimensional random walks.
\newblock {\em Probability Theory and Related Fields}, 184:969--1028, 2022.

\bibitem{Kabluchko+Marynych+Pitters:2024}
Z.~Kabluchko, A.~Marynych, and H.~Pitters.
\newblock Mod-{$\varphi$} convergence of {S}tirling distributions and limit
  theorems for zeros of their generating functions.
\newblock {\em J. Math. Anal. Appl.}, 529(1):Paper No. 127571, 27, 2024.

\bibitem{kabluchko_steigenberger_r_lah}
Z.~Kabluchko and D.~A. Steigenberger.
\newblock {$r$}-{L}ah distribution: properties, limit theorems and an
  application to compressed sensing.
\newblock {\em Adv. in Appl. Math.}, 150:Paper No. 102575, 23, 2023.

\bibitem{KVZ17}
Z.~Kabluchko, V.~Vysotsky, and D.~Zaporozhets.
\newblock Convex hulls of random walks: expected number of faces and face
  probabilities.
\newblock {\em Adv. Math.}, 320:595--629, 2017.

\bibitem{kabluchko:2016}
Z.~Kabluchko, V.~Vysotsky, and D.~Zaporozhets.
\newblock Convex hulls of random walks, hyperplane arrangements, and {W}eyl
  chambers.
\newblock {\em Geometric and Functional Analysis}, 27(4):880--918, 2017.

\bibitem{Kingman}
J.~F.~C. Kingman.
\newblock {\em Poisson processes}.
\newblock The Clarendon Press, Oxford University Press, New York, 1993.

\bibitem{koutras}
M.~Koutras.
\newblock Noncentral {S}tirling numbers and some applications.
\newblock {\em Discrete Math.}, 42(1):73--89, 1982.

\bibitem{Leckey+Neininger:2013}
K.~Leckey and R.~Neininger.
\newblock Asymptotic analysis of {H}oppe trees.
\newblock {\em J. Appl. Probab.}, 50(1):228--238, 2013.

\bibitem{mahmoud_smythe}
H.~M. Mahmoud and R.~T. Smythe.
\newblock On the distribution of leaves in rooted subtrees of recursive trees.
\newblock {\em Ann. Appl. Probab.}, 1(3):406--418, 1991.

\bibitem{maier_triangular_recurrences}
R.~S. Maier.
\newblock Triangular recurrences, generalized {E}ulerian numbers, and related
  number triangles.
\newblock {\em Adv. in Appl. Math.}, 146:Paper No. 102485, 62, 2023.

\bibitem{mansour_book}
T.~Mansour.
\newblock {\em Combinatorics of set partitions}.
\newblock CRC Press, Boca Raton, FL, 2013.

\bibitem{mansour_schork_book_commutation_relations}
T.~Mansour and M.~Schork.
\newblock {\em Commutation relations, normal ordering, and {S}tirling numbers}.
\newblock CRC Press, Boca Raton, FL, 2016.

\bibitem{mezoe_r_bell}
I.~Mez\H{o}.
\newblock The {$r$}-{B}ell numbers.
\newblock {\em J. Integer Seq.}, 14(1):Article 11.1.1, 14, 2011.

\bibitem{mezo_recent_developments_stirling}
I.~Mez\H{o}.
\newblock Recent developments in the theory of {S}tirling numbers.
\newblock {\em RIMS K\^{o}ky\^{u}roku}, 2013:68--80, 2013.

\bibitem{mezo_book}
I.~Mez{\H{o}}.
\newblock {\em Combinatorics and number theory of counting sequences}.
\newblock CRC Press, 2020.

\bibitem{Morisita:1971}
M.~Morisita.
\newblock Measuring of habitat value by environmental density method.
\newblock In G.~P. et~al. Patil, editor, {\em Statistical Ecology}, pages
  379--401. The Pennsylvania State University Press, 1971.

\bibitem{najock_heyde}
D.~Najock and C.~C. Heyde.
\newblock On the number of terminal vertices in certain random trees with an
  application to stemma construction in philology.
\newblock {\em J. Appl. Probab.}, 19(3):675--680, 1982.

\bibitem{nishimura_sibuya}
K.~Nishimura and M.~Sibuya.
\newblock Extended {S}tirling family of discrete probability distributions.
\newblock {\em Comm. Statist. Theory Methods}, 26(7):1727--1744, 1997.

\bibitem{nyul:2015}
G.~Nyul and G.~R\'acz.
\newblock The $r$-{L}ah numbers.
\newblock {\em Discrete Mathematics}, 338(10):1660--1666, 2015.

\bibitem{petersen_book_eulerian_numbers}
T.~Petersen.
\newblock {\em Eulerian numbers}.
\newblock Birkh\"{a}user/Springer, New York, 2015.

\bibitem{pinsky_a_view}
R.~Pinsky.
\newblock A view from the bridge spanning combinatorics and probability.
\newblock {\em Enumer. Comb. Appl.}, 1(3):Paper No. S2S3, 31, 2021.

\bibitem{Pitman_book}
J.~Pitman.
\newblock {\em Combinatorial stochastic processes}.
\newblock Springer-Verlag, Berlin, 2006.

\bibitem{shanmugam_r_stirling}
R.~Shanmugam.
\newblock On central versus factorial moments.
\newblock {\em South African Statist. J.}, 18(2):97--110, 1984.

\bibitem{shattuck:2016}
M.~Shattuck.
\newblock Generalized {$r$}-{L}ah numbers.
\newblock {\em Proc. Indian Acad. Sci. Math. Sci.}, 126(4):461--478, 2016.

\bibitem{sibuya_stirling_family}
M.~Sibuya.
\newblock {\em Stirling family of distributions}.
\newblock John Wiley \& Sons, 2014.

\bibitem{sibuya_nishimura_breaking}
M.~Sibuya and K.~Nishimura.
\newblock Prediction of record-breakings.
\newblock {\em Statist. Sinica}, 7(4):893--906, 1997.

\bibitem{smythe_mahmoud_survey_RRT}
R.~T. Smythe and H.~M. Mahmoud.
\newblock A survey of recursive trees.
\newblock {\em Teor. \u{I}mov\={\i}r. Mat. Stat.}, 51:1--29, 1994.

\bibitem{stam_random_partition}
A.~J. Stam.
\newblock Generation of a random partition of a finite set by an urn model.
\newblock {\em J. Combin. Theory Ser. A}, 35(2):231--240, 1983.

\bibitem{Stanley_book}
R.~Stanley.
\newblock {\em Enumerative combinatorics. {V}ol. 1}.
\newblock Cambridge University Press, 1997.

\bibitem{szymanski_max_degree}
J.~Szyma\'{n}ski.
\newblock On the maximum degree and the height of a random recursive tree.
\newblock In {\em Random graphs '87 ({P}ozna\'{n}, 1987)}, pages 313--324.
  Wiley, Chichester, 1990.

\bibitem{van_der_hofstad_first_passage}
R.~van~der Hofstad, G.~Hooghiemstra, and P.~Van~Mieghem.
\newblock First-passage percolation on the random graph.
\newblock {\em Probab. Engrg. Inform. Sci.}, 15(2):225--237, 2001.

\bibitem{van_mieghem_book}
P.~Van~Mieghem.
\newblock {\em Performance analysis of communications networks and systems}.
\newblock Cambridge University Press, 2006.

\bibitem{van_mieghem_scaling_hopcount_report}
P.~Van~Mieghem, G.~Hooghiemstra, and R.~van~der Hofstad.
\newblock A scaling law for the hopcount in internet.
\newblock {\em Delft University of Technology, report}, 2000125, 2000.

\bibitem{van2001stochastic}
P.~Van~Mieghem, G.~Hooghiemstra, and R.~van~der Hofstad.
\newblock Stochastic model for the number of traversed routers in internet.
\newblock {\em Proceedings of Passive and Active Measurement (PAM), RIPE NCC
  (Amsterdam, The Netherlands, 2001)}, 2001.

\end{thebibliography}

\end{document}